\documentclass[11pt]{amsart}
\usepackage[utf8]{inputenc}
\usepackage{amsmath,amsfonts,amssymb}

\usepackage{latexsym}
\usepackage[font={footnotesize}]{caption}
\usepackage{graphics,epsf,psfrag,epsfig}
\usepackage{eufrak}
\usepackage{xcolor}
\usepackage{hyperref}

\theoremstyle{plain}
\newtheorem{lemma}{Lemma}[section]
\newtheorem{theorem}[lemma]{Theorem}
\newtheorem{corollary}[lemma]{Corollary}
\newtheorem{proposition}[lemma]{Proposition}

\theoremstyle{definition}
\newtheorem{remark}[lemma]{Remark}

\numberwithin{equation}{section}

\DeclareSymbolFont{yhlargesymbols}{OMX}{yhex}{m}{n}
\DeclareMathAccent{\wideparen}{\mathord}{yhlargesymbols}{"F3}

\def\var{\text{var}}

\newcommand{\ZZ}{\mathbb{Z}}

\newcommand{\RR}{\mathbb{R}}

\newcommand{\QQ}{\mathbb{Q}}
\newcommand{\EE}{\mathbb{E}}

\newcommand{\mA}{\mathcal{A}}
\newcommand{\mB}{\mathcal{B}}
\newcommand{\mkB}{\mathfrak{B}}
\newcommand{\mC}{\mathcal{C}}
\newcommand{\mD}{\mathcal{D}}
\newcommand{\mE}{\mathcal{E}}
\newcommand{\mF}{\mathcal{F}}
\newcommand{\mG}{\mathcal{G}}

\newcommand{\mH}{\mathcal{H}}

\newcommand{\mL}{\mathcal{L}}

\newcommand{\mN}{\mathcal{N}}

\newcommand{\mQ}{\mathcal{Q}}

\newcommand{\mU}{\mathcal{U}}

\newcommand{\tu}{\tilde{u}}
\newcommand{\tv}{\tilde{v}}

\newcommand{\hn}{\hat{n}}

\newcommand{\hsi}{\hat{\sigma}}

\newcommand{\hm}{\hat{m}}

\newcommand{\hy}{\hat{y}}

\newcommand{\ep}{\epsilon}

\newcommand{\pa}{\partial}
\newcommand{\bs}{\backslash}

\newcommand{\yt}{y_\theta}

\newcommand{\tred}{\textcolor{red}}
\newcommand{\tcyn}{\textcolor{blue}}
\begin{document}

\title[FPP above critical dimension]{Properties of first passage percolation above the (hypothetical) critical dimension}

\author{Kenneth S. Alexander}
\address{Department of Mathematics \\
University of Southern California\\
Los Angeles, CA  90089-2532 USA}
\email{alexandr@usc.edu}


\begin{abstract} 
It is not known (and even physicists disagree) whether first passage percolation (FPP) on $\ZZ^d$ has an upper critical dimension $d_c$, such that the fluctuation exponent $\chi=0$ in dimensions $d>d_c$. In part to facilitate study of this question, we may nonetheless try to understand properties of FPP in such dimensions should they exist, in particular how they should differ from $d<d_c$.  We show that at least one of three fundamental properties of FPP known or believed to hold when $\chi>0$ must be false if $\chi=0$.  A particular one of the three is most plausible to fail, and we explore the consequences if it is indeed false.  These consequences support the idea that when $\chi=0$, passage times are ``local'' in the sense that the passage time from $x$ to $y$ is primarily determined by the configuration near $x$ and $y$.  Such locality is manifested by certain ``disc--to--disc'' passage times, between discs in parallel hyperplanes, being typically much faster than the fastest mean passage time between points in the two discs.
\end{abstract}

\maketitle

\section{Introduction}

The study of first passage percolation (FPP) on $\ZZ^d$ suffers from the fact that many of the most fundamental properties presumed to hold, which should be the foundation for many other proofs, have resisted proof.  Open problems of this type include:
\begin{itemize}
\item[(1)] Do the fluctuation and wandering exponents $\chi,\xi$ exist, and do they take their conjectured (KPZ) values 1/3, 2/3 for $d=2$?
\item[(2)] Is there a finite upper critical dimension $d_c$, above which the fluctuation exponent $\chi=0$?
\item[(3)] When edge passage times have an exponential moment, is there an exponential bound on the scale of the standard deviation, for deviations of these passage times?
\item[(4)] Does the limit shape have nondegenerate curvature in all directions?
\end{itemize}
Various authors have dealt with this situation in different ways.  One approach is to prove far--from--optimal (yet still difficult) results, for example that the variance of passage times over distance $n$ is at least of order $\log n$ \cite{DHHX20}, or that for fixed $k$, for a special class of passage time distributions, the limit shape cannot be a polygon with fewer than $k$ sides \cite{DEP22}, both for $d=2$.  A second approach, taken in \cite{Al20}, \cite{Al20a}, \cite{Ne95}, and for non--integrable LPP in \cite{GH20}, is to prove conditional results, assuming certain fundamental unproven properties like (3) and/or (4) above, and showing that more delicate properties follow from them.  A third option is to explore interrelations among the fundamental properties, which does not require knowing or assuming which ones are true; this third option, and in part the second option, are our primary approaches in this paper.  In particular we show that at least one of three fundamental properties believed to hold in dimensions below $d_c$ is false above $d_c$.  We then move into conditional--result mode, to explore the consequences if the property which seems the clear candidate to fail above $d_c$ indeed does so.  This gives some insight into the question, ``if $d_c$ is finite, how should we expect FPP in dimensions above and below $d_c$ to differ?''  

The question of a finite $d_c$ or not is particularly challenging, as even physicists do not agree.  Informally, $\chi$ is the exponent for which the variance of the passage time over distance $n$ grows like $n^{2\chi}$. Heuristics and simulations suggest that $\chi$ should decrease with dimension; simulations in \cite{ROM15} for a model believed to be in the same (KPZ) universality class as FPP show a decrease from $\chi=.33$ to $\chi=.054$ as $d$ increases from 2 to 7.  In general, some calculations (e.g.~renormalization group) seem to predict the existence of a finite upper critical dimension, possibly as low as 3.5, above which $\chi=0$ (\cite{Fo08},\cite{LW05}); in contrast, various simulations suggest that $\chi$ may be positive for all $d$ (\cite{AOF14},\cite{MPPR02}), with those in \cite{KK14} showing $\chi>0$ all the way to $d=12$, decaying approximately as $1/(d+1)$. So it is possible that the primary question we are exploring is vacuous ($d_c=\infty$), but in that case an understanding of the hypothetical nature of FPP above $d_c$ could lead to insights as to why such dimensions do not exist.

It is known that on certain large graphs, including for example Galton--Watson trees \cite{DH91} and the product of a rooted $d$--regular tree and $\ZZ$ (\cite{BM14}, \cite{BZ12}), the centered passage time distributions over various distances $n$ form a tight sequence, which of course is stronger than just $\chi=0$. (In the Galton--Watson case this is for point--to--sphere times, not point--to--point.)  On other large graphs these distributions yield tight sequences after dividing by $\log n$; examples include the product of an unrooted $d$--regular tree with $\ZZ$ \cite{BM14}. On the Galton--Watson tree the tightness may be attributed to the fact that the point--to--sphere time from the root is mostly determined by edge passage times within a bounded distance of the root.
All this makes it somewhat plausible that indeed as the dimension grows, $\ZZ^d$ becomes a large enough graph that it, too, has $\chi=0$, or at least that $\chi\to0$ as the dimension grows.  It is not clear whether passage times should be tight if $\chi=0$, or whether point--to--point passage times should be determined primarily near the path endpoints as on the Galton--Watson tree; this ``local,'' or ``near the endpoints,'' aspect is what we will attempt to shed preliminary light on, in the present work.

\begin{remark}\label{localness}
There is the question of how this localness can be manifested quantifiably for FPP in $\ZZ^d$, which we approach as follows. Consider two parallel hyperplanes, tangent to a multiple of the limit shape at opposite points $x$ and $-x$. Consider discs $D_x$ and $D_{-x}$ in these hyperplanes, centered at $\pm x$, having radius that is small relative to the usual transverse wandering of the geodesic $\Gamma_{-x,x}$. Let $\hsi(2x)$ be the standard deviation of the passage time $T(-x,x)$ from $-x$ to $x$. The small radius of the discs means that for $p,p'\in D_{-x}$ and $q,q'\in D_x$ the geodesics $\Gamma_{pq}$ and $\Gamma_{p'q'}$ may coincide for most of their length, along all but short sections near their endpoints. If passage times are nonlocal (as expected at least below $d_c$) then this coinciding means the passage times $T(p,q)$ and $T(p',q')$ are typically nearly equal, within a small multiple of $\hsi(2x)$. Therefore the disc--to--disc passage time $\min_{p,q} T(p,q)$ is typically close (again, within a small multiple of $\hsi(2x)$) to any given point--to--point passage time $T(p,q)$.
By contrast, if passage times are local, then $T(p,q)$ and $T(p',q')$ are primarily determined by the non--coinciding portions of the geodesics, so these two passage times are nearly independent, if $|p-p'|$ and $|q-q'|$ are not too small.  There are then many nearly--independent point--to--point passage times $T(p,q)$ between the discs, so at least \emph{some} point--to--point passage times between $D_{-x}$ and $D_x$, and thus also the disc--to--disc passage time, will likely be much faster than the typical (i.e.~mean) point--to--point passage time, meaning faster by a large multiple of $\hsi(2x)$.  Such a large gap between mean point--to--point and mean disc--to--disc passage times thus manifests localness, and is what we establish, at least conditionally in the sense described above.
\end{remark}

\subsection{Some definitions}
Let $\EE^d$ \label{EEd} denote the set of all edges (i.e.~ nearest-neighbor pairs) of $\ZZ^d$.  The \emph{passage times} of edges are a collection of nonnegative iid variables $\tau = \{\tau_e:e\in\EE^d\}$. The \emph{passage time} of a finite lattice path $\gamma$ (always taken self--avoiding) is 
\[
  \tred{T(\gamma)} := \sum_{e\in\gamma} \tau_e,
\]
and the passage time from $x$ to $y$ is
\begin{equation}\label{Tinf}
  \tred{T(x,y)} := \inf\{T(\gamma): \gamma \text{ is a path from $x$ to $y$ in } \ZZ^d\}.
\end{equation}
To avoid notational clutter, for general $x,y\in\RR^d$ we define $T(x,y)$ to be $T(\hat x,\hat y)$, where $\hat x,\hat y$ are the closest lattice points to $x$ and $y$ respectively, with ties broken by some arbitrary translation--invariant rule.
A path $\gamma$ achieving the infimum in \eqref{Tinf} is called a \emph{geodesic} from $x$ to $y$.  Assuming $\tau_e$ is a continuous r.v., a unique geodesic exists a.s.~for each $x,y$ \cite{WR78}, and we denote it \tred{$\Gamma_{xy}$}. From subadditivity, the limit
\[
  \tred{g(x)} = \lim_n \frac{ET(0,x)}{n}
\]
exists for all $x\in\ZZ^d$; by considering only $n$ for which $nx\in\ZZ^d$ we may extend this definition to $x\in\QQ^d$, and then to a norm on $\RR^d$ by continuity; see \cite{ADH17}.  Define also
\[
  \tred{h(x)} = ET(0,x),\ \ x\in\ZZ^d.
\]

Throughout the paper, $c_0,c_1,\dots$ and $\ep_0,\ep_1,\dots$ are constants which depend only on the dimension and the distribution of the passage times $\tau_e$, unless otherwise specified.

A function $f\geq 0$ on $\ZZ^d$ is \emph{magnitude--based} if there exists $\tred{f_{mag}}:[1,\infty)\to[0,\infty)$ and $c_1>0$ with
\begin{equation}\label{fmag}
  c_1^{-1}f_{mag}(r) \leq f(x) \leq c_1f_{mag}(r)\ \ \text{whenever } \frac r2 \leq |x| \leq 2r.
\end{equation}
Then fixing $r\geq 1$ and $1\leq\alpha\leq 2$, and taking $x$ with $|x|$ close enough to $r$, we get that $c_1^{-1}f_{mag}(r) \leq f(x) \leq c_1f_{mag}(\alpha r)$ and $c_1^{-1}f_{mag}(\alpha r) \leq f(x) \leq c_1f_{mag}(r)$,
so $f_{mag}$ satisfies
\begin{equation}\label{reggr}
  c_1^{-2} f_{mag}(r) \leq f_{mag}(\alpha r) \leq c_1^2 f_{mag}(r) \ \ \text{for all } r\geq 1,\,1\leq \alpha \leq 2.
\end{equation}
Therefore there exist $c_{2},\kappa\geq 1$ such that
\begin{equation}\label{reggr2}
  c_{2}^{-1}\alpha^{-\kappa} f_{mag}(r) \leq f_{mag}(\alpha r) \leq c_{2}\alpha^\kappa f_{mag}(r) \ \ \text{for all } r,\alpha\geq 1.
\end{equation}
For general functions $f:[0,\infty)\to(0,\infty)$, we say $f$ has \emph{growth exponent} $\chi\geq 0$ if
\begin{equation}\label{chidef}
  \lim_{r\to\infty} \frac{\log f(r)}{\log r} = \chi.
\end{equation}
We say $f$ has \emph{regular growth exponent} $\chi\geq 0$ if
\begin{equation}\label{regfluct}
  \lim_{\alpha,r \to\infty} \left| \frac{\log f(\alpha r) - \log f(r)}{\log(\alpha r) - \log r} - \chi \right| = 0
\end{equation}
and for all $\beta>0$,
\begin{equation}\label{regfluct2}
  \liminf_{r\to\infty} \inf_{1\leq \alpha\leq\beta} [\log f(\alpha r) - \log f(r)] > -\infty,\quad
    \limsup_{r\to\infty} \sup_{1\leq \alpha\leq\beta} [\log f(\alpha r) - \log f(r)] < \infty.
\end{equation}
We can restate \eqref{regfluct} and \eqref{regfluct2} as: for every $\ep>0$ there exist $\beta,r_0,c_{3}$ such that
\begin{equation}\label{regfluct3}
  r\geq r_0 \implies \begin{cases} \alpha^{\chi-\ep} \leq \frac{ f(\alpha r)}{ f(r)} \leq \alpha^{\chi+\ep} &\text{if } \alpha\geq\beta,\\
    c_{3}^{-1} \leq \frac{ f(\alpha r)}{ f(r)} \leq c_{3} &\text{if } 1\leq \alpha<\beta. \end{cases}
\end{equation}

The following is an enhancement of a result in \cite{Al20a}; it is proved in Section \ref{lempf}.

\begin{lemma}\label{upperapp}
\tcyn{Suppose $f:[1,\infty)\to(0,\infty)$ has growth exponent $\chi= 0$ and for some $c,\kappa,\delta>0$ satisfies
\begin{equation}\label{growctrl}
  f(r)\geq\delta,\quad c^{-1}\alpha^{-\kappa} f(r) \leq f(\alpha r) \leq c\alpha^\kappa f(r) \ \ \text{for all } \alpha,r\geq 1.
\end{equation}
Then there exists $\tred{f_{up}}:[1,\infty)\to(0,\infty)$ with regular growth exponent $\chi=0$ satisfying
\begin{align}\label{upperf}
  f_{up} \geq f, \quad &\liminf_{r\to\infty} \frac{f_{up}(r)}{f(r)} < \infty, \quad f_{up}\ \text{eventually nondecreasing}, 
    \quad \frac{\log f_{up}(r)}{\log r}\ \text{nonincreasing.}
\end{align}}
\end{lemma}

We call an FPP model \emph{standard} if the edge passage times are continuous r.v.'s with a finite exponential moment, and $P(\tau_e=0)<p_c$, where 
$p_c$ denotes the critical probability of 2--dimensional Bernoulli bond percolation.

Let 
\[
  \tred{\hat\sigma(x)} = \var(T(0,x))^{1/2}, \quad \tred{\hat\Delta(x)} = (|x|\hat\sigma(x))^{1/2}, \quad \tred{D(x)} = ET(0,x) - g(x).
\]
Note $D\geq 0$ by subadditivity of $ET(0,x)$.
When $\hat\sigma$ is magnitude--based, we call a corresponding $\sigma_{mag}$ as in \eqref{fmag} an \emph{approximate standard deviation}, respectively; we call a corresponding function from Lemma \ref{upperapp} (that is, $f_{up}$ with $f=\sigma_{mag}$) an \emph{upper--regular standard deviation} and denote it as $\sigma_{up}$. From \cite{DHS14} (improving on \cite{Al97}), for a standard FPP model, if $D$ is magnitude--based then every corresponding $D_{mag}$ satisfies
\begin{equation}\label{ls1}
  D_{mag}(r) \leq c_{4}(r\log r)^{1/2}\ \ \text{so}\ \ 
  \limsup_{r\to\infty} \frac{\log D_{mag}(r)}{\log r} \leq \frac12.
\end{equation}
Note that $\hat\sigma$ is bounded away from 0, since passage times of edges emanating from 0 have an order-1 influence on passage times $T(0,x)$, so any approximate standard deviation will be bounded away from 0 as well.  Let \tred{$\Pi_{xy}^\infty$} denote the infinite line through $x$ and $y$ and let
\tred{$R(x)$} be the maximum transverse wandering from $\Pi_{0x}^\infty$ by the geodesic $\Gamma_{0x}$, that is,
\[
  R(x) = \sup_{u\in\Gamma_{0x}} d(u,\Pi_{0x}^\infty),
\]
with $d(\cdot,\cdot)$ denoting Euclidean distance. 
When $\hat\sigma$ is magnitude--based, we define $\tred{\Delta_{mag}(r)}=(r\sigma_{mag}(r))^{1/2}$ and $\tred{\Delta_{up}(r)}= (r\sigma_{up}(r))^{1/2}$; then $\Delta_{up}$ has regular growth exponent $\tred{\xi}=(1+\chi)/2$.  Let \tred{$\mkB_g(x,r)$} denote the $g$--ball at $x$ of radius $r$; we abbreviate $\mkB_g(0,1)$ to just \tred{$\mkB_g$}.
$\mkB_g$ is known as the \emph{limit shape}, as, under mild hypotheses \cite{CD81}, it is the a.s.~limit of the rescaled ``wet region'' $t^{-1}(\cup\{x+[-\frac12,\frac12]^d: T(0,x)\leq t\})$.

For a direction (unit vector) $\theta$ let \tred{$y_\theta$} be the point of $\pa\mkB_g$ in direction $\theta$, and let \tred{$\mH_\theta$} be a supporting hyperplane to $\mkB_g$ at $y_\theta$; when uniqueness fails the choice is arbitrary.  Then for $r\in\RR$ let \tred{$H_{\theta,r}$} be the hyperplane parallel to $\mH_\theta$ through $ry_\theta$.  
At times it will be convenient to express a general $u\in\RR^d$ in terms of a basis in which the first vector is $y_\theta$, and the other $d-1$ form an orthonormal basis for $H_{\theta,0}$. (The particular choice of orthonormal basis does not matter.)  We call these $\theta$--\emph{coordinates}, and projection along $H_{\theta,0}$ into $\Pi_{0\theta}^\infty$ is called \emph{tangential $\theta$--projection} and denoted \tred{$\pi_\theta$}.  Projection along $\Pi_{0\theta}^\infty$ into some $H_{\theta,r}$ is called \emph{longitudinal $\theta$--projection}. (This is only truly projection when $r=0$, but we use the terminology anyway.) 
As noted in \cite{Al20}, the angle between $\theta$ and $H_{\theta,0}$ is always at least arcsin $1/\sqrt{d}$; we call this the \emph{arcsin bound}.  The distance from a point $x$ to a set $A$ along $H_{\theta,0}$ is denoted \tred{$d_\theta(x,A)$}.

We say $\theta$ is a \emph{direction of sub--curvature} if there exists $\ep_0(\theta)>0$ and a value $c_{5}(\theta)$ such that for $u\in H_{\theta,1}$ with $|u-\yt|\leq\ep_0(\theta)$ we have 
\begin{equation}\label{curvg}
  g(u) \leq 1 + c_{5}|u-y_\theta|^2;
\end{equation}
$\theta$ is a \emph{direction of curvature} if also there exists $c_{6}(\theta)$ such that for all $u$ as above,
\begin{equation}\label{curvg2}
  g(u) \geq 1 + c_{6}|u-y_\theta|^2.
\end{equation}
There is then also a $c_{7}(\theta)$ such that for $0<\ep\leq\ep_0(\theta)$ and $u\in H_{\theta,1}$ with $|u-\yt|\geq\ep$ we have 
\begin{equation}\label{curvgwide}
  g(u) \geq 1 + c_{7}\ep |u-y_\theta|.
\end{equation}
Provided directions of curvature exist, we may choose values $\ep_0^*,c_5^*,c_6^*,c_7^*$ for which 
\[
  \tred{\mU} = \Big\{\theta\in S^{d-1}: \theta\text{ is a direction of curvature}, 
    \ep_0(\theta)\geq \ep_0^*,c_5(\theta)\leq c_5^*,c_6(\theta)\leq c_6^*,c_7(\theta)\leq c_7^*\Big\} \neq \emptyset.
\]
Equivalently, we omit the $*$ in the notation and take $\ep_0,c_5,c_6,c_7$ as values not depending on $\theta\in\mU$. Then, after reducing $\ep_0$ if necessary, there exists $c_{8}$ such that 
\begin{equation}\label{hyperang}
  \theta\in\mU,\, |\alpha - \theta|<\ep_0 \implies \text{ the angle between $H_{\alpha,0}$ and $H_{\theta,0}$ is at most } c_{8}|\alpha-\theta|.
\end{equation}

Convexity of $\mkB_g$ means the directions of sub--curvature form a set of full (uniform) measure in the unit sphere $S^{d-1}$.
By contrast, though it is believed that (outside of cases where the infimum of the support of $\tau_e$ has large mass, as in \cite{DL81}) every $\theta$ is a direction of curvature, there is no \emph{a priori} reason for the existence of any directions of curvature for $\mkB_g$; a polyhedron, for example, has none. 
We say $(\theta,x)\in S^{d-1}\times\ZZ^d$ is \emph{directionally good} if  
\begin{itemize}
\item[(a)] $\theta\in\mU$,
\item[(b)] $d(x,\Pi_{0\theta}^\infty)\leq d$,
\item[(c)] $|x|^{-1/4}<\ep_0$.
\end{itemize}
and \emph{directionally acceptable} if (a), (c) hold and
\begin{itemize}
\item[(b')] $d(x,\Pi_{0\theta}^\infty)\leq |x|^{1/5}$.
\end{itemize}
We write \tred{$\mG$} for the set of all directionally good pairs $(\theta,x)$,
\tred{$\mG_{\ZZ^d}$} for $\{x:(\theta,x)\in\mG$ for some $\theta\}$, \tred{$\mA$} for the set of all directionally acceptable pairs, 
and \tred{$\mA_{\ZZ^d}$} for $\{x:(\theta,x)\in\mA$ for some $\theta\}$.

Due to the arcsin bound we always have $|x-\pi_\theta x|\leq c_{9}$ for directionally good $(\theta,x)$ 
note that $x,\pi_\theta x$ lie in the same hyperplane $H_{\theta,g(\pi_\theta x)}$. Of course (c) only requires that $|x|$ be sufficiently large, 
so for each $\theta\in\mU$ there exists $r_0(\theta)$ such that
 \begin{equation}\label{thick}
   r\geq r_0(\theta) \implies \exists\, x: (x,\theta)\in\mG,\ ||x|-r|\leq \sqrt{d}.
 \end{equation}

Let \tred{$H_{\theta,r}^+,H_{\theta,r}^-$} denote the halfspaces $\cup_{s\geq r} H_{\theta,s}$ and $\cup_{s\leq r} H_{\theta,s}$, respectively, and for $r\leq t$ define the slab $\tred{S_\theta(r,t)} = \cup_{r\leq s\leq t} H_{\theta,s}$. We need to consider slab and disc--to--disc passage times.  An \emph{infinite $\theta$--cylinder} is a set of form $\{x\in\RR^d: d(x,\Pi_{0\theta}^\infty) \leq r\}$ with $\theta$ a unit vector and $r>0$.
A \emph{bounded $\theta$--cylinder} is the intersection of an infinite $\theta$--cylinder with a (not necessarily perpendicular) slab, whenever this intersection is bounded; we will typically simply say ``cylinder'' to mean ``bounded $\theta$--cylinder.''  We will also have occasion to use cylinders in which $d(x,\Pi_{0\theta}^\infty)$ is replaced by $d_\theta(x,\Pi_{0\theta}^\infty)$ in the definition; we distinguish these by attaching the adjective \emph{skew} (e.g.~an infinite skew $\theta$--cylinder);
note here that $d_\theta(x,\Pi_{0\theta}^\infty) = |x-\pi_\theta x|$.
Any bounded $\theta$--cylinder $\mC$ determines two \emph{end hyperplanes} containing its ends, and the \emph{associated slab} \tred{$S(\mC)$} between them.
Given $x\in\RR^d$ with direction $\phi=x/|x|$, and $s>0$, the \emph{natural hyperplanes} of $(0,x)$ are the hyperplanes parallel to $\mH_\phi$ through 0 and $x$, the \emph{natural slab} of $(0,x)$ is $\tred{S_{nat}(0,x)} = S_\phi(0,g(x))$ which is the region between the natural hyperplanes, and the \emph{natural cylinder} of $(0,x)$ of radius $s$ is the bounded $\phi$--cylinder
\[
  \Big\{ u \in\RR^d: d(u,\Pi_{0x}^\infty) \leq s \Big\} \cap S_{nat}(0,x), 
\]
which has ends in the natural hyperplanes.  
For $x\neq y$ in $\RR^d$, the natural cylinder (of radius $s$), natural hyperplanes, and natural slab of $(x,y)$ are then defined via translation by $x$ of the corresponding objects defined for $(0,y-x)$, and the cylinder and slab are denoted \tred{$\mC_{nat}(x,y,s)$} and $S_{nat}(x,y)$, respectively. For a bounded $\theta$--cylinder $\mC$ each edge which intersects an end hyperplane of $\mC$ and has an endpoint in the interior of $S(\mC)$ is called an \emph{end edge} of $\mC$, and the endpoint not in the interior is called an \emph{end vertex}.
A pair $(u,v)$ of end vertices at opposite ends of $\mC$ is called an \emph{end pair} of $\mC$, and the set of all end pairs is denoted \tred{$\mE(\mC)$}.  Similarly, for a slab $S$, pairs $(u,v)$ of sites adjacent to, but not in, the interior of $S$, lying on opposite sides of $S$, are called \emph{boundary pairs} of $S$, and the set of all boundary pairs is denoted \tred{$\mB(S)$}.

We fix $\ep_1$ to be specified, and say a slab $S_\alpha(\cdot,\cdot)$ is a \emph{near--natural slab} of $(x,y)$ if $(x,y)\in \mB(S)$, 
\[
  \left| \frac{y-x}{|y-x|} - \alpha \right|<\ep_1,
\]
and $S_\alpha(\cdot,\cdot)$ makes an angle less than $\ep_1$ with $S_{nat}(x,y)$.  The set of all such slabs is denoted \tred{$\mN(x,y)$}. For a fixed $\alpha$, whether $S_\alpha(r,s)\in \mN(x,y)$ does not depend on $r,s$ so we may write ``$S_\alpha(\cdot,\cdot)\in\mN(x,y)$.''  Given a direction of curvature $\theta$, there exists $\ep_2$ such that if both $\alpha$ and $(y-x)/|y-x|$ are within $\ep_2$ of $\theta$ then $S_\alpha(\cdot,\cdot)\in\mN(x,y)$.

We now decribe fully some of the properties FPP systems (at least below $d_c$) are in most cases believed to have, which may be considered ``core properties.'' Unfortunately they are mostly unproven, except in the case of exactly solvable LPP models in $d=2$.  There has been considerable work on LPP in which these (or similar) core properties are taken as ``black--box inputs'' from the algebraic methods, with further proofs then developed by probabilistic methods, for example \cite{BSS19}, \cite{BG21}.  These are some of the core properties, with $\sigma_{mag}$ any approximate standard deviation:

\begin{itemize}
    \item[(i)] \tred{Uniform exponential bound property} on scale $\hat\sigma(\cdot)$:
\begin{equation}\label{expbound2}
  P\Big( |T(x,y) - ET(x,y)| \geq t\hat\sigma(y-x) \Big) \leq c_{11}e^{-c_{12}t} \quad\text{for all } x,y\in\ZZ^d.
\end{equation}
     \item[(ii)] \tred{Lattice negligibility with a growth exponent}: A direction of curvature exists, $\hat\sigma$ is magnitude--based, and $\sigma_{mag}$ has a growth exponent $\chi\geq 0$.
     
     \noindent (a) First variant: (ii) holds and $D$ is magnitude--based.
     
     \noindent (b) Second variant:  (ii) holds and $\sigma_{mag}$ has a regular growth exponent $\chi\geq 0$.
     
     \item[(iii)] Limit shape has \tred{nondegenerate boundary curvature} (twice differentiable with positive definite Hessian) everywhere.
     \item[(iv)] \tred{Uniform moderate--gap property}: 
\[ 
  \limsup_{|x|\to\infty} \frac{ET(0,x) - g(x)}{\hat\sigma(x)} <\infty.
\]
    \noindent (a) Semi--uniform variant:
\[ 
  \liminf_{{|x|\to\infty}\atop{x\in\mG_{\ZZ^d}}} \frac{ET(0,x) - g(x)}{\hat\sigma(x)} < \infty.
\]    
    \noindent (b) Subsequence variant:
\[ 
  \liminf_{|x|\to\infty} \frac{ET(0,x) - g(x)}{\hat\sigma(x)} <\infty.
\]
     \item[(v)] \tred{Controlled wandering property}: 
\begin{equation}\label{weakw}
  \lim_{K\to\infty} \limsup_{{|x|\to\infty} \atop {x\in \mA_{\ZZ^d}}} P(R(x)\geq K\hat\Delta(x)) = 0.
\end{equation}
(a) Exponentially--controlled variant: for all $x\in \mA_{\ZZ^d}$ and all $K\leq c_{13}|x|^{1/2}$,
\[
  P\Big( R(x) \geq K\hat\Delta(x) \Big) \leq c_{14}e^{-c_{15}K^2}.
\]
\end{itemize}
Note (iv)(a) is stronger than (iv)(b). The principle underlying (ii) is that as distances become large, effects of the structure of the underlying lattice, other than its dimension, should become small (with an exception when percolation of the minimal edge weight occurs, as in \cite{DL81}); this should mean every $\theta$ is a direction of curvature, but for our purposes the weaker curvature condition suffices. The parts of (ii) referencing $\hsi$ and $\sigma_{mag}$ essentially say simply that a well--defined $\chi$ exists.
The principle underlying (v) that transverse wandering is typically of order $\hat\Delta(x)$ in a direction of curvature may be viewed as a more precise version of the exponent relation $\xi=(1+\chi)/2$, as we now describe.

\begin{remark}\label{modgap}
Assuming the uniform moderate--gap property (iv), the exponentially--controlled--wandering property (v)(a) in directions of curvature is closely related to the uniform exponential bound property (i).  Informally the reason is as follows.  If $R(x)\geq K\hat\Delta(x)$, there is a point $u\in\Gamma_{0x}$ at distance approximately $K\hat\Delta(x)$ from the line $\Pi_{0x}^\infty$ through 0 and $x$.  Due to curvature, this forces the geodesic to travel an extra distance:
\[
  g(u) + g(x-u) - g(x) \geq c_{16}\frac{(K\hat\Delta(x))^2}{|x|} = c_{16}K^2\hat\sigma(x).
\] 
For $K$ large relative to the lim sup in (iv) this implies
\[
  h(u) + h(x-u) - h(x) \geq \frac{c_{16}}{2}K^2\hat\sigma(x),
\] 
while
\[
  T(0,u) + T(u,x) - T(0,x) = 0.  
\]
This forces one of the following to occur:
\[
  T(0,u)-ET(0,u)\leq -\frac{c_{16}}{8}K^2\hat\sigma(x),\quad T(u,x)-ET(u,x)\leq -\frac{c_{16}}{8}K^2\hat\sigma(x),
\]
\[
  T(0,x)-ET(0,x)\geq \frac{c_{16}}{4}K^2\hat\sigma(x).
\]
Provided $\hat\sigma(u)$ and $\hat\sigma(x-u)$ are not of larger order than $\hat\sigma(x)$ (true if (ii)(b) holds), the probability for each of these is then bounded in view of (i), yielding (v)(a).  This is a heuristic, not a proof, because we have ignored that the point $u$ is random; what it shows is that (i) and (v)(a) are strongly interrelated.  In fact the idea that geodesics wander by order $\hat\Delta(x)$, when expressed at the level of exponents, becomes the standard relation $\xi=(1+\chi)/2$ proved (under the assumption these exponents exist, in a certain sense) by Chatterjee \cite{Ch13}.

Since $\hat\Delta(x)$ is at least of order $|x|^{1/2}$, this heuristic, or just the relation $\xi=(1+\chi)/2$, suggests that transverse wandering should also always be at least of order $|x|^{1/2}$.  This means that the natural assumption for the controlled--wandering property is not that $x/|x|$ be precisely a direction of curvature, but rather that $|x-\pi_\theta x| = O(|x|^{1/2})$ for some direction of curvature $\theta$.  We have defined ``directionally acceptable'' vertices $x$ in keeping with this, replacing the exponent 1/2 with 1/5 as we do not need 1/2 for our proofs. The value 1/5 is chosen to be smaller than the value 1/4 appearing in (c) in the definition of directionally acceptable.
\end{remark}

We will consider also a variant of (iv), and a property that connects it to (iv):
\begin{itemize}
     \item[(iv')] \tred{Uniform downward--deviation property}: for some $\tred{\ep_3}>0$,
\[
  \liminf_{|x|\to\infty} P\Big( T(0,x) \leq g(x) -\ep_3\hat\sigma(x) \Big) > 0.
\]
    \noindent (a) Semi--uniform variant:
\[ 
  \liminf_{{|x|\to\infty}\atop{x\in\mG_{\ZZ^d}}} P\Big( T(0,x) \leq g(x) -\ep_3\hat\sigma(x) \Big) > 0.
\]    
    \noindent (b) Subsequence variant:
\[ 
  \limsup_{|x|\to\infty} P\Big( T(0,x) \leq g(x) -\ep_3\hat\sigma(x) \Big) > 0.
\]
   \item[(vi)] \tred{Unbounded concentration property}: 
   \[
     \liminf_{|x|\to\infty} P\big(T(0,x) -ET(0,x)\leq -t\hat\sigma(x)\big)>0 \quad\text{for all } t>0.
   \]
\end{itemize} 
We see that (iv') implies (iv), by Chebyshev. Conversely, (iv) and (vi) imply (iv'). One expects (iv) is equivalent to (iv'), even without the unbounded concentration property (vi).  The unbounded concentration property holds provided $(T(0,x) -ET(0,x))/\hat\sigma(x)$ has a nondegenerate limit distribution for which the support is all of $\RR$, as is expected to be true, and is known for integrable LPP in $d=2$ \cite{Jo00}. To ensure the support is all of $\RR$, it is sufficient that the limit distribution be infinitely divisible.

For a boundary pair $u,v$ of a slab $S$, the \emph{slab passage time} \tred{$T(u,v\mid S)$} is the fastest passage time among all paths from $u$ to $v$ with all vertices in $S\cup\{u,v\}$. The standard deviation of the slab passage time is denoted \tred{$\hat\sigma(u,v\mid S)$}.
The \emph{disc--to--disc passage time} of a cylinder $\mC$, denoted \tred{$T_{d2d}(\mC)$}, is the minimum slab passage time $T(u,v\mid S(\mC))$ over all end pairs of $\mC$.  Its standard deviation is denoted \tred{$\hat\sigma_{d2d}(\mC)$}.  Note that disc--to--disc passage time are \emph{not} restricted to paths staying in the cylinder.

By Proposition \ref{slabnot} there is a constant $c_{17}$ such that 
\begin{equation}\label{slabvar}
  E[(T(0,x\mid S)-T(0,x))^2] \leq c_{17} \ \ \text{for all $x\in\mG_{\ZZ^d}$ and all } S\in\mN(0,x).
\end{equation}
It follows that any approximate standard deviation $\sigma_{mag}$ also works for slab passage times: after adjusting $c_1$,
\begin{equation}\label{slabapp}
  c_1^{-1}\sigma_{mag}(|x|) \leq \hat\sigma(0,x\mid S) \leq c_1 \sigma_{mag}(|x|) \ \ \text{for all $x\in\mG_{\ZZ^d}$ and all } S\in\mN(0,x).
\end{equation}

The last of the properties we consider is the localness (determination near the geodesic endpoints) of passage times discussed in Remark \ref{localness}, which is not expected to hold when $\chi>0$.  
To describe the property first informally, given a natural cylinder $\mC_{nat}(0,x,\ep\hat\Delta(x))$ with radius a small multiple of $\hat\Delta(x)$, one may compare the typical slab passage time between each end pair $(u,v)$ and the typical fastest of these slab passage times, which is the disc--to--disc passage time.  The property of interest is essentially that
the typical disc--to--disc passage time for such cylinders is many standard deviations faster than the typical time for each fixed end pair.  Since any two end pairs can share the same geodesic path everywhere except near the endpoints, this large passage--time variation across end pairs must be due to some version of localness. We quantify the localness property precisely as follows.

\begin{itemize}
  \item[(vii)] \tred{Local fluctuation property}:  for every $\ep_+>0$, for all sufficiently small $\ep_-\in(0,\ep_+)$, letting $\tred{\mC_{x,\ep}}= \mC_{nat}(0,x,\ep\hat\Delta(x))$,
  \begin{equation}\label{localf}
    \lim_{r\to\infty} \sup_{{|x|\geq r} \atop {x\in x}} \sup_{\ep_-<\ep<\ep_+} 
      \min_{(u,v)\in\mE(\mC_{x,\ep})} \frac{ ET(u,v\mid S(\mC_{x,\ep})) - ET_{d2d}\big(\mC_{x,\ep} \big) }
      {\hat\sigma(u,v\mid S(\mC_{x,\ep}))} = \infty.
  \end{equation}
\end{itemize}

By contrast, a version of the ''nonlocalness'' scenario, in which point--to--point and disc--to--disc passage times are not very different, has been confirmed in \cite{Al20}, \cite{Ga20} when $\chi>0$ and (i), (ii) hold, so localness is only a realistic possibility when $\chi=0$.

There are various results providing bounds on $|T(u,v) - T(u',v')|$ when $u,v$ are close to $u',v'$ respectively, in contexts where we know or expect $\chi>0$; these are proved using the path--sharing described in Remark \ref{localness}.  In \cite{Ga20} such bounds are established for $d=2$ assuming versions of (i), (ii), (iii); a similar result for general $d$ is is \cite{Al20}.  Though we will not formally prove that the local fluctuation property fails when $\chi>0$, these bounds point strongly toward the nonlocalness scenario in Remark \ref{localness}.

Our first theorem shows that certain properties expected to hold whenever $\chi>0$ (and conditionally proved in some cases---see below) become inconsistent when $\chi=0$.

\begin{theorem}\label{main}
\tcyn{ For a standard FPP in $d$ dimensions, at least one of the following must be false:
\begin{itemize}
\item[(ii)] lattice negligibility, with growth exponent $\chi=0$;
\item[(v)] the controlled--wandering property; 
\item[(iv')(a)] the semi--uniform downward--deviation property.
\end{itemize}}
\end{theorem}

The state of these properties when $\chi>0$ is as follows.
\begin{itemize}
\item[(ii)] is proved with $\chi=1/3$ for integrable models of LPP in $d=2$ \cite{Jo00}. As noted before Remark \ref{modgap}, it is a manifestation of the universality principle that at large scales, the underlying lattice should become irrelevant, except for its dimension. Existence of a regular growth exponent is also a manifestation of the kind of regularity that underlies the existence of a scaling limit.
\item[(v)] is a finer--detailed manifestation of the same principle that yields the exponent relation $\xi=(1+\chi)/2$, as noted in Remark \ref{modgap}.  The stronger exponentially--controlled variant (v)(a) is known for integrable models of LPP in $d=2$ \cite{BSS16}. (v)(a) is also conditionally proved on certain ``lattice--like'' isotropic random graphs in $\ZZ^d$ when $\chi>0$ \cite{Al20a}, under unproven assumptions (i) and (ii)(b) (slightly modified.)
\item[(iv')(a)] is proved for integrable models of LPP in $d=2$ \cite{Jo00}, where it is a consequence of (iv) and the existence of a scaling limit that yields (vi). Its cousin (iv) is conditionally proved for FPP on the same lattice--like random graphs as above in $\RR^d$ when $\chi>0$, assuming again (i) and a modified (ii)(b).
\end{itemize}

Of course one would like to know which of (ii), (v), (iv')(a) in Theorem \ref{main} actually fails when $\chi=0$, and to that end we have the following.  (ii) and (v) seem unlikely candidates to fail, as they are manifestations of non--dimension--dependent principles like universality, the existence of a scaling limit, and the heuristic that underlies $\xi=(1+\chi)/2$.  (iv')(a), on the other hand, does not appear to be a manifestation of any dimension--free general principle, and the conditional proof of its cousin (iv)(a) in \cite{Al20a} makes significant use of $\chi>0$, so it appears the clear candidate to fail when $\chi=0$.

We consider the consequences, if indeed (iv')(a) fails, in the next theorem and corollary, pointing toward the ``local'' nature of passage time fluctuations.  The statement of the theorem is in terms of the related properties (iv)(a)(b), however.

\begin{theorem}\label{cylpass}
\tcyn{ Suppose that for some standard FPP in $d$ dimensions, the following hold:
\begin{itemize}
\item[(ii)(a)] lattice negligibility, with growth exponent $\chi=0$, 
\item[(v)] the controlled--wandering property, 
\end{itemize}
but the subsequential moderate--gap property (iv)(b) fails. Then the local fluctuation property (vii) holds. }

\tcyn{ If (ii)(b) holds in addition to (ii)(a), then we may replace (iv)(b) with (iv)(a). }
\end{theorem}

As mentioned above, for FPP on the above--mentioned ``lattice--like'' isotropic random graphs in $\RR^d$, the uniform moderate--gap property (iv) (so also (iv)(a)(b)) is a consequence of (i) and (ii)(b) when $\chi>0$ \cite{Al20a}, so we expect (iv)(b) to potentially fail only when $\chi=0$.  So we expect the analog of Theorem \ref{cylpass} for $\chi>0$ to be vacuous.

The following is essentially immediate from Theorems \ref{main} and \ref{cylpass}.  

\begin{corollary}\label{cor}
\tcyn{ Assume that for some standard FPP in $d$ dimensions, the following hold:
\begin{itemize}
\item[(ii)(a,b)] lattice negligibility, with regular growth exponent $\chi=0$,  
\item[(v)] the controlled--wandering property, 
\item[(vi)] the unbounded concentration property.
\end{itemize} 
Then the local fluctuation property (vii) holds. }
\end{corollary}
\begin{proof}
Under the first two given hypotheses, by Theorem \ref{main} (iv')(a) is false.  Under (vi) this is equivalent to falsity of (iv)(a).  Then Theorem \ref{cylpass} (particularly the last sentence) says (vii) holds.
\end{proof}

In the corollary, the unbounded concentration property is only needed to bridge the gap from (iv)(a) to (iv')(a), so if one could show (iv)(a) alone implies (iv')(a), then the unbounded concentration property would be unnecessary.  The bridging only works because Theorem \ref{main} and Theorem \ref{cylpass} involve the same semi--uniform variant of (iv)(a) or (iv')(a) when $\sigma_{mag}$ has \emph{regular} growth exponent $\chi=0$, that is, when (ii)(b) holds. Without the ``regular'' aspect, one theorem involves the semi--uniform variant, the other the subsequence variant, and the bridge cannot be made. 

\section{Proof of Theorem \ref{main}}
We will need the following two lemmas, which are proved in Section \ref{lempf}.  We will apply the first to $f=\sigma_{up}$.

\begin{lemma}\label{regzero}
\tcyn{Suppose $f:[1,\infty)\to[1,\infty)$ has regular growth exponent $\chi=0$ with
\[
  \tred{\delta(r)} := \frac{\log f(r)}{\log r} \ \text{nonincreasing}.
\]
Let $\Psi(r)\to\infty,\ \Xi(r)\searrow 0$, and $\eta(r)\searrow 0$ as $r\to\infty$ with
\begin{equation}\label{PsiXi}
  \eta(r) \geq r^{-1/2},\qquad 
    \log\log \frac{1}{\Xi(r)} + \log \frac{1}{\eta(r)} \leq \frac{C}{\delta(r)}
\end{equation}
for some $C>0$. Then
\begin{equation}\label{bigsig}
  \left( \log\frac{1}{\Xi(\eta(r)r)} \right) f\left( \frac{\eta(r)r}{\log \frac{1}{\Xi(\eta(r)r)} } \right) \gg f(r)\ \ \text{as } r\to\infty.
\end{equation}}
\end{lemma}

We define a linear map $\Phi_\theta$, which approximates $g$ in directions near $\theta$ and is constant on hyperplanes $H_{\theta,r}$, by
\[
  \tred{\Phi_\theta(u)} = \begin{cases} g(\pi_\theta u) &\text{if } u\in H_{\theta,0}^+\\ -g(\pi_\theta u) &\text{if } u\in H_{\theta,0}^-. \end{cases}
\]
For vectors $y$ and angles $\rho$ we use 
\[
  \tred{\Theta_\rho(y)} = \frac{|y-\pi_\rho y|}{g(\pi_\rho y)}
\]
as a surrogate for the angle between $y/|y|$ and $\rho$; in view of the arcsin bound that angle and $\Theta_\rho(y)$ differ by at most a constant factor provided the angle is small.  

The second lemma controls backtracking of $\Gamma_{0x}$, measured in any direction $\theta$ close to $x/|x|$, by a distance $r$ or more.

\begin{lemma}\label{backtrk}
\tcyn{Consider a standard FPP in $d$ dimensions. Given $K\geq 1$ there exist constants $c_i$ as follows.  For all sufficiently large $|x|$, all $c_{18}(|x|\log |x|)^{1/2} \leq r \leq K|x|$, and all $\theta$ with $S_\theta(0,\Phi_\theta(x))\in \mN(0,x)$,
\begin{equation}\label{backbound}
  P\Big(\Gamma_{0x} \not\subset S_\theta\big(-r,\Phi_\theta(x)+r)\big) \Big) \leq c_{19}e^{-c_{20}r^2/|x|}.
\end{equation}}
\end{lemma}

Suppose now $\hat\sigma$ is magnitude--based, $\sigma_{mag}$ has growth exponent $\chi= 0$,
and (iv')(a), (v) both hold; we will get a contradiction.  Let $\sigma_{up}$ be an upper--regular standard deviation and let $c_1$ be as in \eqref{fmag} for $f=\hat\sigma,f_{mag}=\sigma_{mag}$.  By \eqref{upperf} there exist $c_{21}$ and an unbounded $\tred{\Lambda} \subset(0,\infty)$ such that 
\[
  r\in\Lambda \implies \sigma_{up}(r) \leq c_{21}\sigma_{mag}(r).
\]
Let
\[
  \tred{\Lambda^+} = \cup_{r\in\Lambda} \left[ \frac34 r,\frac32 r \right], \quad \tred{\Lambda^{++}} = \cup_{r\in\Lambda} \left[ \frac r2, 2r \right].
\]
Then by \eqref{reggr} and \eqref{regfluct3}, for some $c_{22}$,
\begin{align}\label{Lplus}
  r\in\Lambda^{++} &\implies s\in\Lambda\ \text{for some } s\in \left[ \frac r2, 2r \right] \notag\\
  &\implies \sigma_{mag}(r) \geq \frac{1}{c_1^2}\sigma_{mag}(s) \geq \frac{1}{c_{21}c_1^2}\sigma_{up}(s) \geq c_{22}\sigma_{up}(r).
\end{align}

It follows from Lemma \ref{backtrk} that for some $c_{23},c_{24}$, for $\tred{\ep_4(r)} = c_{23}(r^{-1}\log r)^{1/2}, \tred{\ep_5(r)}= r^{-c_{24}}$ we have
\begin{align}\label{epr}
  P\Big(&\Gamma_{0v} \not\subset S_\phi\big(-\ep_4(|v|) g(\pi_\phi v),(1+\ep_4(|v|))g(\pi_\phi v)\big) \Big) \leq \ep_5(|v|) \notag\\
  &\qquad \text{for all $v\in\ZZ^d$ and all $\phi$ with } S_\phi(\cdot,\cdot)\in \mN(0,v).
\end{align}
By (v), for any function $\Psi(r)\to\infty$ there exists $\Xi(r)\searrow 0$ for which
\begin{equation}\label{basicpsi}
  P\big(R(x) \geq \Psi(|x|)\hat\Delta(x)\big)\leq \Xi(|x|)\ \ \text{for all } x\in\mA.
\end{equation}
Our first task is to select particular choices of $\Psi,\Xi$ with other desired properties.  We first select \tred{$\Xi_1(r)$} satisfying
\begin{equation}\label{Xi1}
  \Xi_1(r)\searrow 0,\quad \Xi_1(r) \gg \ep_4\left( \frac r3 \right)^{(d-1)/6}.
\end{equation}
The last property allows us to next select \tred{$\Psi(r)$} satisfying
\begin{equation}\label{Psi}
   \Psi(r) \nearrow\infty,\quad \Psi(r) \leq \min\left( \frac{\Xi_1(r)^{1/(d-1)}}{(3\ep_4(r/3))^{1/6}}, r^{1/32}, e^{1/\delta(r)} \right),
\end{equation}
where, as in Lemma \ref{regzero},
\[
  \delta(r) = \frac{\log\sigma_{up}(r)}{\log r}.
\]
Then using (v) we take \tred{$\Xi_2,\Xi$} satisfying
\begin{equation}\label{Xi2}
  \Xi_2(r)\searrow 0,\quad \Xi_2(r) \geq  \sup_{{|x|\geq r} \atop {x\in \mA_{\ZZ^d}}} P\big(R(x) \geq \Psi(|x|)\hat\Delta(x)\big), \quad
    \Xi = \max\left( \Xi_1,\Xi_2,1/\Psi,\ep_4 \right),
\end{equation}
so $\Xi(r)\searrow 0$. Finally let
\begin{equation}\label{etar}
  \tred{\eta(r)} = \left( \frac{\Xi(r)^{1/(d-1)}}{\Psi(r)} \right)^8,
\end{equation}
so $\eta(r)\searrow 0$. Then from \eqref{Xi2},
\begin{equation}\label{PsiXi2}
  \log\log \frac{1}{\Xi(r)} + \log \frac{1}{\eta(r)} = \log\log \frac{1}{\Xi(r)} + \frac{8}{d-1}\log \frac{1}{\Xi(r)} + 8\log \Psi(r)
    \leq 17\log\Psi(r) \leq \frac{17}{\delta(r)},
\end{equation}
while from \eqref{Psi} and \eqref{Xi2}, since $\Xi\geq\Psi^{-1}$,
\begin{equation}\label{etalow}
  \eta(r) \geq \Psi(r)^{-8d/(d-1)} \geq \Psi(r)^{-16} \geq r^{-1/2}.
\end{equation} 
Thus \eqref{PsiXi} holds, and, by Lemma \ref{regzero}, also \eqref{bigsig} for $f=\sigma_{up}$.  

Fix a directionally good \tred{$(\theta,x)$}. Write \tred{$\eta$} for $\eta(|\pi_\theta x|)$, recall $\ep_3$ from (iv')(a), and let $\tred{n}\geq 1$. Let $\tred{y^*}=y^*(n)$ be the closest lattice site to $\eta \pi_\theta x/2n$, so provided $\eta|\pi_\theta x|/n$ is large, $(\theta,y^*)$ and $(\theta,2y^*)$ are directionally good. Suppose that for some $c_{25}$ to be specified,
\begin{equation}\label{inLam}
   \frac{\eta |\pi_\theta x|}{n} \in\Lambda^+ \ \ \text{and}\ \ 
     n\sigma_{up}\left( \frac{\eta|\pi_\theta x|}{n} \right) \geq c_{25} \sigma_{up}(|\pi_\theta x|).
\end{equation}
Then for $c_1$ from \eqref{reggr} and $c_{22}$ from \eqref{Lplus} we have 
\[
  c_1^{-3}c_{22}\sigma_{up}\left( \frac{\eta |\pi_\theta x|}{n} \right) \leq c_1^{-1}c_{22}\sigma_{up}\left( 2|y^*| \right) 
    \leq c_1^{-1}\sigma_{mag}\left( 2|y^*| \right) 
    \leq \hat\sigma(2y^*).
\]
Therefore by (ii), (v), \eqref{Lplus}, \eqref{inLam} there exists $c_{26}$ such that 
\begin{align}\label{AHlem}
  P&\left( T(0,2ny^*) - g(2ny^*) \leq -c_1^{-3}c_{22}c_{25}\ep_3 \sigma_{up}( |\pi_\theta x|) \right) \notag\\
  &\geq P\left( T(0,2ny^*) - g(2ny^*) \leq  -c_1^{-3}c_{22}\ep_3 n\sigma_{up}\left( \frac{\eta |\pi_\theta x|}{n} \right) \right)  \notag\\
  &\geq P\left( T\left( 0,2y^* \right) - g\left(2y^* \right)
    \leq -c_1^{-3}c_{22}\ep_3\sigma_{up}\left( \frac{\eta |\pi_\theta x|}{n} \right) \right)^n \notag\\
  &\geq P\Big( T\left( 0,2y^* \right) - g\left( 2y^* \right)
    \leq -\ep_3\hat\sigma(2y^*) \Big)^n \notag\\
  &\geq e^{-c_{26}n}.
\end{align}

The particular $n$ of interest is given by
\begin{equation}\label{bhat}
  \tred{\hat n} = \left\lfloor \frac{1}{c_{26}}\log \frac{1}{4\Xi(\eta|\pi_\theta x|/2)} \right\rfloor, 
    \quad\text{so}\quad e^{-c_{26}\hat n} \in\left[ 4\Xi(\eta|\pi_\theta x|/2), 4e^{c_{26}}\Xi(\eta|\pi_\theta x|/2) \right],
\end{equation}
and we write \tred{$\hat y^*$} for $y^*(\hat n)$, the closest lattice site to $\eta \pi_\theta x/2\hn$. Since, as noted after \eqref{etalow}, \eqref{bigsig} holds for $f=\sigma_{up}$, for $|x|$ large the inequality in \eqref{inLam} holds for $n=\hat n$, so \eqref{AHlem} and \eqref{bhat} give
\begin{equation}\label{centlink}
  P\left( T(0,2\hn\hy^*) - g(2\hn\hy^*) \leq -c_1^{-3}c_{22}c_{25}\ep_3 \sigma_{up}( |\pi_\theta x|) \right)
    \geq 4\Xi(\eta|\pi_\theta x|/2).
\end{equation}

Near $H_{\theta,\Phi_\theta(x)/2}$ (halfway between 0 and $\pi_\theta x$) we can select a ``well--separated'' deterministic set $\tred{Z}=\{z_1,\dots,z_m\}\subset\ZZ^2$ of points satisfying (with $c_{27}$ to be specified)
\[
  d(z_i,H_{\theta,\Phi_\theta(x)/2}) \leq d,\quad d_\theta\left(z_i,\Pi_{0\theta}^\infty\right) \leq \Delta_{up}(|\pi_\theta x|)\quad\text{ for all } i,
\]
\begin{equation}\label{zidef}
  |z_i-z_j| > c_{27}\Psi(2\eta^{3/4}|\pi_\theta x|)\Delta_{up}(2\eta^{3/4}|\pi_\theta x|)\quad\text{for all } i\neq j,
\end{equation}
with
\[
   \tred{m}=|Z| \geq 
     c_{28}\left( \frac{\Delta_{up}(|\pi_\theta x|)}{c_{27}\Psi(2\eta^{3/4}|\pi_\theta x|)\Delta_{up}(2\eta^{3/4}|\pi_\theta x|)} \right)^{d-1}.
\]
Then since $\Delta_{up}$ has regular growth exponent 1/2, provided $\eta$ is small we have
\[
  \frac{\Delta_{up}(|\pi_\theta x|)}{\Delta_{up}(2\eta^{3/4}|\pi_\theta x|)} \geq \eta^{-1/4} 
    \geq 2c_{27}\left( \frac{1}{c_{28}} \right)^{1/(d-1)} \eta^{-1/8},
\]
and hence
\begin{equation}\label{mlower}
  m \geq \frac{1}{(\eta^{1/8} \Psi(2\eta^{3/4}|\pi_\theta x|))^{d-1}} 
    \geq \frac{1}{(\eta^{1/8}\Psi(|\pi_\theta x|))^{d-1}} = \frac{1}{\Xi(|\pi_\theta x|)}.
\end{equation}

Generalizing slab passage times, for $\mD\subset \RR^d$ and $w,y\in \mD$ let \tred{$T(w,y\mid\mD)$} denote the fastest time among all paths in $\mD$ from $w$ to $y$.  Let \tred{$\alpha$} be the direction of $\hy^*$ and let $\tred{\hm} = \lfloor 2\eta^{-1/4}\hn \rfloor$, so (loosely) $\hn\hy^*\approx \eta \pi_\theta x/2$ and $\hm\hy^*\approx \eta^{3/4} \pi_\theta x$. We observe that since $1/\Xi(r) \leq \Psi(r) \leq r^{1/32}$ we have using \eqref{etalow} that
\begin{equation}\label{nhatbd}
  \hn\leq c_{30}\log(\eta |\pi_\theta x|) \quad\text{and hence}\quad |\hy^*| \geq c_{31}\frac{|\pi_\theta x|^{1/2}}{\log |\pi_\theta x|}.
\end{equation}
Since $|\hy^* - \pi_\theta\hy^*| \leq c_{32}$ it follows readily that
\begin{equation}\label{angley}
  |\alpha - \theta| \leq c_{33}\Theta_\theta(\hy^*) \leq \frac{c_{34}}{|\hy^*|} \leq c_{35} \frac{\log |\pi_\theta x|}{|\pi_\theta x|^{1/2}}.
\end{equation}
It also follows from \eqref{nhatbd} that $(\theta,\hy^*)\in\mG$, and, since $\Xi(r) \geq r^{-1/32}$,
\begin{equation}\label{mhatbd}
  \hn \leq c_{36}(\eta|\pi_\theta x|)^{1/20} \leq c_{36}\eta^{1/4}|\pi_\theta x|^{3/20} \quad\text{so}\quad \hm \leq 2c_{36}|\pi_\theta x|^{3/20},
\end{equation}
which ensures $(\theta,j\hy^*)\in\mA$ for all $j\leq\hm$.

For each $z_i$ we define the infinite $\alpha$--cylinder \tred{$\mC_i$} with axis through $z_i$ and radius $\Psi(2\eta^{3/4}|\pi_\theta x|)\Delta_{up}((\hm-\hn)|\hy^*|)$ (taking $c_{27}$ large enough in \eqref{zidef} so that the cylinders $\mC_i$ are disjoint), and then define
\[
  \tred{\mC_i^{short}} 
    = \mC_i \cap S_\alpha\Big(\Phi_\alpha(z_i - 2\hn\hy^*),\Phi_\alpha(z_i + 2\hn\hy^*)\Big),
\]
\[
  \tred{\mC_i^{long}} = \mC_i \cap S_\alpha\Big(\Phi_\alpha(z_i - 2\hm\hy^*),\Phi_\alpha(z_i + 2\hm\hy^*)\Big), 
\]
Note the length of $\mC_i^{short}$ is about twice the distance between the points $z_i - \hn\hy^*,z_i + \hn\hy^*$ inside it, and the length of $\mC_i^{long}$ is about four times the distance between the points $z_i + \hm\hy^*,z_i + \hm\hy^*$ inside it. For each $i$ we then get a path of five links from 0 to $x$: $0\to z_i - \hm\hy^*\to z_i - \hn \hy^*\to z_i + \hn \hy^*\to z_i + \hm\hy^*\to x$; see Figure \ref{figThm1.4}.
Define corresponding passage times 
\begin{align}\label{Tvars}
  \tred{T_i^{--}} &= T\Big(0, z_i - \hm\hy^* \ \Big|\ H_{\theta,(1-\eta^{3/4}) \Phi_\theta(x)/2}^- \Big), \notag\\
  \tred{T_i^-} &= T\Big(z_i - \hm\hy^*, z_i - \hn \hy^* \ \Big|\ \mC_i^{long} \Big) \notag\\
  \tred{T_i^0} &= T\Big(z_i - \hn\hy^*,z_i + \hn\hy^* \ \Big|\ \mC_i^{short} \Big) \notag\\
  \tred{T_i^+} &= T\Big(z_i + \hn\hy^*,z_i + \hm\hy^* \ \Big|\ \mC_i^{long} \Big) \notag\\
  \tred{T_i^{++}} &= T\Big(z_i + \hm\hy^*,x \ \Big|\ H_{\theta,(1+\eta^{3/4}) \Phi_\theta(x)/2}^+ \Big),
\end{align}
so
\begin{equation}\label{subadd}
  T(0,x) \leq T_i^{--} + T_i^- + T_i^0 + T_i^+ + T_i^{++}.
\end{equation}
The first and last link we call \emph{outer links}, the third one is the \emph{central link}, and the other two are \emph{intermediate links}.  Note that the central and outer links as given in \eqref{Tvars} occur in disjoint regions so the passage times of all central links are independent of the passage times of all outer links; this is the reason for inserting the intermediate links between them. We will show that with probability bounded away from 0, for some $i$, the central link, though short, is fast enough to make the entire 5--link path ``fast.''

\begin{figure}
\includegraphics[height=5.6cm]{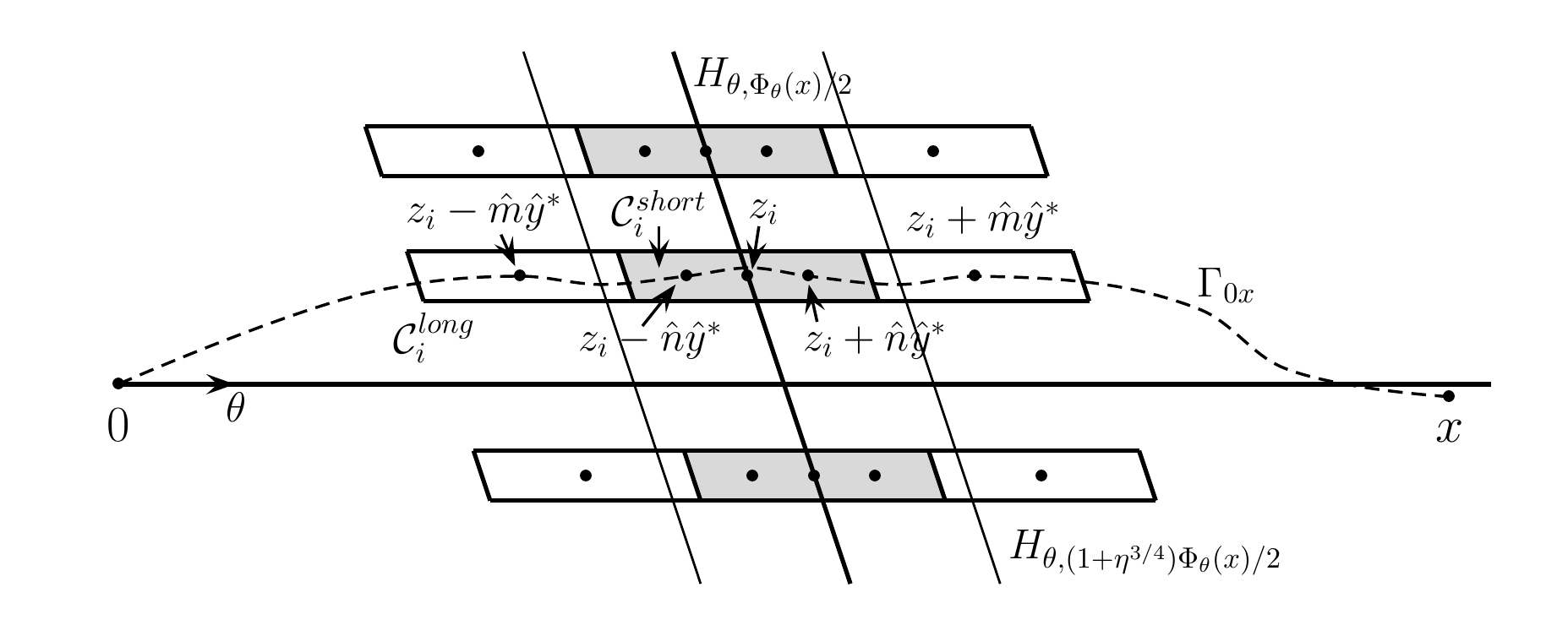}
\caption{ A 5--link path using the $\alpha$--cylinder $\mC_i^{long}$ centered at $z_i$. $\mC_i^{short}$ is shaded. The drawing is not to scale, as the short--cylinder length, the long--cylinder length, and $|x|$ are all on different scales. The rightmost hyperplane $H_{\theta,(1+\eta^{3/4}) \Phi_\theta(x)/2}$ lies about halfway between $z_i+\hat m\hat y^*$ and the center hyperplane $H_{\theta, \Phi_\theta(x)/2}$, and the outer--link path on the right is located to the right of the rightmost hyperplane, so it does not intersect any short cylinder. The cylinders have axis in the direction $\alpha$ of $\hy^*$, slightly different from $\theta$. }
\label{figThm1.4}
\end{figure}

In dealing with paths from some $y$ to $z$ in a halfspace $H_{\theta,r}^\pm$ or slab $S_\alpha(q,r)$ (as for example in \eqref{Tvars}) we have in general $y\in H_s$ and $z\in H_t$ for some $s< t$. In this context we refer to $\rho/(t-s)$ as the \emph{relative margin}, where $\rho$ is the smaller of the $g$--distances of $y$ and $z$ to the boundary of the halfspace (or slab.)  Some approximate relative margins which are relevant here are
\begin{align}
\text{in \eqref{backbound}:}\ &r/\Phi_\theta(x)\ (\text{so at least } c_{29}(|x|^{-1}\log |x|)^{1/2}) \notag\\
\text{in \eqref{epr}:}\ &\ep_4(|v|) \notag\\
\text{in $T_i^0$:}\ &1/2 \notag\\
\text{in $T_i^\pm$:}\ &\text{more than 1} \notag\\
\text{in $T_i^{++}$ and $T_i^{--}$:}\ &\eta^{3/4}/2(1-\eta^{3/4}) \geq \eta^{3/4}/2.
\end{align}
Here in the case of $T_i^{++}$ and $T_i^{--}$ we are using \eqref{hyperang} and \eqref{angley} to conclude that the margin is little affected by the choice between using a $\theta$--based versus $\alpha$--based halfspace in \eqref{Tvars}.
From \eqref{Xi2} and \eqref{epr}, together with 
\[
  2\eta^{3/4}|\pi_\theta x| \geq (\hm-\hn)|\hy^*| \geq \frac{\eta^{3/4}}{2}|\pi_\theta x|,\quad \hm\geq 2\hn, \quad 2\hn|\hy^*|\geq \frac{\eta|\pi_\theta x|}{2}
\]
and the fact that (as noted after \eqref{mhatbd}) $(\hm-\hn)\hy^*\in\mA_{\ZZ^d}$, we have
\begin{align}\label{leavetube}
  P\Big( \Gamma_{z_i - \hn \hy^*,z_i + \hn \hy^*} \not\subset \mC_i^{short} \Big) 
    &\leq P\Big(R(2\hn \hy^*) \geq \Psi(2\eta^{3/4}|\pi_\theta x|)\Delta_{up}((\hm-\hn)|\hy^*|) \Big) \notag \\
  &\qquad + P\Big(\Gamma_{z_i - \hn \hy^*,z_i + \hn \hy^*} 
    \not\subset S_\alpha\Big(\Phi_\alpha(z_i - 2\hn\hy^*),\Phi_\alpha(z_i + 2\hn\hy^*)\Big) \Big) \notag\\
  &\leq \Xi\left( 2\hn |\hy^*| \right) + \ep_5(2\hn|\hy^*|) \notag\\
  &\leq \Xi\left( \frac{\eta}{2}|\pi_\theta x| \right) + \ep_5\left( \frac{\eta}{2}|\pi_\theta x| \right),
\end{align}
where we used that in the last probability in \eqref{leavetube}, the relative margin is near $1/2 \geq \ep_4(2\hn|y^*|)$.
Since $\Xi\geq\ep_5$, from \eqref{centlink} and \eqref{leavetube} we get
\begin{align}\label{eachi}
  P&\Big(T_i^0 \leq g(2\hn\hy^*) - c_1^{-3}c_{22}c_{25}\ep_3 \sigma_{up}(|\pi_\theta x|) \Big) \notag\\
  &\geq P\Big( T(z_i - \hn\hy^*,z_i + \hn\hy^*) 
    \leq g(2\hn\hy^*) - c_1^{-3}c_{22}c_{25}\ep_3 \sigma_{up}(|\pi_\theta x|) \Big) \notag\\
  &\qquad - P\left( \Gamma_{z_i - \hn\hy^*,z_i + \hn\hy^*} \not\subset \mC_i^{short} \right) \notag\\
  &\geq  4\Xi\left( \frac{\eta}{2}|\pi_\theta x| \right) - 2\Xi\left( \frac{\eta}{2}|\pi_\theta x| \right) = 2\Xi\left( \frac{\eta}{2}|\pi_\theta x| \right).
\end{align}
Define \tred{$I$} by
\[
  \min_{i\leq m} T_i^0 = T_I^0.
\]
Then since the variables $T_i^0$ are independent, from \eqref{eachi} and \eqref{mlower} we get
\begin{align}\label{onefast}
  P&\left(T_I^0 \geq g(2\hn\hy^*) - c_1^{-3}c_{22}c_{25}\ep_3 \sigma_{up}(|\pi_\theta x|) \right) 
    \leq \left(1-2\Xi(\eta|\pi_\theta x|/2) \right)^m \notag\\
  &\hskip 2.5cm \leq \exp\left(-2m\Xi(|\pi_\theta x|)\right) \leq  e^{-2} < \frac12.
\end{align}

Similar but simpler reasoning applies to the intermediate--link times $T_I^-$ and $T_I^+$.  Let $\tred{p_0}>0$ be the lim inf in (iv')(a).  Similarly to \eqref{leavetube}, since the relative margin here is more than 1 we have
\begin{align}\label{leavetube2}
  P\Big( &\Gamma_{z_i + \hn\hy^*,z_i + \hm\hy^*} \not\subset \mC_i^{long} \Big) \notag\\
  &\leq P\Big(R((\hm-\hn)\hy^*) \geq \Psi(2\eta^{3/4}|\pi_\theta x|)\Delta_{up}((\hm-\hn)|\hy^*| \Big) \notag \\
  &\qquad + P\Big(\Gamma_{z_i + \hn\hy^*,z_i + \hm\hy^*}
    \not\subset S_\alpha\Big(\Phi_\alpha(z_i - 2\hm\hy^*),\Phi_\alpha(z_i + 2\hm\hy^*) \Big) \Big) \notag\\
  &\leq P\Big(R((\hm-\hn)\hy^*) \geq \Psi(2\eta^{3/4}|\pi_\theta x|)\Delta_{up}((\hm-\hn)|\hy^*|) \Big) 
    + \ep_5((\hm-\hn)|y^*|) \notag\\
  &\leq \Xi\left( \frac{\eta^{3/4}}{2}|\pi_\theta x| \right) + \frac{p_0}{8},
\end{align}
where in the last inequality we used $2\eta^{3/4}|\pi_\theta x| \geq (\hm-\hn)|y^*| \geq \eta^{3/4}|\pi_\theta x|/2$.
Then as in \eqref{eachi}, provided $|x|$ is large we have
\begin{align}\label{midi}
  P\Big(T_i^+ \leq g\big((\hm-\hn)\hy^*\big) \Big) 
    &\geq P\Big( T(z_i + \hn\hy^*,z_i + \hm\hy^*) \leq g\big((\hm-\hn)y^*\big) \Big) \notag\\
  &\qquad - P\Big( \Gamma_{z_i + \hn\hy^*,z_i+\hm\hy^*} 
    \not\subset \mC_i^{long} \Big) \notag\\
  &\geq  \frac{p_0}{2} - \Xi\left( \frac{\eta^{3/4}}{2}|\pi_\theta x| \right) - \frac{p_0}{8} \notag\\
  &\geq \frac{p_0}{4}.
\end{align}
Symmetrically we have
\begin{equation}\label{midi2}
  P\Big(T_i^- \leq g\big((\hm-\hn)y^*\big) \Big) \geq \frac{p_0}{4}.
\end{equation}

Considering now the outer--link time $T_i^{++}$, from \eqref{etalow}, \eqref{mhatbd}, and the definitions of $\hy^*,\hm$ we have for large $|x|$
\begin{equation}\label{mhat}
  |\hm\hy^* - \eta^{3/4} \pi_\theta x| \leq c_{37}\hm 
    \quad\text{so}\quad
    |\Phi_\theta(\hm\hy^*) - \eta^{3/4} \Phi_\theta(x)| \leq c_{38}\hm \leq \frac{1}{12} \eta^{3/4} \Phi_\theta(x),
\end{equation}
and hence 
\begin{equation}\label{mhat2}
    \frac{13}{12} \eta^{3/4} \Phi_\theta(x) \geq \Phi_\theta(\hm\hy^*) \geq \frac{11}{12} \eta^{3/4} \Phi_\theta(x).
\end{equation}
It follows that 
\begin{equation}\label{marg}
  \Phi_\theta(z_i+\hm\hy^*) - \frac{1+\eta^{3/4}}{2}\Phi_\theta(x) \geq \frac{\eta^{3/4}}{3}\Phi_\theta(x) \geq
    \frac{\eta^{3/4}}{3} \Phi_\theta(x-z_i-\hm\hy^*)
\end{equation}
while
\begin{equation}\label{marg2}
  |x-z_i-\hm\hy^*| \geq |\pi_\theta x - \pi_\theta z_i| - |\pi_\theta(\hm\hy^*)| - |x-\pi_\theta x| - |z_i+\hm\hy^* - \pi_\theta(z_i+\hm\hy^*)|
    \geq \frac{|\pi_\theta x|}{3},
\end{equation}
since the subtracted terms are all much smaller than $|x|$.
Note that the relative margin for each outer link is given by the left side of \eqref{marg} divided by the $\Phi_\theta(\cdot)$ on the right, so \eqref{marg} says the relative margin is at least $\eta^{3/4}/3$.
By \eqref{Psi} we have $\eta(r)\geq (3\ep_4(r/3))^{4/3}$, which with \eqref{marg2} gives
\[
  \eta^{3/4}/3 \geq \ep_4(|\pi_\theta x|/3) \geq \ep_4(|x-z_i-\hm\hy^*|).
\]
Together with \eqref{epr} and \eqref{marg}, this yields
\begin{align}\label{outeri0}
  P\Big( \Gamma_{z_i + \hm\hy^*,x} 
    \not\subset H_{\theta,(1+\eta^{3/4}) \Phi_\theta(x)/2}^+ \Big) \leq \ep_5(|x-z_i-\hm\hy^*|) \leq \frac{p_0}{4}.
\end{align}
It then follows from the definition of $p_0$ that
\begin{align}\label{outeri}
  P\Big(T_i^{++} \leq g\big(x-(z_i+\hm\hy^*)\big) \Big) 
    &\geq P\Big( T(z_i + \hm\hy^*,x) \leq g\big(x-(z_i+\hm\hy^*)\big) \Big) \notag\\
  &\qquad - P\Big( \Gamma_{z_i + \hm\hy^*,x} 
    \not\subset H_{\theta,(1+\eta^{3/4}) \Phi_\theta(x)/2}^+ \Big) \notag\\
  &\geq  \frac{p_0}{2} - \frac{p_0}{4} = \frac{p_0}{4},
\end{align}
and symmetrically
\begin{align}\label{outeri2}
  P\Big(T_i^{--} \leq g(z_i-\hm\hy^*) \Big) \geq \frac{p_0}{4}.
\end{align}

We have $|z_i - \pi_\theta z_i| \leq \Delta_{up}(|\pi_\theta x|)$ and $|\hy^*-\pi_\theta \hy^*| \leq c_{39}$; together with \eqref{mhatbd} this gives that for large $|x|$,
\[
  |z_i - \pi_\theta z_i + \hm(\hy^*-\pi_\theta \hy^*)| \leq c_{40}\Delta_{up}(|\pi_\theta x|).
\]
Since $\theta$ is a direction of curvature, for the outer links we then have
\begin{align}\label{gcomp}
  g(x-(z_i+\hm\hy^*)) &\leq g(x-\pi_\theta x) + g(\pi_\theta x-\pi_\theta(z_i+\hm\hy^*)) \notag\\
  &\qquad + [g(\pi_\theta x-(z_i+\hm\hy^*)) - g(\pi_\theta x-\pi_\theta(z_i+\hm\hy^*))] \notag\\
  &\leq c_{41} + g\big(\pi_\theta x - \pi_\theta(z_i+\hm\hy^*) \big) 
    + c_{42}\frac{|z_i - \pi_\theta z_i + \hm(\hy^*-\pi_\theta \hy^*)|^2}{|\pi_\theta x|} \notag\\
  &\leq \Phi_\theta(x - (z_i+\hm\hy^*)) + c_{43}\sigma_{up}(|\pi_\theta x|)
\end{align}
and symmetrically
\begin{equation}\label{gcomp2}
 g(z_i-\hm\hy^*) \leq \Phi_\theta(z_i-\hm\hy^*) + c_{43}\sigma_{up}(|\pi_\theta x|).
\end{equation}
Further, in view of the arcsin bound we have for the combined central and intermediate links that 
\begin{equation}\label{yproj}
  \left| \hy^* - \frac{\eta \pi_\theta x}{2\hn} \right| \leq \sqrt{d} \quad \text{so}\quad
  |\hy^*- \pi_\theta\hy^*| \leq d \quad \text{so}\quad g(2\hm\hy^*) \leq g(2\hm\pi_\theta\hy^*) + \frac{c_{44}\hm}{|\hy^*|}.
\end{equation}

Let 
\[
  \tred{J} = \min\left\{i\leq m: T_i^0 < g(2\hn\hy^*) - c_1^{-3}c_{22}c_{25}\ep_3\sigma_{up}(|\pi_\theta x|)\right\}
\]
when such $i$ exists, otherwise $J=\infty$. Since $T_j^\pm$ are independent of $\{T_i^0, i\neq j\}$ (due to disjointness of the cylinders $\mC_i$), 
using FKG, \eqref{midi}, and \eqref{midi2} we get for $j\leq m$ that
\begin{align}\label{midi3}
  P&\Big(\max(T_J^+,T_J^-) \leq g\big( (\hm-\hn)y^* \big)\ \Big|\ J=j \Big) \notag\\
  &= P\left(\max(T_j^+,T_j^-) \leq g\big((\hm-\hn)y^* \big)\ \Big|\ 
    T_j^0 < g(2\hn\hy^*) - c_1^{-3} c_{22}c_{25}\ep_3 \hn\sigma_{up}(|\pi_\theta x|) \right) \notag\\
  &\geq P\Big(\max(T_j^+,T_j^-) \leq g\big( (\hm-\hn)y^* \big)\Big) \notag\\
  &\geq \frac{p_0^2}{16}.
\end{align}
Similarly since $H_{\theta,(1+\eta^{3/4}) \Phi_\theta(x)/2}^+$ is disjoint from the cylinders $\mC_i^{short}$, using FKG, \eqref{outeri}, and \eqref{outeri2},
\begin{align} \label{outeri3}
  P&\Big(T_J^{++} \leq g\big(x-(z_i+\hm\hy^*)\big),\,T_J^{--} \leq  g(z_i-\hm\hy^*)\ \Big|\ 
    J=j,\ \max(T_J^+,T_J^-) \leq g\big((\hm-\hn)\hy^* \big) \Big) \notag\\
  &= P\Big(T_j^{++} \leq g\big(x-(z_i+\hm\hy^*)\big),\, T_j^{--} \leq  g(z_i-\hm\hy^*) \notag\\
  &\qquad\qquad \Big|\ T_j^0 < g(2\hn\hy^*) - c_1^{-3} c_{22}c_{25}\ep_3 \hn\sigma_{up}(|\pi_\theta x|),
    \max(T_j^+,T_j^-) \leq g\big((\hm-\hn)\hy^* \big) \Big) \notag\\
  &\geq P\Big(T_j^{++} \leq g\big(x-(z_i+\hm \hy^*)\big),\, T_j^{--} \leq  g(z_i-\hm \hy^*) \Big) \notag\\
  &\geq \frac{p_0^2}{16}.
\end{align}
The analogous statement holds for $T_J^{--}$.
From \eqref{gcomp}, \eqref{gcomp2}, and \eqref{yproj}, the $g$--length of each five--link path satisfies
\begin{align}\label{gcomb}
  g&(z_i-\hm\hy^*) + 2g\big((\hm-\hn)\hy^* \big) + g(2\hn\hy^*) + g\big(x-(z_i+\hm\hy^*)\big) \notag\\
  &\leq \Phi_\theta(z_i-\hm\hy^*) + \Phi_\theta(2\hm \hy^*) + \Phi_\theta(x- (z_i+\hm\hy^*)) 
    + 2c_{43}\sigma_{up}(|\pi_\theta x|) + \frac{c_{44}\hm}{|\hy^*|} \notag\\
  &= \Phi_\theta(x) + 2c_{43}\sigma_{up}(|\pi_\theta x|) + \frac{c_{44}\hm}{|\hy^*|} \notag\\
  &\leq g(x) + 3c_{43}\sigma_{up}(|\pi_\theta x|).
\end{align}
From \eqref{regfluct3} we see that given $K>0$, provided we take $c_{25}=c_{25}(K)$ large enough in \eqref{inLam}, we have 
\begin{align}\label{Krole}
  c_1^{-3}c_{22}c_{25}\ep_3\sigma_{up}(|\pi_\theta x|) &\geq \frac12 c_1^{-3}c_{22}c_{25}\ep_3\sigma_{up}(|\pi_\theta x|) 
    + 3c_{43}\sigma_{up}(|\pi_\theta x|) \notag\\
  &\geq \frac12 c_1^{-3}c_{3}^{-1}c_{22}c_{25}\ep_3\sigma_{up}(|x|) + 3c_{43}\sigma_{up}(|\pi_\theta x|) \notag\\
  &\geq K\hat\sigma(x) + 3c_{43}\sigma_{up}(|\pi_\theta x|).
\end{align}

Applying successively Chebyshev's inequality, \eqref{Krole}, \eqref{gcomb}, \eqref{subadd} (also valid with $J$ in place of $I$), \eqref{midi3}, \eqref{outeri3}, and \eqref{onefast}, we obtain that under \eqref{inLam},
\begin{align}\label{contrad}
  \frac{1}{K^2} &\geq P\Big( T(0,x) \leq g(x) - K\hat\sigma(x) \Big) \notag\\
  &\geq P\Big( T(0,x) \leq g(x) - c_1^{-3}c_{22}c_{25}\ep_3\sigma_{up}(|\pi_\theta x|) 
    + 3c_{43}\sigma_{up}(|\pi_\theta x|) \Big) \notag\\
  &\geq P\Big( T(0,x) \leq g(z_i-\hm\hy^*) + 2g\big((\hm-\hn)\hy^* \big) \notag\\
  &\qquad\qquad + g(2\hn\hy^*) + g\big(x-(z_i+\hm\hy^*)\big) 
    - c_1^{-3}c_{22}c_{25}\ep_3\sigma_{up}(|\pi_\theta x|) \Big) \notag\\
  &\geq P\Big( T_J^{--} + T_J^- + T_J^0 + T_J^+ + T_J^{++} \leq g(z_i-\hm\hy^*) + 2g\big((\hm-\hn)\hy^* \big) \notag\\
  &\qquad\qquad  + g(2\hn\hy^*) + g\big(x-(z_i+\hm\hy^*)\big)
    - c_1^{-3}c_{22}c_{25}\ep_3\sigma_{up}(|\pi_\theta x|) \Big) \notag\\
  &\geq \frac{p_0^4}{256} P(J\leq m) \notag\\
  &\geq \frac{p_0^4}{512}.
\end{align}
But we can always increase $|x|$ (keeping $(\theta,x)\in\mG$) so that \eqref{inLam} holds, and if $K$ is large then \eqref{contrad} cannot then be true, as the left side is smaller than the right side.  Thus we have a contradiction, which proves Theorem \ref{main}.

\section{Proof of Theorem \ref{cylpass}}
Suppose $\hat\sigma$ is magnitude--based, $D$ is magnitude--based, and $\sigma_{mag}$ has growth exponent $\chi=0$. As in the last section, we let $c_1$ be as in \eqref{fmag} for $f=\hat\sigma,f_{mag}=\sigma_{mag}$; by increasing $c_1$ we may assume \eqref{fmag} also holds for $f=D,f_{mag}=D_{mag}$.
As noted after \eqref{ls1}, $\hat\sigma$ is bounded away from 0, and hence so is $\sigma_{mag}$; therefore we may assume $\sigma_{mag}\geq 1$.  This and \eqref{reggr2} show that \eqref{growctrl} is valid for $\sigma_{mag}$, so applying Lemma \ref{upperapp} to $f=\sigma_{mag}$ shows that $\sigma_{up}$ satisfying \eqref{upperf} exists. In the present proof we do not make use of the last two properties in \eqref{upperf} for $\sigma_{up}$; this means that if $\sigma_{mag}$ has \emph{regular} growth exponent $\chi=0$, we can take $\sigma_{up}=\sigma_{mag}$, which satisfies all the other properties in Lemma \ref{upperapp}.

Let us now give a sketch of the proof.  We work again by contradiction: suppose the subsequential moderate gap property (iv)(b) fails,
the controlled wandering property (v) holds, and the local fluctuation property (vii) fails.  Fixing a large integer $k$, we can find arbitrarily large $x$ for which $D_{mag}(|x|)/\sigma_{mag}(|x|)$ is large (at least of order $k^a$ for a particular $a$.)  We consider paths from 0 to $kx$ inside a natural cylinder $\mC^{(k)}$ of radius a large multiple of $\Delta(k|x|)$, where the geodesic is likely to lie, by (v). We subdivide $\mC^{(k)}$ into $k$ equal cylinders $Q_i$ of length about $|x|$ (see Figure \ref{Thm1.5_Fig1}) and consider the disc-to-disc passage times $T_{d2d}(Q_i)$, which satisfy $T(0,x) \geq \sum_{i=1}^k T_{d2d}(Q_i)$.  By definition of $D_{mag}$ there exists $c$ such that $ET(p,q) \geq g(x) + cD_{mag}(|x|)$ for all end pairs $(p,q)$ of $Q_i$, and we want to show that there exists $c' < c$ for which $T_{d2d}(Q_i)$ is unlikely to be below $g(x) + c'D_{mag}(|x|)$. The cylinder $Q_i$ is ``too fat'' for what we want to do, so we form a collection of a bounded number (order $k^\rho$ for some $\rho$) of thinner cylinders $\hat\mC(u,v)$ whose end discs cover the end disks of $Q_i$, so that $T_{d2d}(Q_i) \geq \min_{u,v} T_{d2d}(\hat\mC(u,v))$.  We enlarge each $\hat\mC(u,v)$ slightly to get a natural cylinder $\hat\mC_{nat}(u,v)$; see Figure \ref{Thm1.5_Fig2}. In Lemma \ref{sigcylbd} we show that the standard deviation of $T_{d2d}(\hat\mC_{nat}(u,v))$ is at most of order $(\sigma_{mag}(|x|)D_{mag}(|x|))^{1/2}$, which is much smaller than $D_{mag}(|x|)$. With the help of this standard deviation bound we get the desired upper bound \eqref{TQibound} on the lower tail of $T_{d2d}(Q_i) - [g(x) + c'D_{mag}(|x|)]$, which after summing over $i$ lets us show that with probability bounded away from 0 we have $T(0,kx) \geq g(kx) + c''kD_{mag}(|x|)$; see \eqref{assemb}.  But this means that for some $\ep$ we have $D_{mag}(k|x|) \geq \ep kD_{mag}(|x|)$, for (at least some) arbitarily large values of $|x|$.  One can iterate this to obtain something like $D_{mag}(k^n|x|) \geq (\ep k)^nD_{mag}(|x|)$ for general $n$, showing that $\limsup_{r\to\infty} (\log D_{mag}(r))/\log r \geq 1$, which contrdicts \eqref{ls1}.

We turn now to the details of the proof. As mentioned we suppose the uniform moderate gap property (iv)(b) fails,
the controlled wandering property (v) holds, and (vii) fails: for some \tred{$\ep_+$}, for every $0<\ep_-<\ep_+$, there exists $\tred{K_1}=K_1(\ep_-,\ep_+)$ for which
\begin{equation}\label{nonloc}
  \min_{(p,q)\in\mE(\mC_{x,\ep})} ET(p,q\mid S(\mC_{x,\ep})) - ET_{d2d}\big(\mC_{x,\ep} \big) \leq K_1\hat\sigma(0,x\mid S(\mC_{x,\ep})) \ \ \text{for all $x$ and all }
    \ep_-\leq \ep\leq \ep_+.
\end{equation}
Let $\frac12<\tred{\xi_1<\xi_2<\xi_3}<1$ with 
\begin{equation}\label{xicond}
  2(d+1)(\xi_3-\xi_2)< 1-\xi_1,
\end{equation}
let
\[
  \tred{D_{mag}^+} = D_{mag} + 2\sigma_{up},
\]
let \tred{$\theta$} be a direction of curvature, and recall $\mC_{x,\ep}$ from (vii). 
By Lemma \ref{upperapp} and failure of (iv)(b), we have for some $c_{45}$
\begin{equation}\label{nomod}
  \lim_{r\to\infty} \frac{D_{mag}^+(r)}{\sigma_{mag}(r)} = \infty,\qquad \liminf_{r\to\infty}\frac{\sigma_{up}(r)}{\sigma_{mag}(r)} \leq c_{45}.
\end{equation}
Fixing \tred{$k$} large it follows that we can find arbitrarily large $r$ satisfying
\begin{equation}\label{core}
   \frac{D_{mag}^+(r)}{\sigma_{up}(r)} \geq c_1^2c_{3} (k^{2(d+1)\xi_3}+2)
\end{equation}
and then in view of \eqref{reggr}, \eqref{regfluct3}, and \eqref{thick} we can find \tred{$x$} with $||x|-r|<\sqrt{d}$ and 
\begin{equation}\label{core2}
  (\theta,x)\in\mG,\quad \frac{D_{mag}^+(|x|)}{\sigma_{up}(|x|)} \geq k^{2(d+1)\xi_2}+2, \ \text{or equivalently}\ \ 
    \frac{D_{mag}(|x|)}{\sigma_{up}(|x|)} \geq k^{2(d+1)\xi_2}.
\end{equation}

Note if instead of (iv)(b) failing we only assume (iv) fails, but $\sigma_{mag}$ has \emph{regular} growth exponent $\chi=0$ (so (ii)(b) holds and, as noted above, we can take $\sigma_{up}=\sigma_{mag}$), then we only have ``lim sup'' on the left in \eqref{nomod}, but  since $\sigma_{up}=\sigma_{mag}$ we still have arbitrarily large $r$ with \eqref{core} holding, so \eqref{core2} applies; the proof is otherwise the same.

Recall $\ep_0$ from the definition of direction of curvature and of $\mU$.

\begin{lemma}\label{sigcylbd}
Under the hypotheses of Theorem \ref{cylpass}, there exists \tred{$K_3$} such that for all $\ep<1,\,M\geq 1$, and all sufficiently large $y\in\ZZ^d$ satisfying both 
\begin{equation}\label{nearth}
  \left|\theta-\frac{y}{|y|} \right| < \frac{\ep_0}{2}\quad\text{for some } \theta\in\mU
\end{equation}
and 
\begin{equation}\label{okcurv}
  g(b-a) \leq g(y) + MD_{mag}^+(|y|) \quad\text{for all } (a,b)\in\mE\big(\mC_{nat}\big(0,y,\ep\Delta_{up}(|y|)\big)\big), 
\end{equation}
we have 
\begin{equation}\label{sigmas}
  \hat\sigma_{d2d}\left( \mC_{nat}\big(0,y,\ep\Delta_{up}(|y|)\big) \right) \leq K_3\big(M\sigma_{up}(|y|)D_{mag}^+(|y|)\big)^{1/2}.
\end{equation}
\end{lemma}

Note that \eqref{okcurv} simply says that the $g$--distance between an arbitrary end pair $(a,b)$ is not too much greater than the $g$--distance for the end pair $(0,y)$ which are the centers of the cylinder ends; necessarily $g(b-a)\geq g(y)$.
Provided $\ep$ is small, since $D_{mag}^+\geq \sigma_{up}$, \eqref{okcurv} is satisfied with $M=1$ whenever $y/|y|$ is a direction of curvature, by \eqref{curvg}. The essential aspect of\eqref{sigmas} is that if $\sigma_{up}(r)\ll D_{mag}(r)$ then the right (and hence also left) side is $\ll D_{mag}(|y|)$. This means fluctuations of $T_{d2d}(\cdot)$ are typically not large enough to overcome a gap of order $D_{mag}(\cdot)$ for the relevant cylinders; we use this in \eqref{t2bound} below.

\begin{proof}[Proof of Lemma \ref{sigcylbd}]
Fix $0<\ep<1$. We consider separately the upward and downward deviations from the mean contributing to $\hat\sigma_{d2d}\left( \mC_{nat}\big(0,y,\ep\Delta_{up}(|y|)\big) \right)$.  To deal with upward deviations, observe first that, writing \tred{$\mF_y$} for $\mC_{nat}\big(0,y,\ep\Delta_{up}(|y|)\big)$, for all $(p,q)\in \mE(\mF_y)$,
\[
  T_{d2d}\left( \mF_y \right) \leq T(p,q\mid S(\mF_y)).
\]
The angle between $q-p$ and $y$ is $O(\Delta(|y|)/|y|)$ so in view of \eqref{nearth} provided $|y|$ is large we have $S(\mF_y) \in \mN(p,q)$.
Therefore taking $K_1$ and a minimizing $(p,q)\in \mE(\mF_y)$ from \eqref{nonloc}, and recalling \eqref{fmag}, \eqref{reggr}, and \eqref{slabvar}, we have that for all $r\geq 2K_1c_1^3\sigma_{mag}(|y|)$, 
\begin{align}\label{cylpq}
  P\Big( T_{d2d}\left( \mF_y \right) \geq ET_{d2d}\left( \mF_y \right) + r \Big)
    &\leq P\Big( T(p,q\mid S(\mF_y)) \geq ET(p,q\mid S(\mF_y)) - K_1c_1^3\sigma_{mag}(|y|) + r \Big) \notag\\
  &\leq P\left( T(p,q\mid S(\mF_y)) \geq ET(p,q\mid S(\mF_y)) + \frac r2 \right).
\end{align}
Then taking $r=\sqrt{t}$ and integrating this over $t\in [4K_1^2c_1^6\sigma_{mag}(|y|)^2,\infty)$ yields, using again \eqref{slabvar} and enlarging $K_1$ if necessary in \eqref{nonloc}
\begin{align}\label{upvar}
  E\Big( T_{d2d}&\left( \mF_y \right) - ET_{d2d}\left( \mF_y \right) \Big)_+^2 = \int_0^\infty 
    P\Big( T_{d2d}\left( \mF_y \right) - ET_{d2d}\left( \mF_y \right) \geq \sqrt{t} \Big)\,dt
    \notag\\
  &\leq 4K_1^2c_1^6\sigma_{mag}(|y|)^2 + 4E\Big( T(p,q\mid S(\mF_y)) - ET(p,q\mid S(\mF_y)) \Big)_+^2 \leq 8K_1^2c_1^6
    \sigma_{mag}(|y|)^2.
\end{align}
Turning to downward deviations, let
\[
  \tred{G} = T_{d2d}\left( \mF_y \right) - ET_{d2d}\left( \mF_y \right).
\]
We divide into larger deviations ($G^-\gg D_{mag}^+(|y|)$) and smaller deviations; fix $C$ to be specified and consider first smaller deviations, meaning $G^-<  CD_{mag}^+(|y|)$.  We have by \eqref{upvar}
\begin{equation}\label{moments}
  EG^- = EG^+ \leq 3K_1c_1^3\sigma_{mag}(|y|)
\end{equation}
so 
\begin{equation}\label{downvar1}
  E( (G^-)^2\,1_{\{ G^- < CD_{mag}^+(|y|) \}} ) \leq CD_{mag}^+(|y|)EG^- \leq 3K_1c_1^3C\sigma_{mag}(|y|)D_{mag}^+(|y|).
\end{equation}
Turning to larger deviations, meaning $G^-\geq r\geq CD_{mag}^+(|y|)$, if $T_{d2d}\left( \mF_y \right) \leq ET_{d2d}\left( \mF_y \right) - r$ then there exists a random $\tred{(X,Y)}\in\mE(\mF_y)$ with $T(X,Y\mid\mF_y) \leq ET_{d2d}( \mF_y ) - r$; if there are multiple such $(X,Y)$ we assume a particular one has been chosen by some arbitrary algorithm.  As in \eqref{cylpq}, letting $\tred{\alpha}=y/|y|$ and $\tred{\mL_{u,v}}=S_\alpha(-\Phi_\alpha(u),\Phi_\alpha(v))$, from \eqref{fmag}, \eqref{reggr}, and \eqref{slabvar} we have for all $(a,b)\in\mE(\mF_y)$
\begin{align}\label{slabeffect}
  ET(-y,a) \geq E\big( T(-y,a) \mid \mL_{-y,a} \big) - c_{46}, \quad &ET(b,2y) \geq E\big( T(b,2y) \mid \mL_{b,2y} \big) - c_{46}, \notag\\
  \hat\sigma(-y,a \mid  \mL_{-y,a}) \leq c_1^3\sigma_{mag}(|y|), \quad &\hat\sigma(b,2y \mid \mL_{b,2y}) \leq c_1^3\sigma_{mag}(|y|).
\end{align}
See Figure \ref{Lemma3.1_Fig}; as in \cite{BSS16} we now exploit the fact that
\[
  T(-y,2y) \leq T(-y,X)+T(X,Y)+T(Y,2y),
\]
which means that a fast value for $T(X,Y)$ for random $(X,Y)$ likely produces a fast value of $T(-y,2y)$ which involves nonrandom points.
\begin{figure}
\includegraphics[height=4cm]{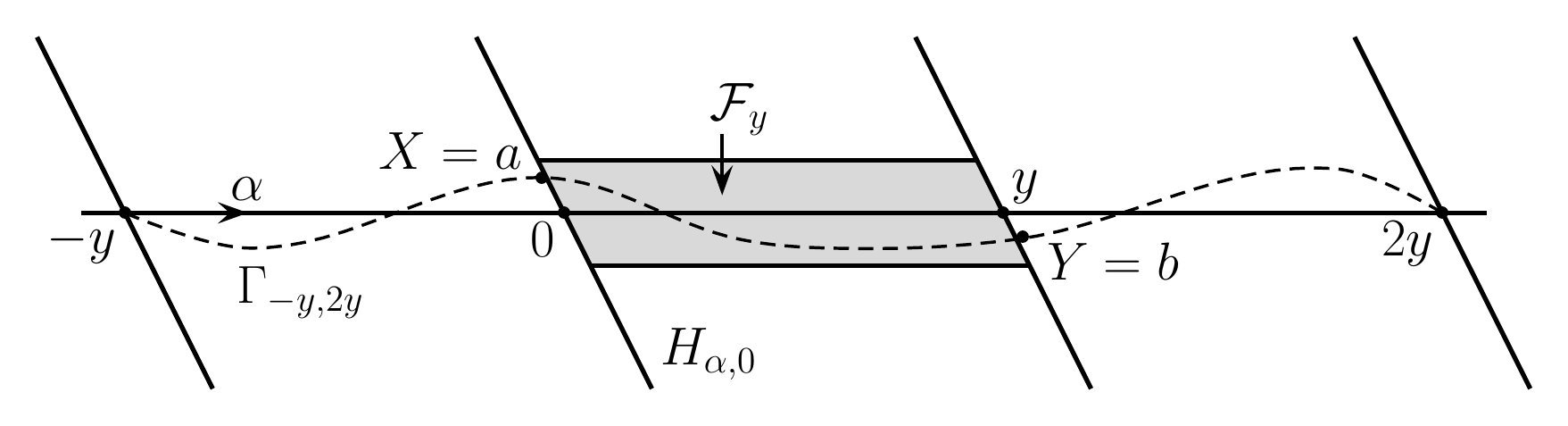}
\caption{ Illustration for the proof of Lemma \ref{sigcylbd}. If there is a fast path from some $X=a$ to $Y=b$ then with probability bounded away from 0 it can be extended to a fast path from $-y$ to $2y$. }
\label{Lemma3.1_Fig}
\end{figure}
Specifically, since events in disjoint slabs are independent, with \eqref{slabeffect} we have from Chebyshev's inequality, after enlarging $c_1$ if necessary,
\begin{align}\label{Cheby}
  P\Big( &T(-y,2y) \leq ET(-y,a) + ET(0,y) -r + ET(b,2y) + 6c_1^3\sigma_{mag}(|y|)\ \big|\ X=a,Y=b \Big) \notag\\
  &\geq P\Big( T(-y,2y) \leq ET(-y,a \mid S(\mL_{-y,a} )) + ET_{d2d}\left( \mF_y \right) -r 
    + ET(b,2y \mid S(\mL_{b,2y}) ) \notag\\
  &\hskip 5cm + 2\hat\sigma(-y,a \mid  S(\mL_{-y,a})) + 2\hat\sigma(b,2y \mid S(\mL_{b,2y})) \Big|\ X=a,Y=b \Big) \notag\\
  &\geq P\Big( T(-y,a \mid S(\mL_{-y,a})) \leq ET(-y,a \mid S(\mL_{-y,a} )) + 2\hat\sigma(-y,a \mid  S(\mL_{-y,a}))\Big) \notag\\
  &\qquad\cdot P\Big(T(b,2y \mid S(\mL_{b,2y})) \leq ET(b,2y \mid S(\mL_{b,2y} )) + 2\hat\sigma(b,2y \mid S(\mL_{b,2y})) \Big) \notag\\
  &\geq \frac{1}{16},
\end{align}
and it follows using \eqref{okcurv} that
\begin{align}\label{Gminus}
  P&(G^-\geq r) \notag\\
  &= P\Big( T_{d2d}\left( \mF_y \right) \leq ET_{d2d}\left( \mF_y \right) - r \Big) \notag\\
  &\leq 16 \max_{(a,b)\in \mE(\mF_y)} P\Big( T(-y,2y) \leq ET(-y,a) + ET(0,y) + ET(b,2y) - r + 6c_1^3\sigma_{mag}(|y|) \Big) \notag\\
  &\leq 16 \max_{(a,b)\in \mE(\mF_y)} P\Big( T(-y,2y) \leq g(a+y) + g(y) + g(2y-b) - r + c_{47}D_{mag}^+(|y|) \Big) \notag\\
  &\leq 16 P\Big( T(-y,2y) \leq 3g(y) - r + (c_{47}+3M)D_{mag}^+(|y|) \Big) \notag\\
  &\leq 16 P\left( T(-y,2y) \leq ET(-y,2y) - \frac r2 \right), 
\end{align}
where we take $C=2(c_{47}+3M)$ in our assumption $r\geq CD_{mag}^+(|y|)$ to yield the last inequality.  Here we have cheated slightly in applying \eqref{okcurv} to bound $g(a+y)$, because $(-y,a)$ is not necessarily an end pair of $\mE(\mC_{nat}(-y,0,\ep\Delta_{up}(|y|)))$, as $a$ may be slightly on the wrong side of the end hyperplane. But there is a site $z$ within a bounded distance of $a$ on the ``correct'' side of that hyperplane for which $(-y,z)$ is an end pair. The same holds for bounding $g(2y-b)$, so the third inequality in \eqref{Gminus} is valid.
Taking $r=\sqrt{t}$ and integrating \eqref{Gminus} over $t\in [C^2D_{mag}^+(|y|)^2,\infty)$ yields that for some $c_{49}$,
\begin{align}\label{downvar}
  E((G^-)^2\,1_{\{G^-\geq CD_{mag}^+(|y|)\}}) 
    &\leq C^2D_{mag}^+(|y|)^2 P\Big( ET(-y,2y) - T(-y,2y) \geq \frac12 CD_{mag}^+(|y|),\infty) \Big) \notag\\
  &\qquad + 64E\Big( T(-y,2y) - ET(-y,2y) \Big)_-^2 \notag\\
  &\leq 68\var\big( T(-y,2y) \big) \leq 68c_1^2\sigma_{mag}(3|y|)^2 \leq c_{49}\sigma_{mag}(|y|)^2.
\end{align}
Together with \eqref{upvar}, \eqref{downvar1}, and \eqref{reggr} this shows that 
\[
  EG^2 \leq 8K_1^2c_1^6\sigma_{mag}(|y|)^2 + 3K_1c_1^3C\sigma_{mag}(|y|)D_{mag}^+(|y|) + c_{49}\sigma_{mag}(|y|)^2,
\]
which proves the lemma.
\end{proof}

\begin{figure}
\includegraphics[height=4cm]{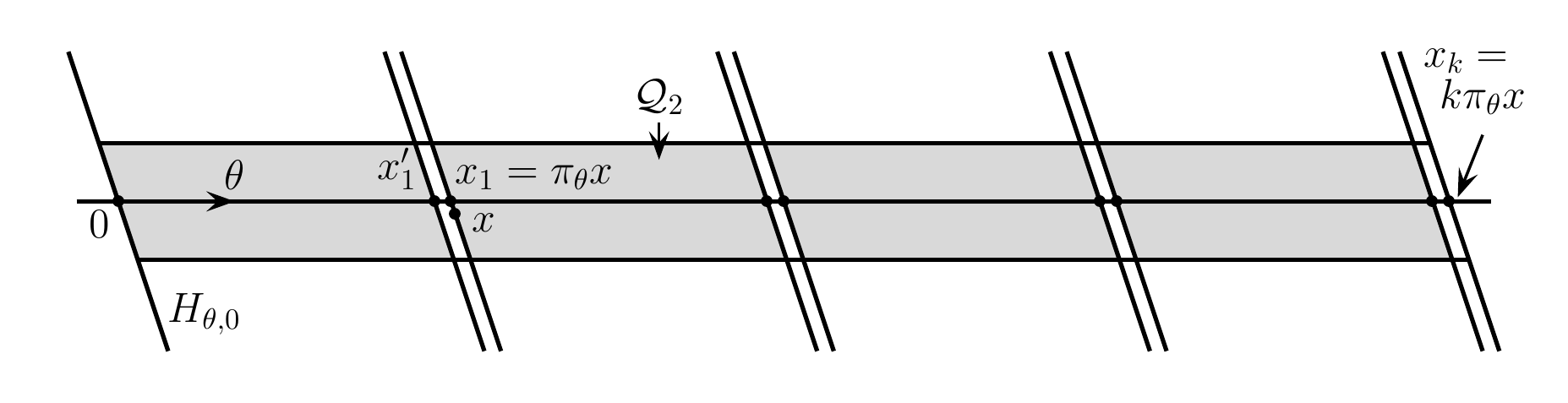}
\caption{ The cylinder $\mC^{(k)}$ and sub--cylinders $\mQ_i$ (shaded), with $k=4$. }
\label{Thm1.5_Fig1}
\end{figure}

Let
\begin{equation}\label{Ck}
  \tred{\mC^{(k)}} = \mC_{nat}(0,k\pi_\theta x,K_2\Delta_{up}(k|x|)\big), \quad \tred{x_i} = i\pi_\theta x, \quad \tred{x_i'} = (i-\tfrac{c_{48}}{|x|})\pi_\theta x,
\end{equation}
with \tred{$K_2$} (large) to be specified, and cut $k$ equal cylinders from $\mC^{(k)}$ (which has axis in direction $\theta$):
\[
   \tred{\mQ_i} = \mC_{nat}(x_{i-1},x_i',K_2\Delta_{up}(k|x|)),\quad i\leq k;
 \]
these are separated by the hyperplanes $H_{\theta,i\Phi_\theta(x)}$.  See Figure \ref{Thm1.5_Fig1}. In \eqref{Ck} $c_{48}$ is chosen large enough so that no edges intersect both $\mQ_{i-1}$ and $\mQ_i$, which ensures that slab passage times are independent for these two slabs.
The geodesic $\Gamma_{0,kx}$ must contain a ``slab crossing segment,'' by which we mean a path crossing $S(\mC^{(k)})$ with all vertices in $S(\mC^{(k)})$ except possibly the endpoints; we write \tred{$\Gamma_{0,kx}^{scs}$} for the first such segment.
Then 
\begin{equation}\label{superadd}
  \Gamma_{0,kx}^{scs} \subset \mC^{(k)} \implies T(0,kx) \geq T_{d2d}( \mC^{(k)} ) \geq \sum_{i=1}^k T_{d2d}(\mQ_i).
\end{equation}
Our aim is  to use this to show
\[
  D_{mag}^+(k|x|) \geq ckD_{mag}^+(|x|)
\]
for some $c>0$.  We can then iterate this to show $D_{mag}^+(r)$ grows almost linearly in $r$, which contradicts \eqref{ls1}.

Fix $\tred{\ep_6}>0$ (small) to be specified.  For each $i\leq k$, we select a finite set \tred{$Z_i$} of ``coarse--grain'' points in $\mC^{(k)}\cap H_{\theta,g(x_i)}$ satisfying (assuming $k$ large)
\[
  \min_{z\in Z_i} |u-z| \leq \frac{\ep_6}{4}\Delta_{up}(|x|) \ \ \text{for all } u \in \mC^{(k)} \cap H_{\theta,i\Phi_\theta(x)}, 
\]
\begin{equation}\label{Zsize}
    |Z_i| \leq c_{49}\left( \frac{K_2\Delta_{up}(k|x|)}{\ep_6\Delta_{up}(|x|)} \right)^{d-1} 
      \leq c_{50}\left( \frac{K_2}{\ep_6} \right)^{d-1} k^{(d-1)\xi_1},
\end{equation}
and we select a similar set \tred{$Z_i'$} in $\mC^{(k)}\cap H_{\theta,g(x_i')}$.  Now fix $i\leq k$.
Fixing $\tred{u}\in Z_{i-1},\tred{v}\in Z_i'$, and corresponding direction $\tred{\phi} = (v-u)/|v-u|$, we have a corresponding cylinder \tred{$\hat\mC(u,v)$} with axis $\Pi_{uv}^\infty$, radius $2\ep_6\Delta_{up}(|x|)$, and end hyperplanes $H_{\theta,g(x_{i-1})}$ and $H_{\theta,g(x_i')}$.
In general $\hat\mC(u,v)$ 
is not a natural cylinder, as the end hyperplanes are parallel to $\mH_\theta$ rather than $\mH_\phi$.  However $\hat\mC(u,v)$ is contained in a minimal natural cylinder with the same axis and radius, 
which we denote \tred{$\hat\mC_{nat}(u,v)$}; let \tred{$\hat u,\hat v$} be the points where the axis $\Pi_{uv}^\infty$ intersects the ends of $\hat\mC_{nat}(u,v)$.  See Figure \ref{Thm1.5_Fig2}.

\begin{figure}
\includegraphics[height=5cm]{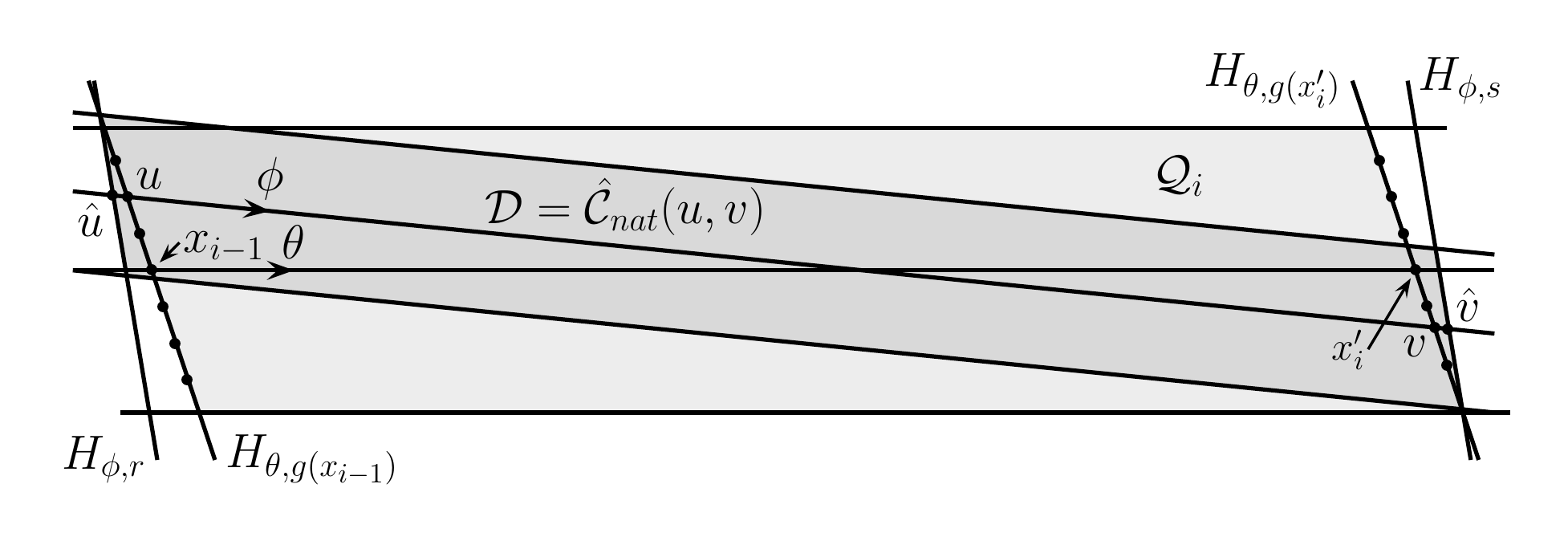}
\caption{ The cylinders $\mQ_i$ (light gray), $\mD=\hat\mC_{nat}(u,v)$ (darker gray), and $\hat\mC(u,v)$, which is the part of $\mD$ between $H_{\theta,g(x_{i-1})}$ and $H_{\theta,g(x_i')}$. $\mD$ is the smallest natural cylinder in direction $\phi$ containing $\hat\mC(u,v)$. The dots in $H_{\theta,g(x_{i-1})}$ and $H_{\theta,g(x_i')}$ are points of $Z_{i-1}$ and $Z_i'$, respectively, including $u$ and $v$. See Figure \ref{Thm1.5_Fig3} for more detail. }
\label{Thm1.5_Fig2}
\end{figure}

The angle between $\phi$ and $\theta$ is at most
\[
  c_{51}\frac{K_2\Delta_{up}(k|x|)}{|x|},
\]
with $K_2$ from \eqref{Ck}; since $\theta$ is a direction of subcurvature this means the angle between $\mH_\phi$ and $\mH_\theta$ is at most 
\[
  c_{52}\frac{K_2\Delta_{up}(k|x|)}{|x|}.
\]
For a point $w$ outside the slab $S(\mQ_i)$, let \tred{$\tilde w$} denote the longitudinal $\phi$--projection of $w$ into the closer end hyperplane of $\mQ_i$.
From the above angle bound and routine geometry, when $w$ is an end vertex of $\hat\mC_{nat}(u,v)$ we have 
\begin{equation}\label{uvshatu}
  |w-\tilde w| \leq c_{53} \frac{K_2\Delta_{up}(k|x|)}{|x|} \ep_6\Delta_{up}(|x|) \leq c_{53}\ep_6 K_2 k^{\xi_1}\sigma_{up}(|x|).
\end{equation}
See Figure \ref{Thm1.5_Fig3}.

\begin{figure}
\includegraphics[height=5cm]{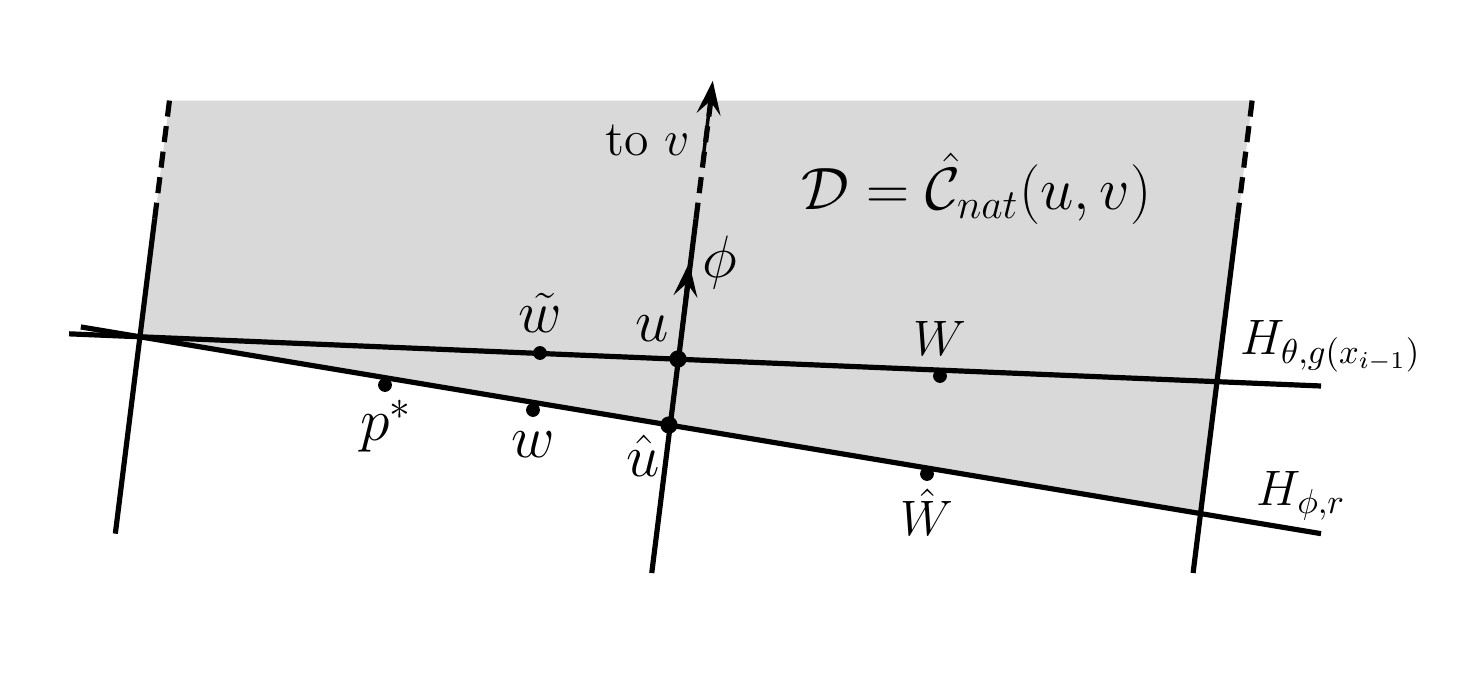}
\caption{ Closeup of the left end of Figure \ref{Thm1.5_Fig2}, rotated $90^\circ$. }
\label{Thm1.5_Fig3}
\end{figure}

Observe that
\begin{equation}\label{epchoice}
  \hat\mC_{nat}(u,v) = \mC_{nat}\big(\hat u,\hat v,\ep\Delta_{up}(|\hat v - \hat u|)\big) \text{ for } \ep 
    = 2\ep_6 \frac{\Delta_{up}(|x|)}{\Delta_{up}(|\hat u - \hat v|)}.
\end{equation}
Since  $|x|/|\hat u - \hat v|$ is near 1 for large $|x|$, the ratio on the right is bounded away from 0 and $\infty$, so recalling $\ep_+$ from \eqref{nonloc} we may choose $\ep_6$ (in \eqref{Zsize}) and another constant $\tred{\ep_7}$ so that the $\ep$ in \eqref{epchoice} always lies in $[\ep_7,\ep_+\wedge 1]$.
Writing \tred{$\mD$} for $\hat\mC_{nat}(\hat u,\hat v)$, \eqref{nonloc} then applies and we have 
\begin{align}\label{cylpt}
  ET_{d2d}\big(\mD \big) \geq \min_{(p,q)\in\mE(\mD)} ET(p,q\mid S(\mD)) - K_1\hat\sigma(\hat u,\hat v\mid S(\mD)).
\end{align}
For each $p,q\in\ZZ^d$ we can specify some path \tred{$\gamma_{pq}$} from $p$ to $q$ of length $|q-p|_1$.  
Given $T_{d2d}\big(\hat\mC(u,v) \big) \leq t$ for some $t$, there exists a random $\tred{(W,Z)}\in\mE(\hat\mC(u,v))$ with $T(W,Z)\leq t$, selected by some arbitrary algorithm if there is more than one such pair.  Let  $w_\phi,z_\phi$ be the longitudinal $\phi$--projections of $W,Z$, each into the closest end hyperplane of $\mD$, and let \tred{($\hat W, \hat Z$)} be an end pair of $\mD$ with $|w_\phi-\hat W|\leq d,|z_\phi-\hat Z|\leq d$.  See Figure \ref{Thm1.5_Fig3}. Using \eqref{uvshatu} we have
\begin{equation}\label{XtohatX}
  \max\big(|W-\hat W|_1,|Z-\hat Z|_1\big) \leq c_{54}\ep_6 K_2 k^{\xi_1}\sigma_{up}(|x|).
\end{equation}
By \eqref{XtohatX}, independence, and Chebyshev's inequality, for all $t>0$,
\begin{align}\label{addon}
  P&\Big( T_{d2d}\big(\mD \big) \leq t+4c_{54}\ep_6 K_2 k^{\xi_1}E(\tau_e)\sigma_{up}(|x|)\ \Big|\ T_{d2d}\big(\hat\mC(u,v) \big) \leq t \Big) \notag\\
  &\geq P\Big( \max\big(T(\gamma_{W,\hat W}),T(\gamma_{Z,\hat Z})\big) \leq 2c_{54}\ep_6 K_2 k^{\xi_1}E(\tau_e)\sigma_{up}(|x|)
    \ \Big|\ T_{d2d}\big(\hat\mC(u,v) \big) \leq t \Big) \notag\\
  &\geq P\Big( T(\gamma_{W,\hat W}) \leq 2E(\tau_e)|\gamma_{W,\hat W}|,\, T(\gamma_{Z,\hat Z}) \leq 2E(\tau_e)|\gamma_{Z,\hat Z}|
    \ \Big|\ T_{d2d}\big(\hat\mC(u,v) \big) \leq t \Big) \notag\\
  &\geq \frac14.
\end{align}

Let \tred{$(p^*,q^*)$} be the pair in $\mE(\mD)$ achieving the minimum in \eqref{cylpt}.  
Observe that since $\theta$ is a direction of curvature and $|x-\pi_\theta x|\leq d$, we have
\begin{equation}\label{grel}
  g(q^*-p^*) \geq g(\hat v - \hat u) \geq g(v-u) \geq g(x_i'-x_{i-1}) \geq g\left(\left( 1 - \frac{c_{48}}{|x|}\right) \pi_\theta x\right) \geq g(x) - c_{55}
\end{equation}
and hence
\begin{align}\label{ETpq}
  ET(p^*,q^*\mid S(\mD)) &\geq ET(p^*,q^*) \notag\\
  &\geq g(q^*-p^*) + c_1^{-1}D_{mag}(|q^*-p^*|) \notag\\
  &\geq g(x) + c_{56}D_{mag}(|x|).
\end{align}
We would like to apply Lemma \ref{sigcylbd} to the cylinder $\mD = \hat\mC_{nat}\big(\hat u,\hat v\big)$ from \eqref{epchoice}, so we need to check the condition \eqref{okcurv}. Let $(a,b)\in\mE(\mD)$; then since $u,v\in \mQ_i$, $\mD$ is contained in a slightly--fattened $\mQ_i$, so that
\[
  \tilde a,\tilde b \in \mC_{nat}\big(x_{i-1},x_i',2K_2\Delta_{up}(k|x|)\big);
\]
note this cylinder has the same axis as $\mQ_i$, in direction $\theta$, with twice the diameter.
Therefore from \eqref{curvg}, provided $k$ is large,
\begin{align}
  g(\tilde b-\tilde a) &\leq g(x_i'-x_{i-1}) + c_{5}\frac{4K_2^2\Delta_{up}(k|x|)^2}{|y_\theta| |x|} \notag\\
  &\leq g(\hat v - \hat u) + c_{57}k^{2\xi_1}\sigma_{up}(|x|) \notag\\
  &\leq g(\hat v - \hat u) + c_{58}k^{2\xi_1}\sigma_{up}(|\hat v - \hat u|),
\end{align}
which with \eqref{uvshatu} yields
\[
  g(b-a) \leq g(\hat v - \hat u) + c_{59}k^{2\xi_1}\sigma_{up}(|\hat v - \hat u|),
\]
proving \eqref{okcurv} with $M=c_{59}k^{2\xi_1}$.

Provided $k$ and then $|x|$ are taken sufficiently large,  we now have for $t\geq 1$ that
\begin{align}\label{t2bound}
  4t^{-2} &\geq 4P\Big( T_{d2d}\left( \mD \right) \leq ET_{d2d}\left( \mD \right) 
    - t\hat\sigma_{d2d}\left( \mD \right) \Big) \notag\\
  &\geq  P\Big( T_{d2d}\left( \hat\mC(u,v) \right) \leq ET_{d2d}\left( \mD \right) - 4c_{54}\ep_6 K_2 k^{\xi_1}E(\tau_e)\sigma_{up}(|x|) 
    - t\hat\sigma_{d2d}(\mD) \Big)  \notag\\
  &\geq  P\Big( T_{d2d}(\hat\mC(u,v)) \leq ET(p^*,q^*\mid S(\mD)) - K_1\hat\sigma( \hat u,\hat v\mid S(\mD)) 
    - 4c_{54}\ep_6 K_2 k^{\xi_1}E(\tau_e)\sigma_{up}(|x|) \notag\\
  &\hskip 2.5in -t c_{60}k^{\xi_1}\big(\sigma_{up}(|\hat v - \hat u|)D_{mag}^+(|\hat v - \hat u|)\big)^{1/2} \Big) \notag\\
  &\geq  P\Big( T_{d2d}(\hat\mC(u,v)) \leq g(x) + c_{56}D_{mag}(|x|) 
    - 5c_{54}\ep_6 K_2 k^{\xi_1}E(\tau_e)\sigma_{up}(|x|) \notag\\
  &\hskip 2.5in -t c_{60}k^{\xi_1}\big(\sigma_{up}(|\hat v - \hat u|)D_{mag}^+(|\hat v - \hat u|)\big)^{1/2} \Big) \notag\\
  &\geq  P\left( T_{d2d}(\hat\mC(u,v)) \leq g(x) + \frac{c_{56}}{2}D_{mag}(|x|) 
    - t c_{60}k^{\xi_1}\big(\sigma_{up}(|\hat v - \hat u|)D_{mag}^+(|\hat v - \hat u|)\big)^{1/2} \Big) \right),
\end{align}
where the second inequality follows from \eqref{addon}, the third from \eqref{cylpt} and Lemma \ref{sigcylbd} (with $M=c_{59}k^{2\xi_1}$ as above), the fourth from \eqref{slabapp} and \eqref{ETpq}, and the last from \eqref{core2}.
For each $(p,q)\in\mE(\mQ_i)$, one of the cylinders $\hat\mC(u,v)$ has $(p,q) \in \mE(\hat\mC(u,v))$, with $u\in Z_{i-1},v\in Z_i'$.  It follows from \eqref{reggr}, \eqref{Zsize}, and \eqref{t2bound} that for some $c_{61}$,
\begin{align}\label{TQibound}
  P&\Big( T_{d2d}(\mQ_i) \leq g(x) + \frac{c_{56}}{2}D_{mag}(|x|) - t c_{60}k^{\xi_1}\big(\sigma_{up}(|x|)D_{mag}(|x|)\big)^{1/2} \Big) \notag\\
  &\leq P\Big( \min_{u\in Z_{i-1},v\in Z_i} T_{d2d}(\hat\mC(u,v)) \leq g(x) 
    + \frac{c_{56}}{2}D_{mag}(|x|) - t c_{61}k^{\xi_1}\big(\sigma_{up}(|\hat v - \hat u|)D_{mag}^+(|\hat v - \hat u|)\big)^{1/2} \Big) \notag\\
  &\leq 4|Z_{i-1}||Z_i'| t^{-2} \notag\\
  &\leq 4c_{50}^2\left( \frac{K_2}{\ep_6} \right)^{2(d-1)} k^{2(d-1)\xi_1} t^{-2}.
\end{align}
Equivalently, letting
\[
  \tred{A_k} = 2c_{50}\left( \frac{K_2}{\ep_6} \right)^{d-1} k^{(d-1)\xi_1}, \quad
  \tred{Y_i} = -\frac{ T_{d2d}(\mQ_i) - g(x) 
    - \frac{c_{56}}{2}D_{mag}(|x|) }{c_{60}k^{\xi_1}A_k\big(\sigma_{up}(|x|)D_{mag}(|x|)\big)^{1/2} },
\]
we have (using $s = t/A_k$)
\[
  P(Y_i\geq s) \leq s^{-2}, \quad s\geq 1.
\]
For technical convenience we replace 2 here with a smaller exponent.  Let \tred{$Y_1^*,\dots,Y_k^*$} be iid with 
\[
  P(Y_i^* \geq s) = s^{-3/2},  \quad s\geq 1,
\]
so $Y_1^*$ is stochastically larger than $Y_1$.  Then $EY_1^*=3$, and it is standard that, since the tail exponent for $Y_1^*$ lies in $(1,2)$,
\[
  P\left( \sum_{i=1}^k Y_i^* \geq kEY_1^* + s \right) \sim ks^{-3/2} \ \ \text{as } s/k^{2/3} \to\infty,
\]
so for large $k$,
\begin{align}\label{cylsum}
  P&\left( \sum_{i=1}^k T_{d2d}(\mQ_i) \leq g(kx) + \frac{c_{56}}{2}kD_{mag}(|x|) 
    - 5c_{60}k^{1+\xi_1}A_k\big(\sigma_{up}(|x|)D_{mag}(|x|)\big)^{1/2} \right) \notag\\
  &\qquad= P\left( \sum_{i=1}^k Y_i \geq 5k \right) \leq k^{-1/2}.
\end{align}
From (v), provided $K_2$ is large we have that for fixed $k$ and then $|x|$ large (recalling the slab crossing segment defined before \eqref{superadd}),
\[
  P( \Gamma_{0,kx}^{scs} \not\subset \mC^{(k)} ) \leq P\Big( R(kx) \geq K_2\Delta_{up}(k|x|) \Big) \leq \frac14.
\]
With this and \eqref{core2}, \eqref{superadd}, \eqref{cylsum} we get that, again for fixed large $k$ and then $|x|$ large, with probability at least $\frac12-k^{-1/2}$,
\begin{align}\label{assemb}
  g(kx) + c_1D_{mag}^+(k|x|) &\geq g(kx) + c_1D_{mag}(k|x|) + 2c_1\sigma_{mag}(k|x|)  \notag\\
  &\geq ET(0,kx) + 2\hat\sigma(kx) \notag\\
  &\geq T(0,kx) \notag\\
  &\geq \sum_{i=1}^k T_{d2d}(\mQ_i) \notag\\
  &\geq g(kx) + \frac{c_{56}}{2}kD_{mag}(|x|) 
    - 5c_{60}k^{1+\xi_1}A_k\big(\sigma_{up}(|x|)D_{mag}(|x|)\big)^{1/2} \notag\\
  &\geq g(kx) + \frac{c_{56}}{4}kD_{mag}^+(|x|).
\end{align}
Thus (see \eqref{core2}) we have shown that if we take $k$ then $x$ large we have 
\begin{equation}\label{firstit}
  (\theta,x)\in\mG,\quad \frac{D_{mag}(|x|)}{\sigma_{up}(|x|)} \geq k^{2(d+1)\xi_2} \implies D_{mag}^+(k|x|) \geq \frac{c_{56}}{4c_1} kD_{mag}^+(|x|).
\end{equation}
Let $\tred{\ep_8} = c_{56}/4c_1$.  Relation \eqref{firstit} can be iterated: if we start with $k$ satisfying $k^{1-\xi_1} \geq 2c_{3}/c_1^2\ep_8$ and $x$ with $(\theta,x)\in\mG$ satisfying $D_{mag}(|x|)/\sigma_{up}(|x|) \geq k^{2(d+1)\xi_2}$ (as in \eqref{core2}), using \eqref{xicond}we have
\begin{equation}\label{Dmsu}
   \frac{D_{mag}(|kx|)}{\sigma_{up}(|kx|)} + 2 =
  \frac{D_{mag}^+(|kx|)}{\sigma_{up}(|kx|)} \geq \ep_8k^{1-\xi_1} \frac{D_{mag}^+(|x|)}{\sigma_{up}(|x|)} \geq 2c_1^2c_{3} k^{2(d+1)\xi_3}
    \geq c_1^2c_{3}(k^{2(d+1)\xi_3} + 2),
\end{equation}
so \eqref{core} holds for $r=k|x|$.  Therefore we can find $x'$ with $(\theta,x')\in\mG$ with $||x'|-k|x||<\sqrt{d}$ such that \eqref{core2} holds for $x'$, so by \eqref{reggr} and \eqref{firstit} we have  
\begin{equation}\label{Dmkx}
  D_{mag}^+(k^2|x|) \geq c_1^{-2} D_{mag}^+(k|x'|) \geq \ep_8c_1^{-2} kD_{mag}^+(|x'|) \geq \ep_8c_1^{-4} kD_{mag}^+(k|x|)
    \geq \ep_8^2c_1^{-4} k^2 D_{mag}^+(|x|).
\end{equation}
Repeating this, since \eqref{core2} holds for $x'$, there exists $x''$ with $(\theta,x'')\in\mG$ and $||x''|-k|x'||<\sqrt{d}$ such that \eqref{core2} holds for $x''$, so \eqref{Dmkx} holds for $x',x''$ in place of $x,x'$ respectively, so we have
\begin{equation}\label{Dmkx2}
  D_{mag}^+(k^3|x|) \geq c_1^{-2} D_{mag}^+(k^2|x'|) \geq \ep_8^2c_1^{-6} k^2 D_{mag}^+(|x'|) \geq \ep_8^2c_1^{-8} k^2 D_{mag}^+(k|x|)
    \geq \ep_8^3c_1^{-8} k^3 D_{mag}^+(|x|).
\end{equation}
Continuing this, we get
\[
  D_{mag}^+(k^n|x|) \geq \ep_8^nc_1^{-(4n-4)} k^nD_{mag}^+(|x|) \ \ \text{for all } n\geq 1,
\]
so letting $n\to\infty$ shows
\[
  \limsup_{r\to\infty} \frac{\log D_{mag}^+(r)}{\log r} \geq 1 - \frac{\log(\ep_8^{-1}c_1^4)^{-1}}{\log k}.
\]
Since $k$ can be arbitrarily large this shows
\[
  \limsup_{r\to\infty} \frac{\log D_{mag}^+(r)}{\log r} \geq 1,
\]  
which contradicts \eqref{ls1}, completing the proof of Theorem \ref{cylpass}.

\section{Slab vs unrestricted passage times}

For each direction $\psi$, we define discs in each corresponding tangent hyperplane: 
\[
  \tred{\Omega_{\psi,r,\ep}} = \{u\in H_{\psi,r}: |u-\pi_\psi u|\leq \ep r\}.
\]
Recall that $h(x)= ET(0,x)$.  Given $\beta>0$ and a direction of curvature $\theta$, we say another direction $\varphi$ is $\beta$--\emph{directionally concordant with} $\theta$, and write $\varphi\in \tred{DC_{\beta}(\theta)}$, if $\varphi\in H_{\theta,0}^+$ and $\varphi$ makes an angle of at least $\beta$ with $H_{\theta,0}$. 
For such $\varphi$, for $u\in H_{\theta,1}^-$ let \tred{$\pi_{\varphi,\theta} u$} denote the intersection with $H_{\theta,1}$ of the line through $u$ in direction $\varphi$. Let \tred{$\pi_{\varphi,\theta}^{\pa}$} denote the restriction of $\pi_{\varphi,\theta}$ to $\pa\mkB_g$, and let 
\[
  \tred{f_{\varphi,\theta}(u)} = \left| u - (\pi_{\varphi,\theta}^{\pa})^{-1}u\right|,
\]
which represents ``the distance from $u$ to $\mkB_g$ in direction $-\varphi$''
and is well--defined for $u$ in some neighborhood of $\yt$ in $H_{\theta,1}$.

We first prove that a geodesic $\Gamma_{0x}$ is very likely to cross a hyperplane, roughly parallel to the natural hyperplanes of $(0,x)$, close to where the line $\Pi_{0x}^\infty$ crosses the hyperplane.  We will use the following lemma. Recall $\ep_1$ from the definition of near--natural slab.

\begin{lemma}\label{gsmooth}
\tcyn{Given $\beta>0$ there exist constants $c_i$ as follows. Let $\theta$ be a direction of curvature.  There exists a neighborhood $U$ of $\yt$ in $H_{\theta,1}$ such that for all $u,v\in U$,
\begin{equation}\label{gsmooth2}
  |f_{\varphi,\theta}(v) - f_{\varphi,\theta}(u)| \leq c_{62}|v-u| \big(|u-\yt| + |v-\yt|\big)\ \ \text{for all } \varphi \in DC_\beta(\theta),
\end{equation}
and
\begin{equation}\label{gsmooth3}
  |g(v) - g(u)| \leq c_{63}|v-u| \big(|u-\yt| + |v-\yt|\big).
\end{equation}}
\end{lemma}

\begin{proof}
We first prove \eqref{gsmooth2}.
We may assume $f_{\varphi,\theta}(v)>f_{\varphi,\theta}(u)$. Take $U$ small enough so $(\pi_{\varphi,\theta}^{\pa})^{-1}$, and thus also $f_{\varphi,\theta}$, is well--defined on $U$.
Write \tred{$\hat u$} for $(\pi_{\varphi,\theta}^{\pa})^{-1}u,u\in U$.  We use the fact that by convexity of $\mkB_g$, $\Pi_{\hat u\hat v}^\infty$ intersects the interior $\mkB_g^\circ$ at most in the line segment $[\hat u,\hat v]$, so letting $\tred{p_\lambda}=\hat u + \lambda(\hat v-\hat u)$, for $\lambda\geq 1$ we have $p_\lambda\notin\mkB_g^\circ$ and therefore $f_{\varphi,\theta}(\pi_{\varphi,\theta}(p_\lambda)) \geq |p_\lambda - \pi_{\varphi,\theta}(p_\lambda)|$. See the upper diagram in Figure \ref{Lem4.1_Fig1}. It follows using \eqref{curvg} and the $\beta$--directionally--concordant property that
\begin{align}\label{heights}
  f_{\varphi,\theta}(u) + \lambda\big(f_{\varphi,\theta}(v) - f_{\varphi,\theta}(u)\big) &= |p_\lambda - \pi_{\varphi,\theta}(p_\lambda)| \notag\\
  &\leq f_{\varphi,\theta}(\pi_{\varphi,\theta}(p_\lambda)) \notag\\
  &\leq c_{63}|\pi_{\varphi,\theta}(p_\lambda) - \yt|^2.
\end{align}
Now we choose
\[
  \lambda = \frac{|u-\yt|+|v-\yt|}{|v-u|},
\]
for which we have
\[
  |\pi_{\varphi,\theta}(p_\lambda) - u| = \lambda|v-u| = |u-\yt| + |v-\yt|,
\]
and thereby we obtain
\[
  |\pi_{\varphi,\theta}(p_\lambda) - \yt| \leq  |\pi_{\varphi,\theta}(p_\lambda) - u| + |u-\yt| = 2|u-\yt| + |v-\yt|.
\]
With \eqref{heights} (omitting $f_{\varphi,\theta}(u)$ on the left) this shows
\begin{equation}
  \frac{|u-\yt|+|v-\yt|}{|v-u|}\big(f_{\varphi,\theta}(v) - f_{\varphi,\theta}(u)\big) \leq 4c_{63} (|u-\yt| + |v-\yt|)^2,
\end{equation}
which is \eqref{gsmooth2}.

\begin{figure}
\includegraphics[height=9cm]{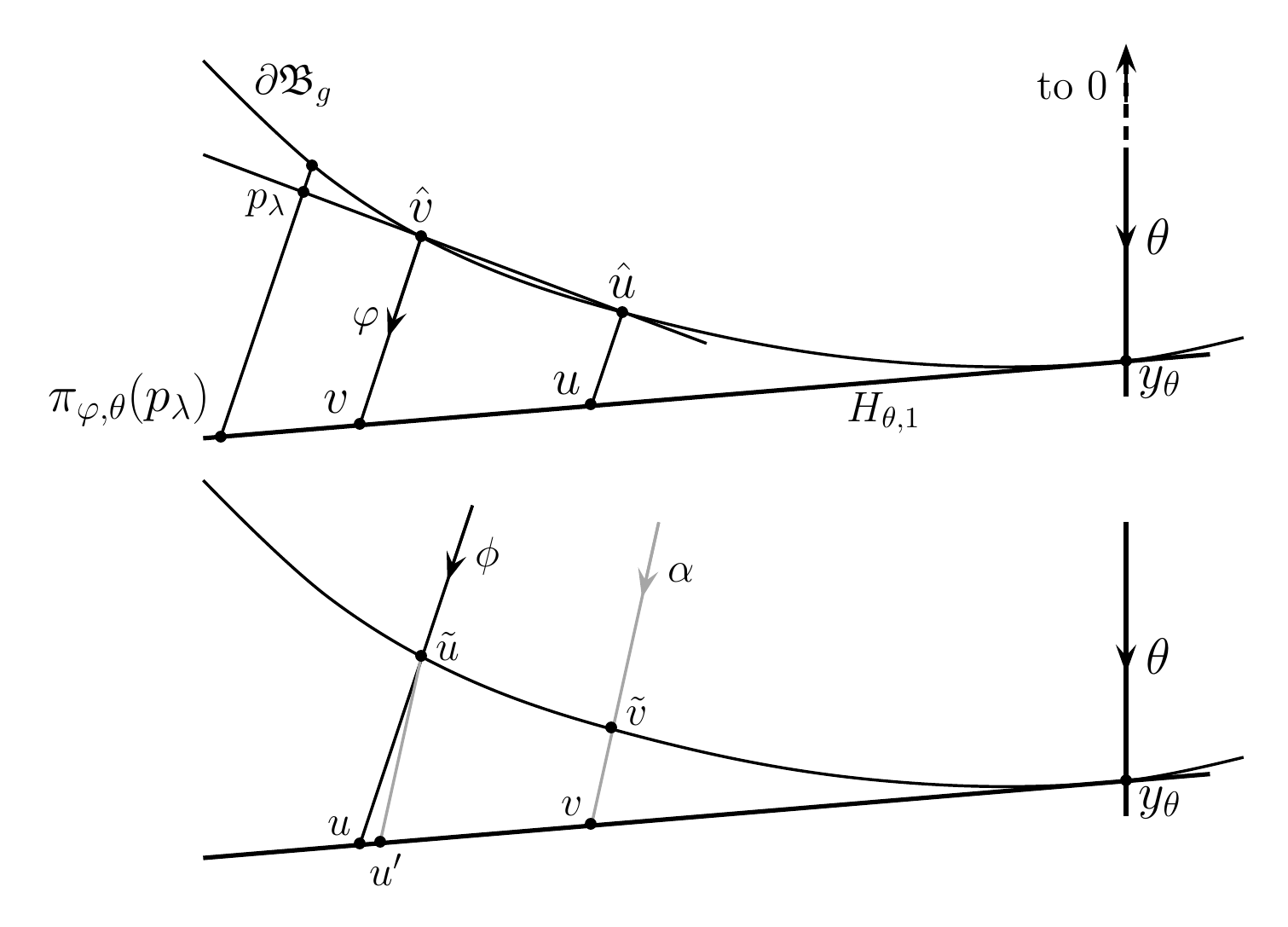}
\caption{ Diagrams, not to scale, for the proof of \eqref{gsmooth2} (upper) and \eqref{gsmooth3} (lower). In the upper figure, the lines downward from $p_\lambda,\hat v$, and $\hat u$ are parallel. In the lower figure, the 3 lines in directions $\phi,\alpha,\theta$ all originate at 0, and the 2 lines downward from $\tilde u$ are in directions $\phi$ and $\alpha$. }
\label{Lem4.1_Fig1}
\end{figure}

Turning to \eqref{gsmooth3}, let $\tred{\tu,\tv}$ denote the points where $\Pi_{0u}^\infty$ and $\Pi_{0v}^\infty$, respectively, intersect $\pa \mkB_g$, let $\tred{\alpha}= v/|v|,\,\tred{\psi}=u/|u|$, and $\tred{u'} = \pi_{\alpha,\theta}\tu$.  See the lower diagram in Figure \ref{Lem4.1_Fig1}. We note first that for all angles $\zeta\in H_{\theta,0}^+$,
\begin{equation}\label{conc}
  \frac{|\pi_{\zeta,\theta}\tu-\tu|}{|\pi_{\zeta,\theta}\tv-\tv|} = \frac{g(\pi_{\zeta,\theta}\tu-\tu)}{g(\pi_{\zeta,\theta}\tv-\tv)},
\end{equation}
since the numerator and denominator involve increments in the same direction $\zeta$.
Provided $U$ is small, we have from the arcsin bound that $\varphi\in DC_\beta(\theta)$ for $\beta=\frac12\arcsin 1/\sqrt{d}$, so from \eqref{gsmooth2} we have
\begin{align}\label{use41}
  \big| |u'-\tu| - |v-\tv| \big| &\leq c_{63} |u'-v|\big( |v-\yt| + |u'-\yt| \big) \notag\\
  &\leq c_{63}(|u'-u| + |u-v|) \big( |v-\yt| + |u'-u| + |u-\yt| \big).
\end{align}
Since the triangle $\Delta u0v$ is just a dilation of $\Delta u\tu u'$ by a factor $|u|/|u-\tu|$,
we have, using \eqref{conc},
\begin{align}\label{gsizes}
  |g(u)-g(v)| &= g(u-\tu)-g(v-\tv) \notag\\
  &\leq |g(u-\tu) - g(u'-\tu)| + |g(u'-\tu) -g(v-\tv)| \notag\\
  &= \frac{|u-\tu|}{|u|} |g(u)-g(v)| + \frac{g(v-\tv)}{|v-\tv|} \big| |u'-\tu| - |v-\tv| \big|,
\end{align}
and
\begin{align}\label{gsizes3}
  |u'-u| = |v-u| \frac{|u-\tu|}{|u|} \leq |u-v| \leq |v-\yt| + |u-\yt|.
\end{align}
Assuming $U$ is small enough we have $|u-\tu|/|u| <1/2$ so \eqref{use41}, \eqref{gsizes}, and \eqref{gsizes3} yield
\begin{align}\label{gsizes2}
  |g(u)-g(v)| &\leq 2\frac{g(v-\tv)}{|v-\tv|} \big| |u'-\tu| - |v-\tv| \big|  \leq c_{64}|u-v| \big( |v-\yt| + |u-\yt| \big),
\end{align}
proving \eqref{gsmooth3}.
\end{proof}

We will need the following exponential concentration result from \cite{DHS14}. The result there is actually stronger, with $(|x|/\log|x|)^{1/2}$ in place of $|x|^{1/2}$, but the improvement doesn't help us here. Earlier version appeared in \cite{Ke93} (restricted to $t\leq C|x|$ for some $C$) and \cite{Ta95} (similarly restricted, but improved to $t^2$ in place of $t$ in the exponent.)

\begin{lemma}\label{DHS} \cite{DHS14}
\tcyn{For a standard FPP model, there exist constants $c_i$ such that for all $x\in\ZZ^2$ and all $t>0$,
\begin{equation}\label{concen}
  P\left( |T(0,x) - ET(0,x)|\geq t|x|^{1/2}\right) \leq c_{65}e^{-c_{66}t}.
\end{equation}}
\end{lemma}

Variants of the following have been proved under an assumption (or proven fact, for solvable LPP) of curvature uniform in a neighborhood of $\theta$ (\cite{Al20}, \cite{BSS19}, \cite{Ga20}); here we reduce the assumption to $\theta$ alone being a direction of curvature, which adds significant technicality.

\begin{lemma}\label{missball}
\tcyn{For a standard FPP in $d$ dimensions, for all sufficiently small $\ep>0$ there exist $c_i,\delta>0$ as follows.  Let $\theta\in\mU$, let $\psi$ be another unit vector, and let $x\in\ZZ^d$ with
\begin{equation}\label{xcond}
  |\psi - \theta|  \leq \delta, \quad \left|\frac{x}{|x|}-\theta\right| \leq g(x)^{-1/6}.
\end{equation}
For all $c_{67}\leq t\leq g(\pi_\theta x)/4$, 
\begin{equation}\label{missball2}
  P\Big( \Gamma_{0x}\cap (H_{\psi,t} \bs \Omega_{\psi,t,\ep}) \neq \emptyset \Big) \leq c_{68}\exp\left(-c_{69}\ep^2t^{1/48}\right).
\end{equation}}
\end{lemma}  

\begin{proof}
For notational convenience we prove \eqref{missball2} with $\Omega_{\psi,t,3\ep}$ in place of $\Omega_{\psi,t,\ep}$.
Recall $\ep_0$ from \eqref{curvg}--\eqref{hyperang}, let
$0<\tred{\ep}\leq\ep_0/2$, and let \tred{$G$} be the event that $\Gamma_{0x}\cap (H_{\psi,t} \bs \Omega_{\psi,t,3\ep}) \neq \emptyset$.  We take $\delta$ in \eqref{xcond} small enough so that
\begin{equation}\label{deltacond}
  u \in H_{\psi,t} \bs \Omega_{\psi,t,2\ep} \implies |u-\pi_\theta u| > \ep t.
\end{equation}
Recall that by definition, $x\in H_{\varphi,g(\pi_\varphi x)}$ for all $\varphi$. 

We now define an open region $\Lambda=\Lambda(x,\theta,t,\ep)$ with the property that when the event $G$ occurs, $\Gamma_{0x}$ must exit $\Lambda$.  We combine the cone and infinite skew $\theta$--cylinder
\[
  \tred{\mC_{\theta,\ep}(x)} = \left\{u: \Phi_\theta(u)>0, \Theta_\theta(u) < \ep\right\}, \quad 
    \tred{\mD_{\theta,2\ep t}} = \{u: |u-\pi_\theta u| < 2\ep t\},
\]
to create
\[
  \tred{\Lambda}= \mC_{\theta,\ep}(x) \cup \mD_{\theta,2\ep t};
\]
see Figure \ref{Lem4.2case3_Fig}. 
Note that the intersection of the boundaries of the cylinder and cone lies in $H_{\theta,2t}$, and that by \eqref{deltacond} we have $H_{\psi,t} \bs \Omega_{\psi,t,2\ep} \subset \Lambda^c$. 
For a region $\Xi\subset\RR^d$ we write $\tred{\pa_{\ZZ^d}\Xi}$ for the set of all sites in $\Xi^c$ adjacent to sites in $\Xi$.
On the event $G$ let \tred{$x^{(0)}$} be the last point of $\Gamma_{0x}$ in $\pa_{\ZZ^d}\Lambda$. Observe that
\begin{equation}\label{start}
  \Theta_\theta(x^{(0)}) \geq \ep\ \text{if}\ \Phi_\theta(x^{(0)}) > 0,\quad \text{and}\ \quad \Theta_\theta(x) \leq c_{70}g(x)^{-1/6},
\end{equation}
the latter coming from \eqref{xcond}. We consider cases, mainly according to the value of $\Phi_\theta(x^{(0)})$.

{\bf Case 1.} No backtracking to $x^{(0)}$, and $\Phi_\theta(x^{(0)})$ and $t$ are not too large:
\begin{equation}\label{smallt}
  \Phi_\theta(x^{(0)}) > 0,\qquad 0 <  \Phi_\theta(x^{(0)}) \vee 2t < \Phi_\theta(x)^{4/5}.
\end{equation}
Fix $\tred{\nu}\in (\frac54,\frac85)$, let $\tred{s_0} = \Phi_\theta(x^{(0)}) \vee 2t$, and define 
\[
  \tred{s_i} = s_0^{\nu^i}, \ \ i\geq 0.
\]
Then for $i\geq 1$ let \tred{$x^{(i)}$} be the first point of $\Gamma_{x^{(0)}x}$ with $\Phi_\theta(x^{(i)}) \geq s_i$ (necessarily in $\Lambda$), for $i$ for which such a point exists.  The union of the intervals $(s_0^{(5/4)^i},s_0^{(8/5)^i}), i\geq 1$ is $(s_0^{5/4},\infty)$, which contains $\Phi_\theta(x)$, so
we can choose $\nu$ so that $s_m = \Phi_\theta(x)$ for some $\tred{m}\geq 1$ (this being the purpose of the allowed range $\nu\in (\frac54,\frac85)$.)  This means $x^{(i)}$ is defined for $0\leq i\leq m$, and we now redefine $\tred{x^{(m+1)}}=x, \tred{s_{m+1}}=s_m$.  Observe that we have
\begin{equation}\label{bdrybd}
  (\Theta_\theta(x) \wedge \ep)s_i \leq |x-\pi_\theta x| \leq 2(\Theta_\theta(x) \wedge \ep)s_i \ \text{ for all $i\geq 0$ and } 
    x\in H_{\theta,\Phi_\theta(x^{(i)})}\cap(\Lambda \cup \pa_{\ZZ^d}\Lambda)
\end{equation}
(where we can omit the ``$\wedge\,\ep$'' if $i\geq 1$.)

We consider how the ``angle'' $\Theta_\theta$ changes as $\Gamma_{0x}$ progresses through the points $x^{(0)},\dots,x^{(m+1)}=x$.  
Provided $t$ (hence $s_0$) is large we have from \eqref{start} that
\[
  s_0^{1/12}\Theta_\theta(x^{(0)}) \geq 1 \quad\text{and}\quad s_{m+1}^{1/12}\Theta_\theta(x^{(m+1)}) \leq c_{71}g(x)^{-1/12}.
\]
Therefore there exists an index $1\leq \tred{\ell} \leq m+1$ for which $\Theta_\theta(\cdot)$ drops sharply from $x_{\ell-1}$ to $x_\ell$, in the sense that
\begin{equation}\label{elldef}
  s_{\ell-1}^{1/12}\Theta_\theta(x^{(\ell-1)}) \geq 1, \quad s_\ell^{1/12}\Theta_\theta(x^{(\ell)}) < 1.
\end{equation}
Fixing $\tred{\eta}>0$, provided $t$ is large we have
\begin{equation}\label{anglechg}
  \frac{\Theta_\theta(x^{(\ell)})}{\Theta_\theta(x^{(\ell-1)})} \leq \frac{s_{\ell-1}^{1/12}}{s_\ell^{1/12}} = s_{\ell-1}^{-(\nu-1)/12} < \eta.
\end{equation}
We want a lower bound for the ``extra distance''
\begin{equation}\label{xdist}
  g(x^{(\ell-1)}) + g(x^{(\ell)} - x^{(\ell-1)}) - g(x^{(\ell)})
\end{equation}
caused by \eqref{elldef} and \eqref{anglechg}. We consider two subcases.

{\bf Case 1a.} $\ell\leq m$. Define the point
\[
    \tred{y^{(\ell-1)}} = H_{\theta,\Phi_\theta(x^{(\ell-1)})} \cap \Pi_{0x^{(\ell)}}
\]
so \eqref{xdist} can be expressed as
\begin{equation}\label{xdist2}
  [g(x^{(\ell-1)}) - g(y^{(\ell-1)})] - [g(x^{(\ell)} - y^{(\ell-1)}) - g(x^{(\ell)} - x^{(\ell-1)})].
\end{equation}
Let \tred{$z^{(\ell-1)}$} be the longitudinal $\theta$--projection of $x^{(\ell)}$ into $H_{ \theta,\Phi_\theta(x^{(\ell-1)}) }$. See Figure \ref{Lem4.2case1a_Fig}. 
From \eqref{curvg}, \eqref{curvg2}, 
\[
   \Theta_\rho(y^{(\ell-1)}) 
     = \Theta_\rho(x^{(\ell)})\ \forall \rho \ \ \text{(by collinearity with 0)},
\]
 \begin{equation}\label{Wtilde}
   \Theta_\theta(y) \leq \ep \implies 
     c_{6}\Theta_\theta(y)^2g(\pi_\theta y) \leq g(y) -g(\pi_\theta y) \leq c_{5}\Theta_\theta(y)^2g(\pi_\theta y).
 \end{equation}
 
\begin{figure}
\includegraphics[height=5cm]{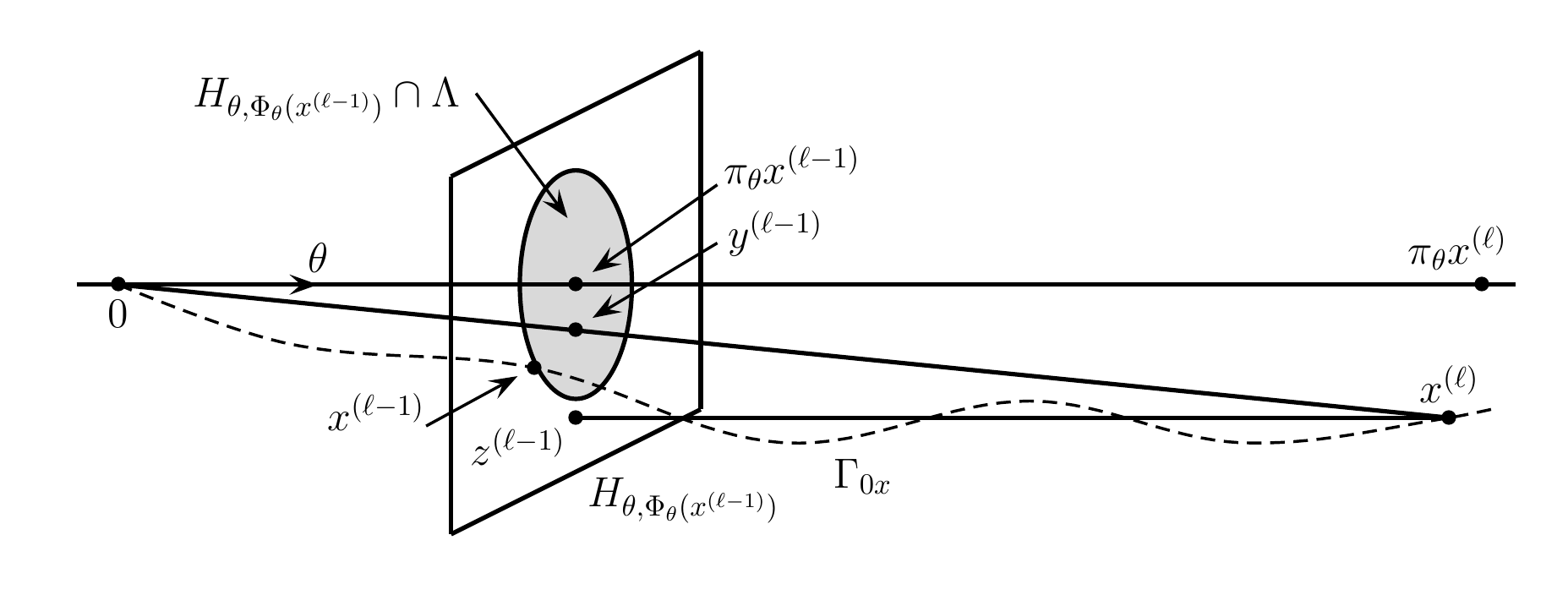}
\caption{ Diagram for Case 1a. All labeled points except $x^{(\ell-1)}$ are coplanar. }
\label{Lem4.2case1a_Fig}
\end{figure}

We first find a lower bound for the first difference in \eqref{xdist2}. We have from \eqref{curvg}--\eqref{curvgwide}, \eqref{anglechg}, and \eqref{Wtilde} that for some $c_{72}=c_{72}(\theta)$,
\begin{align}\label{grels}
  |g(y^{(\ell-1)}) - g(\pi_\theta x^{(\ell-1)})| &\leq c_{5}  \Theta_\theta(x^{(\ell)})^2 g(\pi_\theta x^{(\ell-1)}) 
    \leq s_{\ell-1}^{-(\nu-1)/12} \Theta_\theta(x^{(\ell)}) \Theta_\theta(x^{(\ell-1)}) g(\pi_\theta x^{(\ell-1)})  \notag\\
  g(x^{(\ell-1)}) - g(\pi_\theta x^{(\ell-1)}) &\geq c_{72} \left( \Theta_\theta(x^{(\ell-1)}) \wedge \ep \right) \Theta_\theta(x^{(\ell-1)})
    g(\pi_\theta x^{(\ell-1)}),
\end{align}
so
\begin{align}\label{short1}
  g(x^{(\ell-1)}) - g(y^{(\ell-1)}) &\geq \frac{c_{72}}{2} \left( \Theta_\theta(x^{(\ell-1)}) \wedge \ep \right) \Theta_\theta(x^{(\ell-1)})
    g(\pi_\theta x^{(\ell-1)}) \notag\\
  &= \frac{c_{72}}{2} \left( \Theta_\theta(x^{(\ell-1)}) \wedge \ep \right) |x^{(\ell-1)} - \pi_\theta x^{(\ell-1)}|.
\end{align}
Pursuing next an upper bound for the second difference in \eqref{xdist2}, we want to use Lemma \ref{gsmooth}.  From \eqref{bdrybd} we have
\begin{equation}\label{xpix}
  \left( \Theta_\theta(x^{(\ell-1)}) \wedge \ep \right) s_{\ell-1}  \leq |x^{(\ell-1)} - \pi_\theta x^{(\ell-1)}| 
    \leq 2\left( \Theta_\theta(x^{(\ell-1)}) \wedge \ep \right) s_{\ell-1} 
\end{equation}
and
\begin{equation}\label{ypix}
  |y^{(\ell-1)} - \pi_\theta x^{(\ell-1)}| \leq 2\Theta_\theta(y^{(\ell-1)}) s_{\ell-1} = 2\Theta_\theta(x^{(\ell)}) s_{\ell-1}
\end{equation}
so
\begin{equation}\label{xydist}
  |y^{(\ell-1)} - x^{(\ell-1)}| \leq 4|x^{(\ell-1)} - \pi_\theta x^{(\ell-1)}|.
\end{equation}
Then note that $\pi_\theta x^{(\ell-1)}, y^{(\ell-1)},z^{(\ell-1)}$ are collinear with $y^{(\ell-1)}$ between the other two, so
\begin{equation}\label{zy}
  |z^{(\ell-1)} - y^{(\ell-1)}| \leq |z^{(\ell-1)} - \pi_\theta x^{(\ell-1)}| = |x^{(\ell)} - \pi_\theta x^{(\ell)}| \leq 2\Theta_\theta(x^{(\ell)}) s_\ell
\end{equation}
which with \eqref{xpix} shows that
\begin{equation}\label{zx}
  |z^{(\ell-1)} - x^{(\ell-1)}| \leq |z^{(\ell-1)} - \pi_\theta x^{(\ell-1)}| + |\pi_\theta x^{(\ell-1)} - x^{(\ell-1)}| \leq 
    2\Theta_\theta(x^{(\ell)}) s_\ell + 2\left( \Theta_\theta(x^{(\ell-1)}) \wedge \ep \right) s_{\ell-1}.
\end{equation}
Lemma \ref{gsmooth} together with \eqref{short1}, \eqref{xydist}, \eqref{zy}, and \eqref{zx} then show that
\begin{align}
  g(x^{(\ell)} - &y^{(\ell-1)}) - g(x^{(\ell)} - x^{(\ell-1)}) \notag\\
  &\leq c_{62}\frac{ |y^{(\ell-1)} - x^{(\ell-1)}| \left( |z^{(\ell-1)} - x^{(\ell-1)}| 
    + |z^{(\ell-1)} - y^{(\ell-1)}| \right)}{ g(\pi_\theta x^{(\ell)} - \pi_\theta x^{(\ell-1)}) } \notag\\
  &\leq 4c_{62}\frac{ |x^{(\ell-1)} - \pi_\theta x^{(\ell-1)}|
    \left( 2\left[ \Theta_\theta(x^{(\ell-1)}) \wedge \ep \right] s_{\ell-1} + 4\Theta_\theta(x^{(\ell)}) s_\ell \right) }{ s_\ell } \notag\\
  &\leq \frac12 \left( g(x^{(\ell-1)}) - g(y^{(\ell-1)}) \right).
\end{align}
Hence in view of \eqref{elldef}, \eqref{xdist2}, and \eqref{short1} the difference in \eqref{xdist} satisfies
\begin{align}\label{xtdist}
  g(&x^{(\ell-1)}) + g(x^{(\ell)} - x^{(\ell-1)}) - g(x^{(\ell)}) \notag\\
  &= [g(x^{(\ell-1)}) - g(y^{(\ell-1)})] - [g(x^{(\ell)} - y^{(\ell-1)}) - g(x^{(\ell)} - x^{(\ell-1)})] \notag\\
  &\geq \frac12 \left( g(x^{(\ell-1)}) - g(y^{(\ell-1)}) \right) \notag\\
  &\geq \frac{c_{72}}{4} \left( \Theta_\theta(x^{(\ell-1)}) \wedge \ep \right) |x^{(\ell-1)} - \pi_\theta x^{(\ell-1)}| \notag\\
  &\geq \frac{c_{72}}{8} \left( \Theta_\theta(x^{(\ell-1)})^2 \wedge \ep^2 \right) s_{\ell-1} \notag\\
  &\geq \frac{c_{72}}{8} s_{\ell-1}^{5/6}.
\end{align}
Let $\tred{\kappa}=5/(6\nu)-1/2\in(\frac{1}{48},\frac16)$, so $s_{\ell-1}^{5/6}=s_\ell^{\kappa+1/2}$.
By \eqref{ls1}, \eqref{elldef}, and \eqref{xtdist} we have
\begin{align}\label{ggap}
  h&(x^{(\ell-1)}) + h(x^{(\ell)} - x^{(\ell-1)}) - h(x^{(\ell)}) \geq \frac{c_{72}}{8} s_\ell^{\kappa+1/2} - c_{73}(s_\ell \log s_\ell)^{1/2}
    \geq \frac{c_{72}}{16} s_\ell^{\kappa+1/2},
\end{align}
while since $x^{(\ell-1)} \in \Gamma_{0x^{(\ell)}}$ we have
\begin{equation}\label{Tnogap}
  T(0,x^{(\ell-1)}) + T(x^{(\ell-1)},x^{(\ell)}) - T(0,x^{(\ell)}) = 0.
\end{equation}
Therefore 
\[
  \max\Big( |T(0,x^{(\ell-1)}) - h(x^{(\ell-1)})|, |T(x^{(\ell-1)},x^{(\ell)}) - h(x^{(\ell)} - x^{(\ell-1)})|, |T(0,x^{(\ell)}) - h(x^{(\ell)})| \Big) 
    \geq \frac{c_{72}}{48}  s_\ell^{\kappa+1/2}.
\]
Under Case 1 we have $\max(g(x^{(\ell-1)}),g(x^{(\ell)})) \leq 2s_\ell$.  Therefore defining events
\[
  \tred{G_i}: \text{ there exist $u,v\in \mkB_g(0,2s_i)$ with $|T(u,v) - h(v-u)| \geq \frac{c_{72}}{48}  s_i^{\kappa+1/2}$},
\]
we see using Lemma \ref{DHS} that 
\begin{align}\label{case1a}
  P\Big( \Gamma_{0x}\cap (H_{\psi,t} \bs \Omega_{\psi,t,\ep}) \neq \emptyset \text{ and Case 1a holds}\Big) 
    &\leq \sum_{i=1}^m P(G_i) \notag\\
  &\leq \sum_{i=1}^m |\mkB_g(0,2s_i)\cap\ZZ^d|^2 c_{74}e^{-c_{75}s_i^\kappa} \notag\\
  &\leq c_{76}e^{-c_{77}t^\kappa}.
\end{align}

{\bf Case 1b.} $\ell=m+1$. Here $x^{(m)}$ and $x^{(m+1)}=x$ are both in or next to $H_{\theta,g(\pi_\theta x)}$.  From \eqref{elldef} we have
\[
  |x^{(m)} - \pi_\theta x^{(m)}| = \Theta_\theta(x^{(m)}) g(\pi_\theta x^{(m)}) \geq s_m^{11/12}
\]
while from \eqref{curvg} and \eqref{xcond},
\begin{equation}\label{xtopi}
  |x - \pi_\theta x| \leq c_{78} s_m^{5/6} \quad\text{so}\quad |g(x) - g(\pi_\theta x)| \leq c_{79}s_m^{2/3}.
\end{equation}
It then follows from \eqref{curvg}--\eqref{curvgwide} that $g(x^{(m)}) \geq g(x)$ so
\[
  g(x^{(m)}) + g(x-x^{(m)}) - g(x) \geq g(x-x^{(m)}) \geq c_{80}|x-x^{(m)}| \geq \frac{c_{80}}{2}
    |x^{(m)} - \pi_\theta x^{(m)}| \geq \frac{c_{80}}{2} s_m^{11/12}
\]
and then from \eqref{ls1},
\[
  h(x^{(m)}) + h(x-x^{(m)}) - h(x) \geq \frac{c_{80}}{2} s_m^{11/12} - c_{62}(s_m\log s_m)^{1/2} \geq \frac{c_{80}}{4} s_m^{11/12}.
\]
Defining the event
\[
  \tred{Q_m}: \text{ there exist $u,v\in \mkB_g(0,2s_m)$ with } |T(u,v) - h(v-u)| \geq \frac{c_{80}}{12} s_m^{11/12}
\]
it again follows as in \eqref{Tnogap}--\eqref{case1a} that
\begin{align}\label{case1b}
  P\Big( \Gamma_{0x}\cap (H_{\psi,t} \bs \Omega_{\psi,t,\ep}) \neq \emptyset \text{ and Case 1b holds}\Big) &\leq P(Q_m) \notag\\
  &\leq |\mkB_g(0,2s_m)\cap\ZZ^d|^2 c_{82}e^{-c_{83}s_m^{1/3}} \notag\\
  &\leq c_{82}e^{-c_{83}s_m^{1/3}/2}.
\end{align}

{\bf Case 2.} Backtracking occurs: $\Phi_\theta(x^{(0)}) \leq 0$.  This time let \tred{$x^{(1)}$} be the first point of $\Gamma_{x^{(0)}x}$ with $\Phi_\theta(x^{(1)}) \geq 2t$ (necessarily in $\Lambda$), define the point 
\[
  \tred{w^{(0)}} = \Pi_{x^{(0)},x^{(1)}} \cap H_{\theta,0},
\]
and let \tred{$u^{(0)}$} be the longitudinal $\theta$--projection of $x^{(1)}$ into $H_{\theta,0}$.  See Figure \ref{Lem4.2case2_Fig}. In place of \eqref{xdist2} we express \eqref{xdist} as
\begin{equation}\label{xdist3}
    g(x^{(0)}) + g(w^{(0)} - x^{(0)}) - [g(x^{(1)}) - g(x^{(1)} - w^{(0)})].
\end{equation}
We have
\begin{align}\label{comp}
  g(x^{(0)}) &> c_{84}|x^{(0)}| \geq c_{84}\ep t, \quad |u^{(0)}| = |x^{(1)} - \pi_\theta x^{(1)}| \leq 3\ep t, \notag\\
    &|w^{(0)}| \leq \max(|x^{(0)} - \pi_\theta x^{(0)}|,|x^{(1)} - \pi_\theta x^{(1)}|) \leq 3\ep t,
\end{align}
so
\begin{equation}\label{uwbound}
  |u^{(0)} - w^{(0)}| \leq |u^{(0)}| + |w^{(0)}| \leq 6\ep t.
\end{equation}
\begin{figure}
\includegraphics[height=5cm]{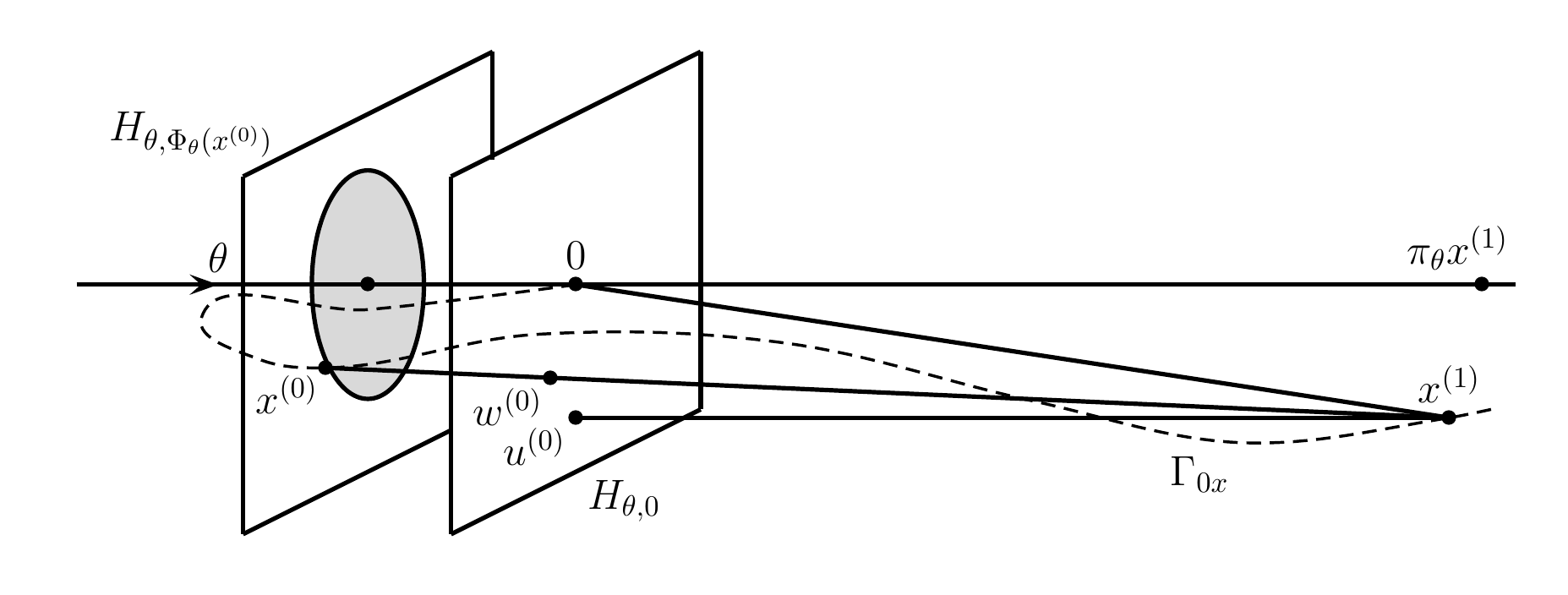}
\caption{ Diagram for Case 2. The shaded region is the intersection of the hyperplane with $\Lambda$. The geodesic $\Gamma_{0x}$ follows a path $0\to x^{(0)} \to x^{(1)} \to x$. }
\label{Lem4.2case2_Fig}
\end{figure}
From Lemma \ref{gsmooth} together with \eqref{comp}--\eqref{uwbound} we then get
\begin{align}\label{longgap}
  |g(x^{(1)}) - g(x^{(1)} - w^{(0)})| 
    \leq c_{63} \frac{ |w^{(0)}| \left( |u^{(0)}| + |u^{(0)} - w^{(0)}| \right) }{\Phi_\theta(x^{(1)})} 
    \leq \frac{27c_{63}}{2} \ep^2 t,
\end{align}
which with \eqref{xdist3} and \eqref{comp} shows that, provided we take $\ep$ small,
\begin{equation}\label{xdist4}
  g(x^{(0)}) + g(x^{(1)} - x^{(0)}) - g(x^{(1)}) \geq \frac12 g(x^{(0)}). 
\end{equation}

In view of \eqref{comp}, for some $k\geq 1$ we have $2^{k-1}c_{84}\ep t<g(x^{(0)})\leq 2^{k-1}c_{84}\ep t$ and hence
\[
  \max\big( g(x^{(0)}),g(x^{(1)} - x^{(0)}),g(x^{(1)}) \big) \leq c_{85}2^k t,
\]
which with \eqref{ls1} and \eqref{xdist4} shows that
\begin{align}\label{hgap}
  h(x^{(0)}) + h(x^{(1)} - x^{(0)}) - h(x^{(1)}) &\geq \frac12 g(x^{(0)}) - c_{4}\big( 2^kt
    \log(2^k t)  \big)^{1/2} \geq \frac14 g(x^{(0)}) \geq 2^k c_{86}\ep t.
\end{align}
We can now follow \eqref{ggap}--\eqref{case1a}, defining
\[
  \tred{G_k'}: \text{ there exist $u,v\in \mkB_g(0,c_{85}2^kt)$ with } |T(u,v) - h(v-u)| \geq \frac{2^k c_{86}\ep t}{3}
\]
so that
\begin{align}\label{case2}
  P\Big( \Gamma_{0x}\cap (H_{\psi,t} \bs \Omega_{\psi,t,\ep}) \neq \emptyset \text{ and Case 2 holds}\Big)
    &\leq \sum_{k\geq 1} P(G_k') \notag\\
  &\leq \sum_{k\geq 1} |\mkB_g(0,c_{85}2^kt)\cap\ZZ^d|^2 c_{87}\exp\left(- c_{88}\ep(2^kt)^{1/2} \right) \notag\\
  &\leq c_{89}e^{-c_{90}\ep t^{1/2}}.
\end{align}

{\bf Case 3.} No backtracking to $x^{(0)}$, and $\Phi_\theta(x^{(0)})$ or $t$ is large:
\begin{equation}\label{smallt}
 \Phi_\theta(x)^{4/5} \leq  \Phi_\theta(x^{(0)}) \vee 2t \leq \Phi_\theta(x) \quad\text{and}\quad \Phi_\theta(x^{(0)}) > 0.
\end{equation}
Define the point
\[
  \tred{p^{(0)}} = \Pi_{0x} \cap H_{\theta,\Phi_\theta(x^{(0)}) }
\]
and let \tred{$q^{(0)}$} be the longitudinal $\theta$--projection of $x$ into $H_{\theta,\Phi_\theta(x^{(0)})}$. See Figure \ref{Lem4.2case3_Fig}.
We want a lower bound for
\begin{equation}\label{xdist5}
 g(x^{(0)}) + g(x - x^{(0)}) - g(x) = [g(x^{(0)}) - g(p^{(0)})] - [g(x - p^{(0)}) - g(x - x^{(0)})],
\end{equation}
analogously to \eqref{xdist2}. We first consider the first difference on the right in \eqref{xdist5}. We have 
\begin{align}
  g(x^{(0)}) - g(p^{(0)}) &= [g(x^{(0)}) - g(\pi_\theta x^{(0)})] - [g(p^{(0)}) - g(\pi_\theta x^{(0)})]
\end{align}
with, by \eqref{curvg} and \eqref{xcond},
\[
  g(p^{(0)}) - g(\pi_\theta x^{(0)}) \leq c_{5}\Theta_\theta(p^{(0)})^2 \Phi_\theta(x^{(0)}) = c_{5}\Theta_\theta(x)^2 \Phi_\theta(x^{(0)})
    \leq c_{91} |x|^{2/3}
\]
and, from \eqref{curvgwide} and \eqref{smallt},
\[
  g(x^{(0)}) - g(\pi_\theta x^{(0)}) \geq c_{92}\ep^2 \big( \Phi_\theta(x^{(0)}) \vee t \big) \geq c_{93}\ep^2 |x|^{4/5},
\]
so 
\begin{equation}\label{xplower}
  g(x^{(0)}) - g(p^{(0)}) \geq \frac12 \big( g(x^{(0)}) - g(\pi_\theta x^{(0)}) \big) \geq c_{94}\ep^2 |x|^{4/5}.
\end{equation}
\begin{figure}
\includegraphics[height=5cm]{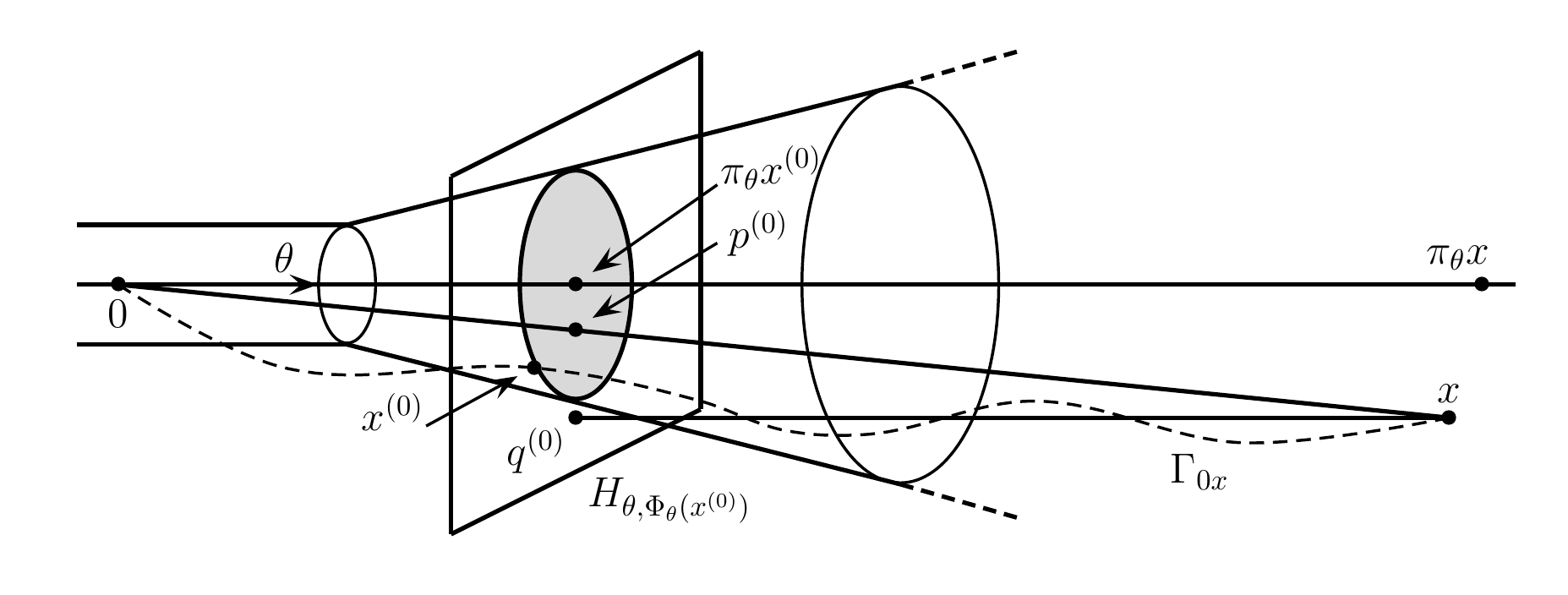}
\caption{ Diagram for Case 3. The cylinder and cone portions of $\partial\Lambda$ intersect in $H_{\theta,2t}$. The hyperplane may cross the cylinder part of $\Lambda$ rather than the cone. }
\label{Lem4.2case3_Fig}
\end{figure}
Considering next the second difference on the right in \eqref{xdist5}, we have using \eqref{curvg}
\[
  g(x - p^{(0)}) - g(x - x^{(0)}) \leq g(x - p^{(0)}) - g(x - q^{(0)}) \leq c_{5}\Theta_\theta(x)^2\Phi_\theta(x-x^{(0)})
    \leq c_{95}|x|^{2/3}.
\]
With \eqref{smallt}, \eqref{xdist5}, and \eqref{xplower} this shows that
\begin{equation}
  g(x^{(0)}) + g(x - x^{(0)}) - g(x) \geq \frac{c_{94}}{2} \ep^2 |x|^{4/5},
\end{equation}
which with \eqref{ls1} yields
\begin{equation}\label{hgap2}
  h(x^{(0)}) + h(x - x^{(0)}) - h(x) \geq \frac{c_{94}}{4} \ep^2 |x|^{4/5}.
\end{equation}
As before, following \eqref{ggap}--\eqref{case1a} we obtain
\begin{align}\label{case3}
  P\Big( \Gamma_{0x}\cap (H_{\psi,t} \bs \Omega_{\psi,t,\ep}) \neq \emptyset \text{ and Case 3 holds}\Big) 
    \leq c_{96}e^{-c_{97}\ep^2 |x|^{3/10}}.
\end{align}

{\bf Case 4.} $\Gamma_{0x}$ overshoots: $\Phi_\theta(x^{(0)}) > \Phi_\theta(x)$. Define the point $\tred{v^{(0)}} = \Pi_{0x^{(0)}} \cap H_{\theta,\Phi_\theta(x)}$, so that $\Theta_\theta(v^{(0)}) > \ep$ and
\begin{equation}\label{xdist6}
 g(x^{(0)}) + g(x - x^{(0)}) - g(x) = [g(v^{(0)}) - g(x)] + g(x^{(0)} - v^{(0)}) + g(x - x^{(0)}).
\end{equation}
\begin{figure}
\includegraphics[height=5cm]{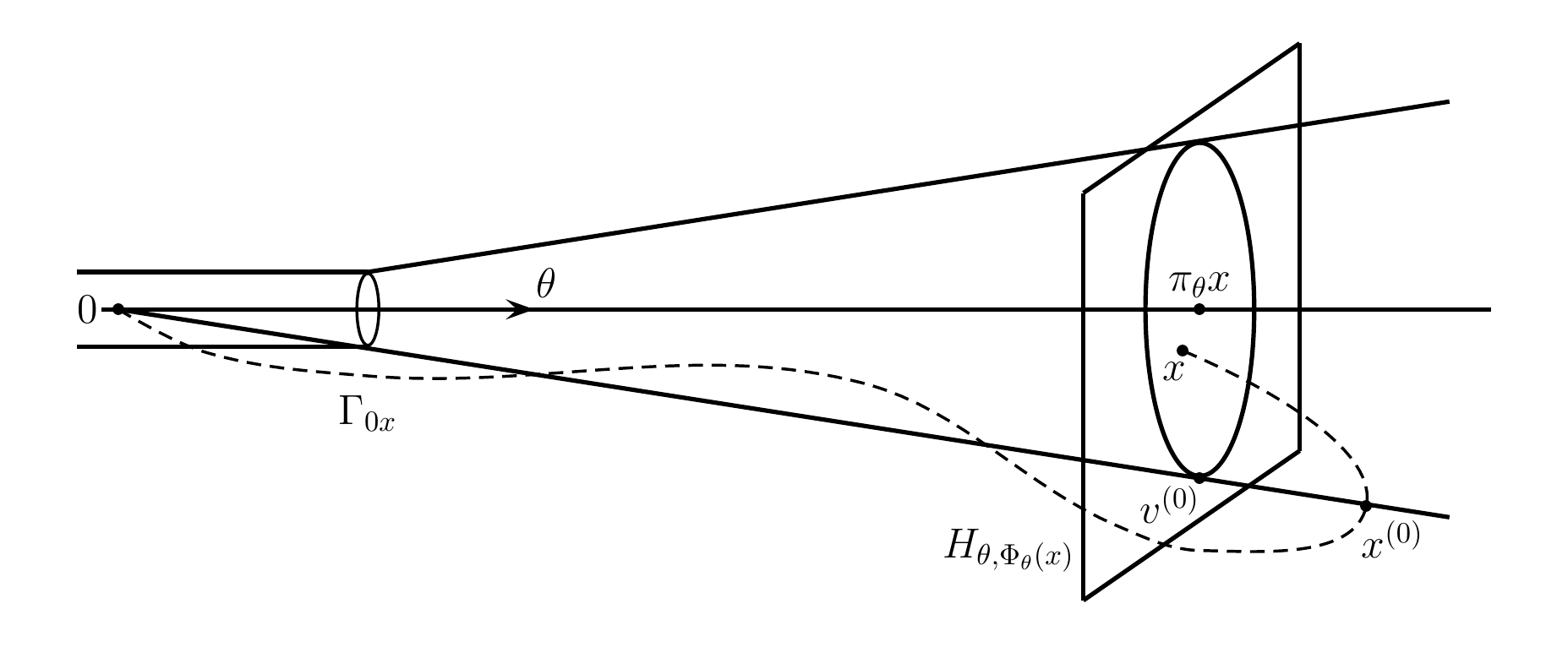}
\caption{ Diagram for Case 4. }
\label{Lem4.2case4_Fig}
\end{figure}
See Figure \ref{Lem4.2case4_Fig}. We have using \eqref{curvgwide} and \eqref{xcond}
\begin{align}\label{v0x}
  g(v^{(0)}) - g(x) &= [g(v^{(0)}) - g(\pi_\theta x)] - [g(x) - g(\pi_\theta x)] \notag\\
  &\geq c_{7}\ep^2 \Phi_\theta(x) - c_{5}\Theta_\theta(x)^2 \Phi_\theta(x) \notag\\
  &\geq 0
\end{align}
and $|x - x^{(0)}| > c_{99}\ep|x|$ so that for some $k\geq 1$, 
\[
  2^{k-1}c_{99}\ep|x| < |x-x^{(0)}| \leq 2^k c_{99}\ep|x|.
\]
With \eqref{xdist6} this shows that
\[
  g(x^{(0)}) + g(x - x^{(0)}) - g(x) \geq g(x - x^{(0)}) \geq 2^k c_{100}\ep|x|
\]
and hence as with \eqref{hgap} and \eqref{case2},
\begin{equation}\label{hgap3}
  h(x^{(0)}) + h(x - x^{(0)}) - h(x) \geq 2^{k-1} c_{100}\ep|x|
\end{equation}
and then after summing over $k$,
\begin{align}\label{case4}
  P\Big( \Gamma_{0x}\cap (H_{\psi,t} \bs \Omega_{\psi,t,\ep}) \neq \emptyset \text{ and Case 4 holds}\Big) 
    \leq c_{101}e^{-c_{102}\ep |x|^{1/2}}.
\end{align}
Since $\kappa>1/48$, combining all cases (specifically, \eqref{case1a}, \eqref{case1b}, \eqref{case2}, \eqref{case3}, and \eqref{case4}) completes the proof.
\end{proof}

Let $\tred{\Omega_{\psi,t,\ep}(\ZZ^d)} = \{x\in H_{\psi,t}^+ \cap \ZZ^d: x$ is an endpoint of an edge intersecting $\Omega_{\psi,t,\ep}\}$.

\begin{proposition}\label{slabnot}
\tcyn{Consider a standard FPP in $d$ dimensions and suppose edge passage times have a finite exponential moment and a direction of curvature $\theta$ exists. There exist constants $c_i$ as follows. For every $x$ with $(\theta,x)\in\mG$ and every slab $S\in\mN(0,x)$,
\begin{equation}\label{slabnot2}
  P(T(0,x\mid S) - T(0,x) \geq t) \leq c_{103}e^{-c_{104}t^{1/48}}\quad \text{for all } t>0.
\end{equation}}
\end{proposition}

\begin{proof}
Let \tred{$\ep$} be as in Lemma \ref{missball}; in the definition of near--natural slab take $\ep_1<\ep/2$.
Let $\tred{(\theta,x)}\in\mG$ and fix $\tred{S}\in\mN(0,x)$; this means $S=S_\psi(0,\Phi_\psi(x))$ for some \tred{$\psi$} with $|\frac{x}{|x|}-\psi|<\ep_1$.  Fix $t>0$ large, and $\tred{\delta}>0$ to be specified. Let $\tred{z}$ be the closest site in $H_{\psi,\delta t}^+$ to $\Pi_{0x}^\infty \cap H_{\psi,\delta t}$, and let $\tred{z'}$ be the closest site in $H_{\psi,\Phi_\psi(x)-\delta t}^-$ to $\Pi_{0x}^\infty \cap H_{\psi,\Phi_\psi(x)-\delta t}$. Let \tred{$W$} be the last vertex of $\Gamma_{0x}$ in $H_{\psi,t}^-$ and \tred{$W'$} the first vertex of $\Gamma_{Wx}$ in $H_{\psi,\Phi_\psi(x)-t}^+$, respectively, so $\Gamma_{WW'}\subset S$. See Figure \ref{Prop4.3_Fig}. Let \tred{$\gamma_z$} be a path in $S$ from 0 to $z$ of length $|z|_1$, and \tred{$\gamma_{z'}'$} a path in $S$ from $z'$ to $x$ of length $|x-z'|$.  Note that the primed quantities here are symmetric to the unprimed ones when we interchange 0 and $x$; symmetrically to $\Omega_{\psi,t,\ep}$ we define
\[
  \tred{\Omega_{\psi,t,\ep}'} = \{u \in H_{\psi,\Phi_\psi(x)-t}: |u-\pi_\psi u| \leq \ep t\}.
\]
Define the events
\[
  \tred{F_{\psi,\delta t}(x)}: \Gamma_{zW} \cup \Gamma_{z'W'} \not\subset S,
\]
\[
  \tred{G_{t,\ep}(x)}: \Gamma_{0x}\cap (H_{\psi,t} \bs \Omega_{\psi,t,\ep}) \neq \emptyset \ \ \text{or}\ \  
    \Gamma_{0x}\cap (H_{\psi,\Phi_\psi(x)-t} \bs \Omega_{\psi,t,\ep}') \neq \emptyset.
\]
For configurations in $F_{\psi,\delta t}(x)^c$ we have
\[
  T(0,x\mid S) \leq T(\gamma_z) + T(z,W) + T(W,W') + T(W',z') + T(\gamma_{z'}'),
\]
\[
  T(0,x) = T(0,W) + T(W,W') + T(W',x),
\]
so
\begin{align}\label{Fc}
  P&\Big( \{T(0,x\mid S) - T(0,x) \geq t\} \cap F_{\psi,\delta t}(x)^c \Big) \notag\\
  &\leq P\left( T(\gamma_z) \geq \frac t4 \right) + P\left( T(z,W) - T(0,W) \geq \frac t4 \right) \notag\\
  &\qquad + P\left( T(W',z') - T(W',x) \geq \frac t4 \right) + P\left( T(\gamma_{z'}') \geq \frac t4 \right) \notag\\
  &\leq 2P\left( T(\gamma_z) \geq \frac t4 \right) + 2P\left( T(\gamma_{z'}') \geq \frac t4 \right) \notag\\
  &= 4P\left( T(\gamma_z) \geq \frac t4 \right).
\end{align}
Provided we choose $\delta$ small we have 
\begin{equation}\label{delta}
  ET(\gamma_z)=|z|_1E\tau_e \leq \frac t8,
\end{equation}
so since our FPP is standard, from Lemma \ref{DHS} we get
\begin{equation}\label{fixpath}
  P\left( T(\gamma_z) \geq \frac t4 \right) \leq c_{105}e^{-c_{106}t}.
\end{equation}

\begin{figure}
\includegraphics[height=6cm]{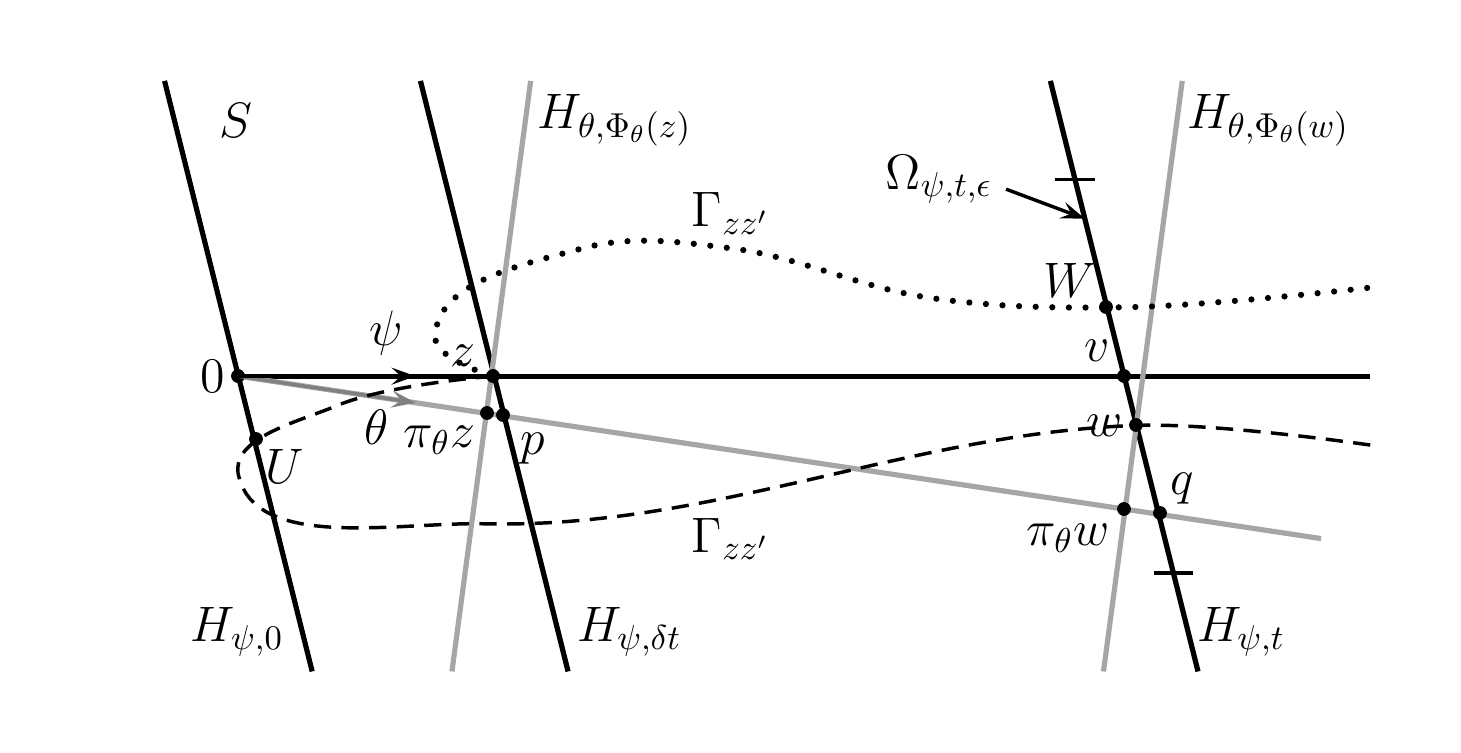}
\caption{ Diagram for Proposition 4.3, showing the end near $z$ of the geodesic $\Gamma_{zz'}$ in two cases; the picture near $z'$ is symmetric. Lines and hyperplanes corresponding to direction $\psi$ are black; those for direction $\theta$ are gray. The dotted geodesic shows the case $\omega\in F_{\psi,\delta t}(x)^c$.  The dashed geodesic shows the case $\omega\in F_{\psi,\delta t}(x) \cap G_{t,\ep}(x)^c$. $\Omega_{\psi,t,\ep}$ is between the hash marks in $H_{\psi,t}$. }
\label{Prop4.3_Fig}
\end{figure}

Next, from Lemma \ref{missball},
\begin{align}\label{Fbound}
  P\big( F_{\psi,\delta t}(x) \big) &\leq P\big( F_{\psi,\delta t}(x) \cap G_{t,\ep}(x)^c \big) + P\big( G_{t,\ep}(x) \big) \notag\\
  &\leq P\big( F_{\psi,\delta t}(x) \cap G_{t,\ep}(x)^c \big) + 2c_{68}e^{-c_{69}\ep^2t^{1/48}}.  
\end{align}
Now
\begin{align}\label{FGc}
  P\big( F_{\psi,\delta t}(x) \cap G_{t,\ep}(x)^c \big) &\leq 2P\big( \Gamma_{zw} \not\subset S \text{ for some } w\in 
    \Omega_{\psi,t,\ep}(\ZZ^d) \big) \notag\\
  &\leq c_{107}(\ep t)^{d-1} \max_{w\in \Omega_{\psi,t,\ep}(\ZZ^d) } P\big( \Gamma_{zw} \not\subset H_{\psi,0}^+ \big).
\end{align}
To bound the last probability, fix $\tred{w}\in \Omega_{\psi,t,\ep}(\ZZ^d)$ and suppose $\Gamma_{zw} \not\subset H_{\psi,0}^+$.
Let \tred{$U$} be the first point of $\Gamma_{zw}$ in $H_{\psi,0}^- \cup (2t\mkB_g)^c$. We need an upper bound for $g(w-z)$, but we cannot readily obtain this using $\psi$--coordinates (i.e.~tangential and longitudinal $\psi$--projections), as $\psi$ need not be a direction of curvature.  So instead we must in effect translate to $\theta$--coordinates.
Let
\[
  \tred{p} = \Pi_{0\theta}^\infty \cap H_{\psi,\delta t}, \quad \tred{q} = \Pi_{0\theta}^\infty \cap H_{\psi,t}, \quad \tred{v}=\Pi_{0\psi}\cap H_{\psi,t};
\]
see Figure \ref{Prop4.3_Fig}. Since
\[
  |\theta-\psi| \leq \left|\theta - \frac{x}{|x|}\right| + \left|\frac{x}{|x|}-\psi\right| \leq \frac{c_{108}}{|x|} + \ep_1 \leq 2\ep_1
\]
we have $|z-p|\leq c_{111}\ep_1\delta t$.  Since $\theta$ is a direction of curvature, the angle between $H_{\theta,0}$ and $H_{\psi,0}$ is at most a constant multiple of $|\theta-\psi|$.  Combining these we get
\[
  |p - \pi_\theta z| \leq c_{109}|z-p|\,|\theta-\psi| \leq c_{110}\ep_1^2\delta t \quad\text{and hence} \quad |z-\pi_\theta z|\leq 2c_{111}\ep_1\delta t.
\]
Further, in view of the arcsin bound,
\begin{equation}\label{wq}
  |w-q| \leq |w-v| + |v-q| \leq 2\ep t + c_{112}t|\theta-\psi| \leq c_{114}\ep_1 t
\end{equation}
and 
\begin{equation}\label{qpiw}
  |q - \pi_\theta w| \leq c_{113}|\theta-\psi|\,|w-q| \leq |w-q|,
\end{equation}
so
\begin{equation}\label{wpiw}
  |w-\pi_\theta w| \leq 2c_{114}\ep_1 t.
\end{equation}
The only condition we have imposed on $\delta$ is \eqref{delta}, so $\delta$ depends only on $E\tau_e$. Therefore we may choose $\ep_1$, in the definition of near--natural slab, to depend on $\delta$. Specifically, 
since $\theta$ is a direction of curvature, taking $\ep_1$ small enough we get using \eqref{curvg}, \eqref{wq}, \eqref{qpiw}, and \eqref{wpiw} that
\begin{align}
  g(w-z) &\leq g(w-p) + g(p-z) \notag\\
  &\leq g(\pi_\theta w - p) + c_{5}\frac{ |w-\pi_\theta w|^2}{\Phi_\theta(w-p)} + c_{115} \ep_1\delta t \notag\\
  &\leq g(\pi_\theta w - q) + g(q-v) + g(v-z) + g(z-p) + c_{116}\ep_1 t \notag\\
  &\leq g(v-z) + c_{117}\ep_1 t \notag\\
  &= (1-\delta)t + c_{117}\ep_1 t \notag\\
  &\leq \left( 1 - \frac \delta2 \right)t.
\end{align}
If $U\in H_{\psi,0}^- \cap 2t\mkB_g$ then $g(w-U)\geq t$; otherwise we have $U\notin 2t\mkB_g$ so $g(U-z) \geq g(U) - g(z)>t$.  Either way we have $g(U-z)+g(w-U)\geq t$ and $U\in \Gamma_{zw}\cap 2t\mkB_g$, and we conclude that
\[
  T(z,U) + T(U,w) - T(z,w) = 0, \quad g(U-z)+g(w-U) - g(w-z) \geq \frac \delta2 t.
\]
Then from \eqref{ls1},
\[
  h(U-z)+h(w-U) - h(w-z) \geq \frac \delta4 t.
\]
It follows that there exist $p,q\in 2t\mkB_g$ with $|T(p,q)-h(q-p)|\geq \delta t/4$, so from \eqref{FGc} and Lemma \ref{DHS},
\begin{align}
  P\big( F_{\psi,\delta t}(x) \cap G_{t,\ep}(x)^c \big) &\leq c_{118}\ep^{d-1}t^{3d-1} \max_{p,q\in 2t\mkB_g}
    P\left( |T(p,q)-h(q-p)|\geq \frac \delta4 t \right) \notag\\
  &\leq c_{119}e^{-c_{120}\delta^{1/2}t^{1/2}},
\end{align}
which combined with \eqref{Fc}, \eqref{fixpath}, and \eqref{Fbound} yields
\[
  P\Big( T(0,x\mid S) - T(0,x) \geq t \Big) \leq c_{121}e^{-c_{122}t^{1/48}},
\]
which completes the proof.
\end{proof}

\section{Proofs of supporting lemmas}\label{lempf}

\begin{proof}[Proof of Lemma \ref{upperapp}]
Let
\[
  \tred{\Upsilon(r)} = \log f(e^r),\quad r\geq 0,
\]
so $\Upsilon(r)=o(r)$ as $r\to\infty$. The first step is to replace $\Upsilon$ with an upper bound which is more regular (piecewise linear with bounded slope.) Let
\[
  \tred{M_k} = \sup\{\Upsilon(r):k-1\leq r\leq k\}, \ \ k\geq 2.
\]
Define \tred{$\Upsilon_{reg}$} on $[0,\infty)$ by
\[
  \Upsilon_{reg}(1) = M_1,\ \ \Upsilon_{reg}(k) = \max(M_k,M_{k+1}), \quad \Upsilon_{reg}(k-\tfrac12) = M_k, \ \ k\geq 2,
\]
with $\Upsilon_{reg}$ linear on each interval $[k-1,k-\tfrac12]$ and $[k-\tfrac12,k]$, so in view of \eqref{growctrl} we have 
\begin{equation}\label{Upvsreg}
  \Upsilon_{reg} - c_{123} \leq \Upsilon\leq\Upsilon_{reg},
\end{equation}
and $\Upsilon_{reg}$ is bounded below. Also from \eqref{growctrl} (taking $\alpha\leq e^2$) we have
\[
  s\in[r,r+2] \implies |\Upsilon(s) - \Upsilon(r)| \leq c_{124} := 2\kappa + \log c,
\]
so $|M_{k+1}-M_k| \leq c_{124}$ for all $k$, and hence 
\begin{equation}\label{slope}
  |\Upsilon_{reg}(s) - \Upsilon_{reg}(r)| \leq 2c_{124}|s-r|\ \ \text{for all } r,s\geq 0.
\end{equation}

We define $\Upsilon_{up}\geq \Upsilon_{reg}$ by two cases, then in each case let
\[
   f_{up}(r) = e^{\Upsilon_{up}(\log r)}, \quad r\geq 1.
\]

{\bf Case 1.} $\Upsilon_{reg}(r)$ is eventually nonpositive as $r\to\infty$, say $\Upsilon_{reg}(r)\leq 0$ for all $r\geq r_0$.  Here we define
\[
  \tred{\Upsilon_{up}(r)} = \begin{cases} 2c_{124}(r_0-r) &\text{if } 0\leq r\leq r_0,\\ 0 &\text{if } r>r_0. \end{cases}
\]
By \eqref{slope} this satisfies $\Upsilon_{up}\geq\Upsilon_{reg}\geq\Upsilon$, so $f_{up}\geq f$, and since $\Upsilon$ is bounded below, the other conditions in \eqref{upperf}, and the regular growth exponent property, are straightforward.

{\bf Case 2.} There exist arbitrarily large $r$ with $\Upsilon_{reg}(r)>0$. By \eqref{slope} $\Upsilon_{reg}$ is bounded on bounded intervals, and $\Upsilon_{reg}(r)=o(r)$, so we can define \tred{$\Upsilon_{up}$} to be the concave majorant of $\Upsilon_{reg}\vee 0$ on $[0,\infty)$, which is nondecreasing. If $\Upsilon_{reg}$ is bounded on all $[0,\infty)$ then $\Upsilon_{up}$ is also bounded, and since $f\geq\delta$ it follows that $f_{up}/f$ is bounded. If instead $\limsup_{r\to\infty} \Upsilon_{reg}(r)=\infty$ then there exist arbitrarily large $r$ where $\Upsilon_{up}(r)=\Upsilon_{reg}(r)$ so it follows from \eqref{Upvsreg} that $\liminf_{r\to\infty} f_{up}(r)/f(r)<\infty$.
Concavity and nonnegativity of $\Upsilon_{up}$ ensure that $\Upsilon_{up}(r)/r$ is nonincreasing, or equivalently, $\log f_{up}(r)/\log r$ is nonincreasing. Thus again all conditions in \eqref{upperf} are satisfied.  The regular growth exponent property for $f_{up}$ follows from the fact that $\Upsilon_{up}'$ (which exists a.e.) decreases to 0.
\end{proof}

\begin{proof}[Proof of Lemma \ref{regzero}]
Let
\[
  \tred{s} = s(r) = \frac{\eta(r)r}{\log \frac{1}{\Xi(\eta(r)r)} }.
\]
Taking logs, we see that we wish to show that given $c>0$, for large $r$,
\[
  \log\log \frac{1}{\Xi(\eta(r)r)} + \delta(s)\log s \geq \delta(r)\log r + c,
\]
or equivalently,
\[
  \frac{ \log\log \frac{1}{\Xi(\eta(r)r)} - c }{\delta(r)\log s} + \frac{\delta(s)}{\delta(r)} \geq 1 + \frac{\log\frac rs}{\log s}.
\]
Since $\delta$ is nonincreasing and $s(r)<r$ for large $r$, we have $\delta(s)/\delta(r)\geq 1$ so it is sufficient to show that for large $r$,
\begin{equation}\label{ratio}
  \frac{ \log\log \frac{1}{\Xi(\eta(r)r)} - c }{\log\frac rs} \geq \delta(r).
\end{equation}
We have by \eqref{PsiXi}
\[
  \log\frac rs = \log\log \frac{1}{\Xi(\eta(r)r)} + \log \frac{1}{\eta(r)} \leq \log\log \frac{1}{\Xi(r)} + \log \frac{1}{\eta(r)} \leq \frac{C}{\delta(r)},
\]
and \eqref{ratio} follows, since $\eta(r)r\to\infty$ as $r\to\infty$.
\end{proof}

\begin{proof}[Proof of Lemma \ref{backtrk}]
In this proof constants $c_i$ may depend on $K$.
Let \tred{$x,\theta,r$} be as in the lemma statement, and let $\tred{\alpha}=x/|x|$, so (by definition of $\mN(\cdot,\cdot)$) $|\theta-\alpha|<\ep_1$ and  $\mH_\theta$ makes an angle of less than $\ep_1$ with $\mH_\alpha$.  
If $\Gamma_{0x} \cap H_{\theta,-r}^- \neq\emptyset$, let \tred{$Z$} be the first vertex of $\Gamma_{0x} \cap H_{\theta,-r}^-$, so $\Phi_\theta(Z)\leq -r$. We want a lower bound for the extra distance $g(Z) + g(x-Z) - g(x)$. 

{\bf Claim.}
\begin{equation}\label{extraback}
  g(Z) + g(x-Z) - g(x) \geq \frac12(r+g(Z)).
\end{equation}
If $g(Z)\geq 2g(x)$ then the left side is bounded below by $g(Z)$, which is at least $r$ since $Z\in H_{\theta,-r}$, so \eqref{extraback} holds.  Thus we may assume $g(Z)<2g(x)$.
Then to prove \eqref{extraback}, let \tred{$V,W$} be the points where $\Pi_{xZ}^\infty$ intesects $H_{\alpha,0}$ and $H_{\theta,0}$, respectively.  See Figure \ref{Lemma2.2_Fig}. Since $V\in H_{\alpha,0}$ we have $g(x-V)\geq g(x)$, and therefore
\begin{equation}\label{gbound1}
  g(x-Z) \geq g(x-W) + r \geq g(x-V) - g(V-W) + r\geq g(x) - g(V-W) + r.
\end{equation}
The arcsin bound provides a minimum angle between $x$ (i.e.~$\alpha$) and $H_{\alpha,0}$, and since $g(Z)<2g(x)$, also then a minimum possible angle, call it \tred{$\varphi_0$}, between $\Pi_{xZ}^\infty$ and $H_{\alpha,0}$.  It then follows from basic geometry that, since the angle between $H_{\alpha,0}$ and $H_{\theta,0}$ is less than $\ep_1$, from $\varphi_0$ we get $c_{125}$ such that
\[
  g(V-W) \leq c_{125}\ep_1g(W) \leq c_{125}\ep_1(g(Z) + g(W-Z)) = c_{125}\ep_1\left(g(Z) + \frac{r}{\Phi_\theta(x)+r} g(x-Z)\right).
\]
Combining this with \eqref{gbound1} and using $g(x-Z)\leq g(x)+g(Z)<3g(x)$ we get
\begin{align}\label{gbound2}
 g(Z) + g(x-Z) - g(x) &\geq g(Z) - g(V-W) + r \notag\\
 &\geq (1 - c_{125}\ep_1)g(Z) + \left( 1 - 3c_{125}\ep_1\frac{g(x)}{\Phi_\theta(x)} \right)r.
\end{align}
Now provided $\ep_1$ is taken small enough (depending only on $g$),
\[
  |x-\pi_\theta x| \leq c_{126}|\alpha-\theta|\,|x| \leq c_{126}\ep_1|x| \quad\text{so}\quad \Phi_\theta(x) = g(\pi_\theta x)
    \geq g(x) - g(x-\pi_\theta x) \geq \frac12 g(x), 
\]
which with \eqref{gbound2} proves the claim \eqref{extraback}.

\begin{figure}
\includegraphics[height=6cm]{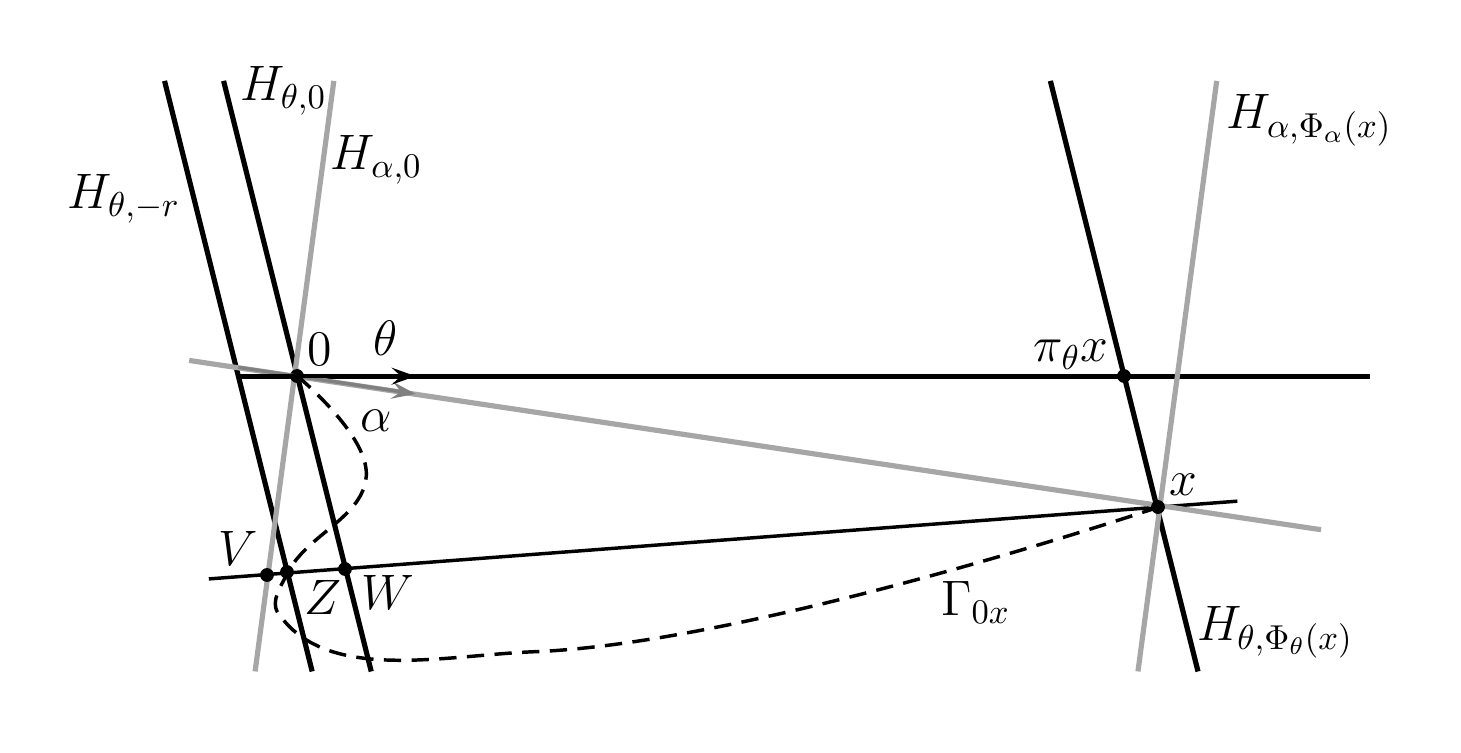}
\caption{ Diagram for the proof of Lemma \ref{backtrk}. Lines and hyperplanes corresponding to direction $\theta$ are black; those for direction $\alpha$ are gray. In the case shown we have $Z\in H_{\alpha,0}^+$, meaning $\Gamma_{0x}$ backtracks to $Z$ in direction $\theta$ but not in direction $\alpha$. In general $Z$ may lie on either side of $V$. }
\label{Lemma2.2_Fig}
\end{figure}

From \eqref{fmag} and \eqref{ls1} we have $h(x)-g(x)\leq c_{1}D_{mag}(|x|)\leq c_{1}c_{4}(|x|\log |x|)^{1/2}$ so
recalling $r\geq c_{18}(|x|\log |x|)^{1/2}$, we have from the claim that provided $c_{18}$ is large,
\[
  h(Z) + h(x-Z) - h(x) \geq \frac12(r+g(Z)) - c_{1}c_{4}(|x|\log |x|)^{1/2}> \frac12 g(Z)
\]
while
\[
  T(0,Z) + T(Z,x) -T(0,x) = 0,
\]
so one of $h(Z)-T(0,Z),h(x-Z)-T(Z,x),T(0,x)-h(x)$ must exceed $g(Z)/6$. Since $Z\in H_{\theta,-r}^-$ we have $g(Z)\geq r$.  
Either $g(Z)\leq g(x)$ or there exists $k\geq 1$ for which $2^{k-1}g(x)<g(Z)\leq 2^kg(x)$. 
Decomposing according to the value of $k$ then gives
\begin{align}
 P\Big( \Gamma_{0x} \cap H_{\theta,-r}^- \neq\emptyset \Big) 
   &\leq P\left( T(0,x)-h(x) \geq \frac{r}{6} \right) \notag\\
 &\qquad + P\bigg( \max\big( h(z)-T(0,z), h(x-z)-T(z,x) \big) \geq \frac{r}{6} \notag\\
 &\hskip 1.2in \text{for some $z$ with } |z| \leq c_{127}|x| \bigg) \notag\\
 &\qquad + \sum_{k=1}^\infty P\bigg( \max\big( h(z)-T(0,z), h(x-z)-T(z,x) \big) \geq \frac{2^{k-1}g(x)}{6} \notag\\
 &\hskip 1.2in \text{for some $z$ with } |z| \leq c_{127}2^k|x| \bigg).
\end{align}
Using Lemma \ref{DHS} and recalling $r\leq K|x|$, this yields
\begin{align}
 P\Big( &\Gamma_{0x} \cap H_{\theta,-r}^- \neq\emptyset \Big) \notag\\
 &\leq c_{128}\exp\left( -c_{129} \frac{r^2}{|x|} \right) + c_{128}|x|^d \exp\left( -c_{129}|x| \right)
   + \sum_{k=1}^\infty c_{128}(2^k|x|)^d\exp\left( -c_{129}2^{2k}|x| \right) \notag\\
 &\leq c_{130}e^{-c_{131}r^2/|x|} + c_{132}e^{-c_{133}|x|} \notag\\
 &\leq c_{134}e^{-c_{135}r^2/|x|}.
\end{align}
The same bound for $P(\Gamma_{0x} \cap H_{\theta,\Phi_\theta(x)+r}^+ \neq\emptyset)$ is obtained symmetrically.
\end{proof}


\begin{thebibliography}{99}

\bibitem{Al97} Alexander, K. S. (1997). Approximation of subadditive functions and rates of convergence in limiting shape results.  \emph{Ann. Probab.} \textbf{24} 30--55.

\bibitem{Al20} Alexander, K. S. (2020). Geodesics, bigeodesics, and coalescence in first passage percolation in general dimension.  arXiv:2001.08736 [math.PR]  

\bibitem{Al20a} Alexander, K. S. (2020). Uniform fluctuation and wandering bounds in first passage percolation. arXiv:2011.07223[math.PR]  

\bibitem{AOF14} Alves, S. G., Oliveira, T. J., and Ferreira, S. C. (2018). Universality of fluctuations in the Kardar-Parisi-Zhang class in high dimensions and its upper critical dimension. \emph{Phys. Rev. E} \textbf{90} 020103.  arXiv:1405.0974 [cond-mat.stat-mech]

\bibitem{ADH17} Auffinger, A., Damron, M., and Hanson, J. (2017). \emph{50 years of first passage percolation. University Lecture Series} {\bf 58}. American Mathematical Society, Providence, RI.  arXiv:1511:03262 [math.pr]

\bibitem{BG21} Basu, R. and Ganguly, S. (2021). Time correlation exponents in last passage percolation. In: \emph{ In and out of equilibrium 3. Celebrating Vladas Sidoravicius, Progr. Probab.} {\bf 77}, 101--123. Birkh\"auser--Springer, Cham.
arXiv:1807.09260 [math.PR]

\bibitem{BSS19} Basu, R., Sarkar, S., and Sly, A. (2019).  Coalescence of geodesics in exactly solvable models of last passage percolation.  \emph{J. Math. Phys.} \textbf{60} 093301, 22 pp.  arXiv:1704.05219 [math.PR]

\bibitem{BSS16} Basu, R., Sidoravicius, V., and Sly, A. (2016).  Last passage percolation with a defect line and the solution of the slow bond problem.  arXiv:1408.3464 [math.PR]

\bibitem{BM14} Benjamini, I. and Maillard, P. (2018).  Point--to--point distance in first passage percolation on (tree)$\times\mathbb{Z}$. In: \emph{Geometric Aspects of Functional Analysis, Israel Seminar (GAFA) 2011--2013}, B.~Klartag, V.~Milman (Eds.), \emph{Lecture Notes in Math.} {\bf 2116}, 47--51, Springer, Heidelberg.  arXiv:1310:4018 [math.pr]

\bibitem{BZ12} Benjamini, I. and Zeitouni, O. (2012).  Tightness of fluctuations of first passage percolation on some large graphs. In: \emph{Geometric Aspects of Functional Analysis, Israel Seminar (GAFA) 2006--2010}, B.~Klartag, S.~Mendelson, V.~Milman (Eds.), \emph{Lecture Notes in Math.} {\bf  2050} 127--132, Springer, Heidelberg.  arXiv:1010:1412 [math.pr] 

\bibitem{Ch13} Chatterjee, S. (2013).  The universal relation between scaling exponents in first-passage percolation.\emph{Ann. of Math. (2)} {\bf 127}, no.~2, 663--697.  arXiv:1105.4566 [math.PR]  

\bibitem{CD81} Cox, J. T. and Durrett, R. (1981).  Some limit theorems for percolation with necessary and sufficient conditions. \emph{Ann. Probab.} \textbf{9}, 809--819.

\bibitem{DHHX20} Damron, M., Hanson, J., Houdr\'e, C., and Xu, C. (2020). Lower bounds for fluctuations in first--passage percolation for general distributions. \emph{Ann. Inst. H.~Poincar\'e Probab. Statist.} {\bf 56}(2), 1336--1357.  arXiv:1810:04270 [math.pr]

\bibitem{DHS14} Damron, M., Hanson, J., and Sosoe, P. (2014). Subdiffusive concentration in first--passage percolation. \emph{Electron. J. Probab.} \textbf{19}, no.~109, 27 pp. arXiv:1401.9017 [math.pr]

\bibitem{DK16} Damron, M. and Kubota, N. (2016). Rate of convergence in first--passage percolation under low moments. \emph{Stoch. Proc. Applic.} \textbf{126}, 3065--3076. arXiv:1406.3105 [math.pr]

\bibitem{DH91} Dekking, F. M. and Host, B. (1991). Limit distributions for minimal displacement of branching random walks. \emph{Probab. Theory Rel. Fields} {\bf 90}, 403--426.  

\bibitem{DEP22} Dembin, B., Elboim, D., and Peled, R. (2022). Coalescence of geodesics and the BKS midpoint problem in planar first--passage percolation. arXiv:2204.02332 [math.pr]

\bibitem{DL81} Durrett, R., and Liggett, T. (1981).  The shape of the limit set in Richardson's growth model. \emph{Ann. Probab.} {\bf 9}, 186--193.

\bibitem{Fo08}  Fogedby, H. C. (2006). Kardar-Parisi-Zhang equation in the weak noise limit: Pattern formation and upper critical dimension. \emph{Phys. Rev. E} \textbf{73} 031104.  arXiv:cond-mat/0510268 [cond-mat.stat-mech] 

\bibitem{Ga20} Gangopadhyay, U. (2020).  Fluctuations of transverse increments in two-dimensional first passage percolation. \emph{Electron. J. Probab.} {\bf 27}, 1--61.  arXiv:2011:14686 [math.pr]

\bibitem{GH20} Ganguly, S. and Hegde, M. (2020). Optimal tail exponents in general last passage percolation via bootstrapping \& geodesic geometry.  arXiv:2007.03594 [math.pr]

\bibitem{Jo00} Johansson, K. (2000). Shape fluctuations and random matrices. \emph{Commun. Math. Phys.} \textbf{209}, 437--476.  arXiv:math/9903134 [math.co]

\bibitem{Ke93} Kesten, H. (1993).  On the speed of convergence in first-passage percolation.  \emph{Ann. Appl. Probab.} \textbf{3} 296--338.

\bibitem{KK14} Kim, S.-W. and Kim, J. M. (2014). A restricted solid-on-solid model in higher dimensions. \emph{J. Stat. Mech.} \textbf{2014} P07005.  

\bibitem{KCDW14} Kloss, T., Canet, L., Delamotte, B. and Wschebor, N. (2014). Kardar--Parisi--Zhang equation with spatially correlated noise: A unified picture from nonperturbative renormalization group. \emph{Phys. Rev. E} \textbf{89} 022108. arXiv:1312.6028 [cond-mat.stat-mech]

\bibitem{LW05} Le Doussal, P. and Wiese, K. J. (2005). Two-loop functional renormalization for elastic manifolds pinned by disorder in $N$ dimensions. \emph{Phys. Rev. E} \textbf{72} 035101.  arXiv:cond-mat/0501315 [cond-mat.dis-nn]

\bibitem{LMR02} Lo\"we, M., Merkl, F., and Rolles, S. (2002).  Moderate deviations for longest increasing subsequences: The lower tail. \emph{J. Theor. Probab.} {\bf 15}, 1031--1047.  

\bibitem{MPPR02} Marinari, E., Pagnani, A., Parisi, G., R\'acz, Z. (2002). Width distributions and the upper critical dimension of Kardar-Parisi-Zhang interfaces. \emph{Phys. Rev. E} \textbf{65} 026136.  arXiv:cond-mat/0105158 [cond-mat.stat-mech]

\bibitem{Ne95} Newman, C. M., A surface view of first passage percolation.  \emph{Proceedings of the International Congress of Mathematicians}, Vol.~1, 2 (Z\"urich, 1994), 1047--1023, Birkh\"auser, Basel (1995).  

\bibitem{ROM15} Rodrigues, E. A., Oliveira, F. A., and Mello, B. A. (2015). On the existence of an upper critical dimension for systems within the KPZ universality class.  \emph{Acta. Phys. Pol. B} \textbf{46}, 1231--1234.  arXiv:cond-mat/1502.06121 [cond-mat.stat-mech]

\bibitem{Ta95} Talagrand, M. (1995). Concentration of measure and isoperimetric inequalities in product spaces. \emph{Publ. Math. I. H. E. S.} \textbf{81}, 73--205.

\bibitem{WR78} Wierman, J. C. and Reh, W. (1978). On conjectures in first passage percolation theory. \emph{Ann. Probab.} {\bf 6}, 388--397. 

\end{thebibliography}
\end{document}